\documentclass[11pt]{article}\UseRawInputEncoding
\usepackage{epic,latexsym,amssymb}
\usepackage{color}
\usepackage{tikz}
\usepackage{amsfonts,epsf,amsmath,comment}
\usepackage{lineno}
\usepackage{latexsym,url}
\usepackage[usenames,dvipsnames]{pstricks}
\usepackage{pstricks-add}
\usepackage{epsfig}
\usepackage{pst-grad}
\usepackage{pst-plot}
\usepackage[space]{grffile}
\usepackage{etoolbox}

\usepackage{subfig}
\usepackage{epic}
\usepackage{amssymb}
\usepackage{latexsym}
\usepackage{tikz}
\usetikzlibrary{decorations.markings}

\usetikzlibrary{arrows}

\newtheorem{theorem}{Theorem}

\newtheorem{question}{Question}

\newtheorem{claim}{Claim}
\newtheorem{subclaim}{Claim}[claim]

\newtheorem{proposition}{Proposition}

\newtheorem{remark}{Remark}

\newcommand{\st}{{\rm st}}
\newcommand{\CG}{{\rm CG}}
\newcommand{\tc}{{\rm tc}}
\newcommand{\odd}{{\rm odd}}
\newcommand{\DB}{{\rm DB}}
\newcommand{\TB}{{\rm TB}}
\newcommand{\DC}{{\rm DC}}
\newcommand{\even}{{\rm even}}
\newcommand{\smalls}{{\rm small}}
\newcommand{\sorder}{{\rm small}}
\newcommand{\lorder}{{\rm large}}
\newcommand{\mdom}{\dot\gamma}
\newcommand{\midom}{\overline{i}}
\newcommand{\mV}{\overline{V}}
\newcommand{\labshow}[1]{\label{#1} {\footnotesize{\texttt{(#1)}}}}
\newcommand{\markn}{\overline{n}}
\newcommand{\w}{{\rm w}}
\newcommand{\cp}{\,\square\,}
\newcommand{\vc}{{\rm vc}}

\newcommand{\proof}{\noindent\textbf{Proof. }}

\newcommand{\modo}{{\rm mod}}
\newcommand{\pn}{{\rm pn}}
\newcommand{\epn}{{\rm epn}}
\newcommand{\ipn}{{\rm ipn}}
\newcommand{\expansion}{\mathop{\mathrm{exp}}}
\newcommand{\cub}{{\rm cubic}}

\newcommand{\barS}{{\overline{S}}}

\newcommand{\qed}{$\Box$}
\newcommand{\QED}{$\Box$}

\newcommand{\QEDmark}{\mbox{\textsc{qed}}}
\newcommand{\smallqed}{{\tiny ($\Box$)}}

\newcommand{\1}{\vspace{0.1cm}}
\newcommand{\2}{\vspace{0.2cm}}
\newcommand{\3}{\vspace{0.3cm}}
\newcommand{\5}{\vspace{0.05cm}}

\def \nH {n_{_H}}
\def \nF {n_{_F}}
\newcommand{\tdom}{{\rm tdom}}

\newcommand{\barX}{{\overline{X}}}
\newcommand{\barx}{{\overline{x}}}
\newcommand{\barG}{{\overline{G}}}
\newcommand{\barY}{{\overline{Y}}}

\newcommand{\cB}{{\cal B}}
\newcommand{\cC}{{\cal C}}
\newcommand{\cD}{{\cal D}}
\newcommand{\cF}{{\cal F}}
\newcommand{\cL}{{\cal L}}
\newcommand{\cS}{{\cal S}}
\newcommand{\cG}{{\cal G}}
\newcommand{\diam}{{\rm diam}}
\newcommand{\Dbar}{\overline{D}}
\newcommand{\Ibar}{\overline{I}}

\newcommand{\cubic}{{\rm cubic}}
\newcommand{\coro}{{\rm cor}}
\newcommand{\CL}{{\rm CL}}
\newcommand{\ML}{{\rm ML}}
\newcommand{\cN}{{\mathcal{N}}}
\newcommand{\cT}{{\mathcal{T}}}

\newcommand{\gt}{\gamma_t}
\newcommand{\pr}{\gamma_{\rm pr}}
\newcommand{\semiT}{\gamma_{\rm t2}}
\newcommand{\semiP}{\gamma_{\rm pr2}}
\newcommand{\nsemiP}{\gamma_{\rm npr2}}

\newcommand{\Mike}[1]{\textcolor{blue}{#1} }

\newcommand{\cH}{{\cal H}}
\def \kkk {{\Huge\sf --- !!! --- }}
\def \komm {\Huge\sf }

\renewcommand{\labelenumi}{\rm(\alph{enumi})}
\let\oldenumerate\enumerate
\renewcommand{\enumerate}{
  \oldenumerate
  \setlength{\itemsep}{1pt}
  \setlength{\parskip}{0pt}
  \setlength{\parsep}{0pt}
}

\def\vertex(#1){\put(#1){\circle*{2}}}
\def\vertexo(#1){\put(#1){\circle{2}}}
\def\vert(#1){\put(#1){\circle*{1.5}}}
\def\verto(#1){\put(#1){\circle{1.5}}}
\def\lab(#1)#2{\put(#1){\makebox(0,0)[c]{#2}}}
\newcommand{\gst}{\gamma_{\rm st}}

\begin{document}

\title{Strong domination number of graphs from primary subgraphs }

\author{
	Saeid Alikhani$^{1}$	\and
Nima Ghanbari$^{2,}$\footnote{Corresponding author}
\and
Michael A. Henning$^{3}$
}

\date{\today}

\maketitle

\begin{center}

$^1$Department of Mathematical Sciences, Yazd University, 89195-741, Yazd, Iran\\
$^2$Department of Informatics, University of Bergen, P.O. Box 7803, 5020 Bergen, Norway\\	
$^{3}$Department of Mathematics and Applied Mathematics, University of Johannesburg, Auckland Park, 2006 South Africa.\\

\medskip
\medskip
{\tt  $^1$alikhani@yazd.ac.ir
~~
$^{2}$Nima.Ghanbari@uib.no
~~
$^3$mahenning@uj.ac.za
 }

\end{center}

%%%%%%%%%%%%%%ABSTRACT%%%%%%%%%%%%%%%%%%%%%%%%%%%%%%%%%%%%%%%%%%%%%%%%%%%%%%%%%%%%

\begin{abstract}

A set $D$ of vertices is a strong dominating set in a graph $G$, if for every vertex $x\in  V(G) \setminus D$ there is a vertex $y\in D$ with $xy\in E(G)$ and $\deg(x) \leq \deg(y)$. The strong domination number $\gst(G)$ of $G$ is the minimum cardinality of a strong dominating set in $G$. Let $G$ be a connected graph constructed from pairwise disjoint connected graphs $G_1,\ldots ,G_k$ by selecting a vertex of $G_1$, a vertex of $G_2$, and identifying these two vertices, and thereafter continuing in this manner inductively. The graphs $G_1,\ldots ,G_k$ are the primary subgraphs of $G$. In this paper, we study the strong domination number of $K_r$-gluing of two graphs and investigate the strong domination number for some particular cases of graphs from their primary subgraphs.
\end{abstract}


%%%%%%%%%%%%%%%%%%%%%%%%%%%%%%%%%%%%%%%%%%%%%%%%%%%%%%%%%%%%%%%%%%%%%%%%%%%%%%%%%
%%%%%%%%%%%%%%%%%%%%%%%%%%%%%%%%%%%%%%%%%%%%%%%%%%%%%%%%%%%%%%%%%%%%%%%%%%%%%%%%%

\newpage
\section{Introduction}

Let $G$ be a graph with vertex set $V(G)$ and edge set $E(G)$, and of order $n(G) = |V(G)|$ and size $m(G) = |E(G)|$. If the graph $G$ is clear from context, we write $V$ and $E$ rather than $V(G)$ and $E(G)$, and write $G = (V,E)$. Two adjacent vertices in $G$ are called \emph{neighbors}. We denote the degree of a vertex $v$ in $G$ by $\deg_G(v)$. For $k \ge 1$ an integer, we use the standard notation $[k] = \{1,\ldots,k\}$.

A set $S$ of vertices in $G$ is a \emph{dominating set} if every vertex not in $S$ is adjacent to a vertex in~$S$. The \emph{domination number} $\gamma(G)$ of $G$ is the minimum cardinality of a dominating set in $G$. For recent books on domination in graphs, we refer the reader to~\cite{HaHeHe-20,HaHeHe-21,HaHeHe-23,HeYe-book}.

The set $S$ is a \emph{strong dominating set} of $G$ if for every vertex $x \in \overline{D} = V \setminus D$, there exists a vertex $y \in D$ that is adjacent to~$x$ and such that $\deg_G(x) \le \deg_G(y)$. The \emph{strong domination number} $\gst(G)$ of $G$ is the minimum cardinality of a strong dominating set in $G$. A $\gst$-\emph{set} of $G$ is a strong dominating set of $G$ of minimum cardinality $\gst(G)$. If $D$ is a strong dominating set in~$G$, then we say that a vertex $u \in \overline{D}$ is \emph{strongly dominated} by a vertex $v \in D$ if $uv \in E(G)$ and $\deg(u) \le \deg(v)$.

The strong domination number was introduced in 1996 by Sampathkumar and Latha~\cite{SaLa-96}. Earlier results on strong domination in graphs were presented by, among others, Hattingh and Henning~\cite{HaHe-98} in 1998, Rautenbach~\cite{Ra-99,Ra-00} in 1999, and Rautenbach and Zverovich~\cite{RaZv-01} in 2001. Recent results are given, for example, by G\"{o}lpek and  Ayta\c{c}~\cite{GoAy-20} in 2020, Do\u{g}an Durgun and Kurt~\cite{DoKu-22} and Alikhani, Ghanbari, and Zaherifar~\cite{AlGhZa-22+} in 2022.

In 1996 Sampathkumar and Latha~\cite{SaLa-96} also defined a set $D\subset V$ to be a \emph{weak dominating set} of $G$, if every vertex $v\in V\setminus S$ is adjacent to a vertex $u\in D$ such that $\deg_G(v) \ge \deg_G(u)$. Earlier results on weak domination in graphs were presented by, among others, Hattingh and Laskar~\cite{HaLa-98} and Rautenbach~\cite{Ra-98} in 1998.

The concept of strong and weak domination has subsequently been extensively studied in the literature. Several papers (see, for example, Boutrig and Chellali~\cite{Boutrig}) present relations between the weak and strong domination numbers of a graph, as well as bounds on these parameters. Alikhani, Ghanbari, and Zaherifar~\cite{AlGhZa-22+} examined the effects on $\gst(G)$ when $G$ is modified by the edge deletion, the edge subdivision and the edge contraction, and also studied the strong domination number of $k$-subdivision of $G$.

Motivated by enumerating the number of dominating sets of a graph and by domination polynomial (see, for example, \cite{euro,saeid1}), the enumeration of  the strong dominating sets for certain graphs was studied by Zaherifar, Alikhani, and Ghanbari~\cite{ZaAlGh-23} in 2023. The strong domination number of graph operations are natural subjects for study. For example, the strong domination number of join and corona products is studied in~\cite{ZaAlGh-23}.

Let $G_1$ and $G_2$ be two graphs and $r\in \mathbb{N} \cup \{0\}$ with $r \leq \min \{\omega(G_1),\omega(G_2)\}$, where $\omega(G)$ is the clique number of $G$. A clique in a graph $G$ is a subset of vertices of $G$ such that every two distinct vertices in the clique are adjacent. For $r \ge 2$, an $r$-clique is a clique of cardinality~$r$. Choose an $r$-clique from each graph $G_i$, $i=1,2$, and form a new graph $G$ from the union of $G_1$ and $G_2$ by identifying the two chosen $r$-cliques. The graph $G$ is called an \emph{$r$-gluing} of $G_1$ and $G_2$ and denoted by $G_1\cup_{K_r} G_2$. If $r=0$, then $G_1\cup_{K_0}G_2$ is simply the disjoint union of $G_1$ and $G_2$. The graph operation $G_1 \cup_{K_1} G_2$ is called \emph{vertex gluing}, and the graph operation $G_1 \cup_{K_2} G_2$ is called \emph{edge gluing}. Notice that there are sometimes several ways to $r$-glue two graphs together (see Figure~\ref{GcupHmorethanone}).  Zykov~\cite{Zykov} provided a shortcut for evaluating the chromatic polynomial of a graph $G$, if $G$ is a $K_r$-gluing of some graphs. Ghanbari and Alikhani~\cite{Nima2} obtained some results on the total dominator coloring of the $r$-gluing of two graphs.

\begin{figure}
	\begin{center}
		\psscalebox{0.6 0.6}
		{
			\begin{pspicture}(0,-5.5014424)(18.394232,1.0956731)
			\psdots[linecolor=black, dotsize=0.4](0.19711533,-0.70144224)
			\psdots[linecolor=black, dotsize=0.4](1.7971153,-0.70144224)
			\psdots[linecolor=black, dotsize=0.4](1.7971153,-2.3014421)
			\psdots[linecolor=black, dotsize=0.4](0.19711533,-2.3014421)
			\psline[linecolor=black, linewidth=0.08](0.19711533,-0.70144224)(1.7971153,-0.70144224)(1.7971153,-2.3014421)(0.19711533,-2.3014421)(0.19711533,-0.70144224)(0.19711533,-0.70144224)
			\psline[linecolor=black, linewidth=0.08](1.7971153,-0.70144224)(0.19711533,-2.3014421)(0.19711533,-2.3014421)
			\psline[linecolor=black, linewidth=0.08](0.19711533,-0.70144224)(1.7971153,-2.3014421)(1.7971153,-2.3014421)
			\psdots[linecolor=black, dotsize=0.4](1.7971153,0.8985577)
			\psdots[linecolor=black, dotsize=0.4](0.19711533,0.8985577)
			\psdots[linecolor=black, dotsize=0.4](1.7971153,-3.9014423)
			\psdots[linecolor=black, dotsize=0.4](0.19711533,-3.9014423)
			\psdots[linecolor=black, dotsize=0.4](4.1971154,-2.3014421)
			\psdots[linecolor=black, dotsize=0.4](5.7971153,-2.3014421)
			\psdots[linecolor=black, dotsize=0.4](4.997115,-0.70144224)
			\psline[linecolor=black, linewidth=0.08](4.1971154,-2.3014421)(4.997115,-0.70144224)(5.7971153,-2.3014421)(4.1971154,-2.3014421)(4.1971154,-2.3014421)
			\psdots[linecolor=black, dotsize=0.4](4.997115,0.8985577)
			\psdots[linecolor=black, dotsize=0.4](10.197115,-0.70144224)
			\psdots[linecolor=black, dotsize=0.4](11.797115,-0.70144224)
			\psdots[linecolor=black, dotsize=0.4](11.797115,-2.3014421)
			\psdots[linecolor=black, dotsize=0.4](10.197115,-2.3014421)
			\psline[linecolor=black, linewidth=0.08](10.197115,-0.70144224)(11.797115,-0.70144224)(11.797115,-2.3014421)(10.197115,-2.3014421)(10.197115,-0.70144224)(10.197115,-0.70144224)
			\psline[linecolor=black, linewidth=0.08](11.797115,-0.70144224)(10.197115,-2.3014421)(10.197115,-2.3014421)
			\psline[linecolor=black, linewidth=0.08](10.197115,-0.70144224)(11.797115,-2.3014421)(11.797115,-2.3014421)
			\psdots[linecolor=black, dotsize=0.4](11.797115,0.8985577)
			\psdots[linecolor=black, dotsize=0.4](10.197115,0.8985577)
			\psdots[linecolor=black, dotsize=0.4](11.797115,-3.9014423)
			\psdots[linecolor=black, dotsize=0.4](10.197115,-3.9014423)
			\psline[linecolor=black, linewidth=0.08](11.797115,-2.3014421)(11.797115,-3.9014423)(10.197115,-3.9014423)(10.197115,-2.3014421)(10.197115,-2.3014421)
			\psdots[linecolor=black, dotsize=0.4](12.997115,-0.70144224)
			\psdots[linecolor=black, dotsize=0.4](15.397116,-0.70144224)
			\psdots[linecolor=black, dotsize=0.4](16.997116,-0.70144224)
			\psdots[linecolor=black, dotsize=0.4](16.997116,-2.3014421)
			\psdots[linecolor=black, dotsize=0.4](15.397116,-2.3014421)
			\psline[linecolor=black, linewidth=0.08](15.397116,-0.70144224)(16.997116,-0.70144224)(16.997116,-2.3014421)(15.397116,-2.3014421)(15.397116,-0.70144224)(15.397116,-0.70144224)
			\psline[linecolor=black, linewidth=0.08](16.997116,-0.70144224)(15.397116,-2.3014421)(15.397116,-2.3014421)
			\psline[linecolor=black, linewidth=0.08](15.397116,-0.70144224)(16.997116,-2.3014421)(16.997116,-2.3014421)
			\psdots[linecolor=black, dotsize=0.4](16.997116,0.8985577)
			\psdots[linecolor=black, dotsize=0.4](15.397116,0.8985577)
			\psdots[linecolor=black, dotsize=0.4](16.997116,-3.9014423)
			\psdots[linecolor=black, dotsize=0.4](15.397116,-3.9014423)
			\psline[linecolor=black, linewidth=0.08](16.997116,-2.3014421)(16.997116,-3.9014423)(15.397116,-3.9014423)(15.397116,-2.3014421)(15.397116,-2.3014421)
			\psdots[linecolor=black, dotsize=0.4](18.197115,-2.3014421)
			\rput[bl](0.7971153,-5.401442){$G$}
			\rput[bl](4.8171153,-5.421442){$H$}
			\rput[bl](10.317116,-5.4814425){$G\cup_{K_3} H$}
			\rput[bl](15.457115,-5.5014424){$G\cup_{K_3} H$}
			\psline[linecolor=black, linewidth=0.04](0.19711533,0.8985577)(0.19711533,0.8985577)
			\psline[linecolor=black, linewidth=0.08](0.19711533,0.8985577)(0.19711533,-0.70144224)(0.19711533,-0.70144224)
			\psline[linecolor=black, linewidth=0.08](1.7971153,0.8985577)(1.7971153,-0.70144224)(1.7971153,-0.70144224)
			\psline[linecolor=black, linewidth=0.08](0.19711533,-2.3014421)(0.19711533,-3.9014423)(1.7971153,-3.9014423)(1.7971153,-2.3014421)(1.7971153,-2.3014421)
			\psline[linecolor=black, linewidth=0.08](4.997115,0.8985577)(4.997115,-0.70144224)(4.997115,-0.70144224)
			\psline[linecolor=black, linewidth=0.08](10.197115,0.8985577)(10.197115,-0.70144224)(10.197115,-0.70144224)
			\psline[linecolor=black, linewidth=0.08](11.797115,0.8985577)(11.797115,-0.70144224)(11.797115,-0.70144224)
			\psline[linecolor=black, linewidth=0.08](11.797115,-0.70144224)(12.997115,-0.70144224)(12.997115,-0.70144224)
			\psline[linecolor=black, linewidth=0.08](15.397116,0.8985577)(15.397116,-0.70144224)(15.397116,-0.70144224)
			\psline[linecolor=black, linewidth=0.08](16.997116,0.8985577)(16.997116,-0.70144224)(16.997116,-0.70144224)
			\psline[linecolor=black, linewidth=0.08](16.997116,-2.3014421)(18.197115,-2.3014421)(18.197115,-2.3014421)
			\end{pspicture}
		}
	\end{center}
	\caption{Graphs $G$, $H$ and all non-isomorphic graphs $G\cup_{K_3}H$, respectively.} \label{GcupHmorethanone}
\end{figure}
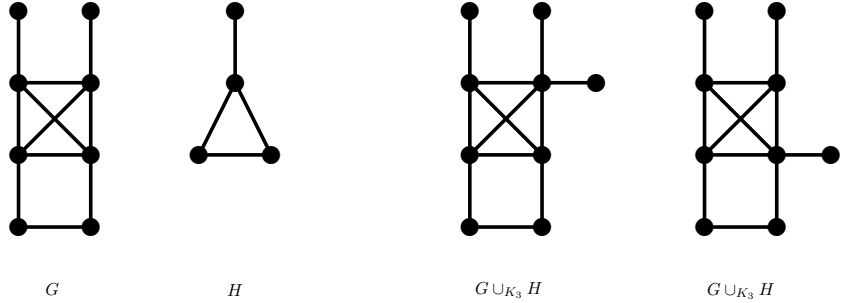

Let $G$ be a connected graph constructed from pairwise disjoint connected graphs $G_1,\ldots ,G_k$ as follows. Select a vertex of $G_1$, a vertex of $G_2$, and identify these two vertices. Thereafter, select a vertex in the resulting graph and select a vertex of $G_3$, and identify these two vertices.  We continue in this manner inductively to construct the graph $G$. Note that the graph $G$ constructed in this way has a tree-like structure, the $G_i$'s being its building stones.  We say that the resulting graph $G$ is obtained by point-attaching from $G_1,\ldots , G_k$ and that the graphs $G_i$'s are the primary subgraphs of $G$ (see, for example,~\cite{Alikhani1,MATCH,Nima0,Moster}).

In this paper, we study the strong domination number of $K_r$-gluing of two graphs and some particular cases of graphs obtained by their primary subgraphs.

\section{$K_r$-gluing}

In this section, we study the strong domination number of $K_r$-gluing of two graphs. By linearity, we have the following result on the strong domination number of the disjoint union of two graphs.

\begin{proposition}
\label{Prop:disconnected}
If $G_1$ and $G_2$ are graphs, then $\gst(G_1\cup_{K_0}G_2) =  \gst(G_1)+\gst(G_2)$.
\end{proposition}

For $k \ge 2$, given disjoint graphs $G_1,\ldots,G_k$ with $u_i\in V(G_i)$ where $i \in [k]$, the \emph{vertex}-\emph{sum} of $G_1,\ldots,G_k$, at the vertices $u_1,\ldots,u_k$, is the graph obtained from $G_1,\ldots,G_k$ by identifying the vertices $u_i$, $i \in [k]$ from each copy of $G_i$ into one new vertex~$u$, and is denoted by
\[
G_1 \underset{u}{+} G_2 \underset{u}{+} \ldots \underset{u}{+} G_k.
\]

Ghanbari and Alikhani~\cite{Strong-vsum} recently obtained the following result and showed that the bounds are tight.

\begin{theorem}{\rm (\cite{Strong-vsum})}
\label{thm:v-sum}
For the vertex-sum of disjoint graphs $G_1,G_2,\ldots,G_k$ with $u_i\in V(G_i)$ and $i \in [k]$ where $k \ge 2$, we have
\[
1 + \sum_{i=1}^{k} \left( \gst(G_i)-\deg_{G_i}(u_i) \right)
\le  \gst(G_1\underset{u}{+} G_2 \underset{u}{+} \cdots \underset{u}{+} G_k)
\le   \left( \sum_{i=1}^{k}\gst(G_i)\right)+1.
\]
\end{theorem}

Since vertex gluing of $G_1$ and $G_2$ is same as the vertex-sum of these two graphs, by Theorem~\ref{thm:v-sum} we have the following result.

\begin{theorem}{\rm (\cite{Strong-vsum})}
\label{thm:1-gluing}
For the vertex	gluing of graphs $G_1$ and $G_2$ with gluing vertices $u\in V(G_1)$ and $v\in V(G_2)$, we have
\[
\gst(G_1)+\gst(G_2)-\deg_{G_1}(u)-\deg_{G_2}(v) +1 \le  \gst(G_1\cup_{K_1}G_2) \le   \gst(G_1)+\gst(G_2)+1.
\]
\end{theorem}

We establish next lower and upper bounds on the edge gluing of two graphs $G_1$ and $G_2$ of orders at least~$3$, that is, we prove bounds on the $K_2$-gluing of two graphs. For $k \ge 1$ an integer, we use the standard notation $[k] = \{1,\ldots,k\}$.  In particular, $[2] = \{1,2\}$. We first establish an upper bound on the edge gluing of two graphs.

\begin{theorem}
\label{thm:2-gluing-upper}
For the edge gluing of connected graphs $G_1$ and $G_2$ both of order at least~$3$,
we have
\[
\gst(G_1\cup_{K_2}G_2) \le   \gst(G_1)+\gst(G_2)+1.
\]
\end{theorem}
\proof
Let $G = G_1 \cup_{K_2} G_2$, and let $uv$ be the edge in $G$ associated with the gluing of the edges  $u_1v_1\in E(G_1)$ and $u_2v_2\in E(G_2)$, where $u$ is the vertex resulting from the gluing of the vertices $u_1$ and $u_2$, and where $v$ is the vertex resulting from the gluing of the vertices $v_1$ and $v_2$. By the edge gluing operation, we note that $\deg_G(u) \ge \max \{ \deg_{G_1}(u_1),\deg_{G_2}(u_2)\}$ and $\deg_G(v) \ge \max \{\deg_{G_1}(v_1),\deg_{G_2}(v_2)\}$. Let $D_1$ and $D_2$ be $\gst$-sets of $G_1$ and $G_2$, respectively. Suppose firstly that $u_1 \in D_1$ and $v_1 \in D_1$. In this case, we let
\[
D=\Big( D_1 \cup D_2 \cup \{u,v\}\Big) \setminus \{u_1,v_1,u_2,v_2\}.
\]

If a vertex is strongly dominated by $u_i$ in $G_i$ for $i \in [2]$, then it is strongly dominated by $u$ in $G$. Analogously, if a vertex is strongly dominated by $v_i$ in $G_i$ for $i \in [2]$, then it is strongly dominated by $v$ in $G$. All other vertices in $\overline{D}$ are strongly dominated as before. Hence, $D$ is a strong dominating set of $G$, and so $\gst(G) \le |D| = (|D_1| - 2) + (|D_2 \setminus \{u_2,v_2\})| + 2 \le (|D_1| - 2 ) + |D_2| + 2 = |D_1| + |D_2| = \gst(G_1) + \gst(G_2)$. Hence we may assume that $u_1 \notin D_1$ or $v_1 \notin D_1$ (or both $u_1 \notin D_1$ and $v_1 \notin D_1$), for otherwise the desired upper bound follows. Renaming vertices if necessary, we may assume that $v_1 \notin D_1$. If $u_1 \in D_1$, then the set
\[
D = \Big(D_1 \cup D_2 \cup \{u,v\}\Big) \setminus\{u_1,u_2,v_2\} \1
\]
is a strong dominating set of $G$, and so $\gst(G) \le |D| = (|D_1| - 1) + |D_2| + 2 = \gst(G_1) + \gst(G_2) + 1$. Hence we may assume that $u_1 \notin D_1$. By our earlier assumption, $v_1 \notin D_1$. Thus, $\{u_1,v_1\} \cap D_1 = \emptyset$. Analogously, if at least one of $u_2$ and $v_2$ belong to the set $D_2$, then we infer that $\gst(G) \le \gst(G_1) + \gst(G_2) + 1$. Hence we may assume that $\{u_2,v_2\} \cap D_2 = \emptyset$. Without loss of generality, we may assume that we have $\deg_G(u) \ge \deg_G(v)$. In this case, the set $D=D_1 \cup D_2\cup \{u\}$ is a strong dominating set of $G$, and so $\gst(G) \le |D| = |D_1| + |D_2| + 1 = \gst(G_1) + \gst(G_2) + 1$.~\QED

\medskip
We next establish a lower bound on the edge gluing of two graphs.

\begin{theorem}
\label{thm:2-gluing-lower}
For the edge gluing of connected graphs $G_1$ and $G_2$ both of order at least~$3$, with gluing edges $u_1v_1\in E(G_1)$ and $u_2v_2\in E(G_2)$, we have
\[
\gst(G_1\cup_{K_2}G_2) \ge \gst(G_1)+\gst(G_2) + 5 - \Psi_{1,2},
\]
where
\begin{equation}
\label{Eq:1}
\Psi_{1,2} = \sum_{i=1}^2 ( \deg_{G_i}(u_i) + \deg_{G_i}(v_i)).
\end{equation}
%Further if $\Psi_{1,2} > 6$, then $\gst(G_1)+\gst(G_2) + 6 - \Psi_{1,2} \le  \gst(G_1\cup_{K_2}G_2)$.
\end{theorem}
\proof
Let $G = G_1 \cup_{K_2} G_2$, and let $uv$ be the edge in $G$ associated with the gluing of the edges  $u_1v_1\in E(G_1)$ and $u_2v_2\in E(G_2)$, where $u$ is the vertex resulting from the gluing of the vertices $u_1$ and $u_2$, and where $v$ is the vertex resulting from the gluing of the vertices $v_1$ and $v_2$. Let $\Phi_{1,2}$ denote the desired lower bound, that is,
\begin{equation}
\label{Eq:2}
\Phi_{1,2} = \gst(G_1) + \gst(G_2) + 5 - \Psi_{1,2}.
\end{equation}

Hence we wish to show that $\gst(G) \ge \Phi_{1,2}$. Let $X_i = N_{G_i}(u_i) \setminus\{v_i\}$ and let $Y_i = N_{G_i}(v_i) \setminus\{u_i\}$ for $i \in [2]$. We note that $\deg_{G_i}(u_i) = |X_i| + 1$ and $\deg_{G_i}(v_i) = |Y_i| + 1$ for $i \in [2]$, and so Equation~(\ref{Eq:1}) can be written as follows.
\begin{equation}
\label{Eq:3}
\Psi_{1,2} = 4 + \sum_{i=1}^2 |X_i| + \sum_{i=1}^2 |Y_i|.
\end{equation}

Since $G_1$ and $G_2$ are connected graphs of orders at least~$3$, we note that
\begin{equation}
\label{Eq:3b}
|X_1| + |Y_1| \ge 1 \hspace*{0.5cm} \mbox{and} \hspace*{0.5cm} |X_2| + |Y_2| \ge 1,
\end{equation}
and so, by Equation~(\ref{Eq:3}), we have $\Psi_{1,2} \ge 4 + 2 = 6$. In particular, this implies by Equation~(\ref{Eq:2}) that
\begin{equation}
\label{Eq:3c}
\Phi_{1,2} \le \gst(G_1) + \gst(G_2) - 1.
\end{equation}

Among all $\gst$-set of $G$, let $S$ be chosen so that $|S \cap \{u,v\}|$ is a minimum. We construct strong dominating sets for $G_1$ and $G_2$ based on $S$. We note that the vertices $u_i$ and $v_i$ do not belong to the graph $G$ (as these vertices are renamed $u$ and $v$, respectively, in $G$), and therefore $u_i \notin S$ and $v_i \notin S$ for $i \in [2]$. Let
\[
V_i = V(G_i) \setminus \{u_i,v_i\} \1
\]
for $i \in [2]$. Thus, $V(G) = V_1 \cup V_2 \cup \{u,v\}$. For $i \in [2]$, we define the vertex $u_i$ as \emph{strong} if $X_i = \emptyset$ or if $|X_i| \ge 1$ and $\deg_{G_i}(u_i) \ge \max \{ \deg_{G_i}(x) \colon x \in X_i\}$, and we define $u_i$ as \emph{weak} otherwise, that is, $u_i$ is weak if $|X_i| \ge 1$ and $\deg_{G_i}(u_i) < \max \{ \deg_{G_i}(x) \colon x \in X_i\}$. Analogously for $i \in [2]$, we define the vertex $v_i$ as \emph{strong} if $Y_i = \emptyset$ or if $|Y_i| \ge 1$ and $\deg_{G_i}(v_i) \ge \max \{ \deg_{G_i}(y) \colon y \in X_i\}$, and we define $v_i$ as \emph{weak} otherwise. Let $A = \{u_1,v_1,u_2,v_2\}$.

We proceed further with four subclaims.

\begin{claim}
\label{c:claim1}
If $u\in S$ and $v\in S$, then $\gst(G) \ge \Phi_{1,2}$.
\end{claim}
\proof Suppose that $u \in S$ and $v\in S$. By our choice of the set $S$, this implies that every $\gst$-set of $G$ contains both $u$ and $v$.

\begin{subclaim}
\label{c:claim1.1}
The following properties hold. \\[-24pt]
\begin{enumerate}
\item[{\rm (a)}] $|X_1| + |X_2| \ge 2$ and $|Y_1| + |Y_2| \ge 2$.
\item[{\rm (b)}] $\deg_G(u) \ge 3$ and $\deg_G(v) \ge 3$.
\item[{\rm (c)}] $\Psi_{1,2} \ge 8$.
\end{enumerate}
\end{subclaim}
\proof (a) Suppose that $|X_1| + |X_2| \le 1$. If $X_1 = X_2 = \emptyset$, then $\deg_{G_i}(u_i) = 1$ and $v_i$ is the unique neighbor of $u_i$ in $G_i$ for $i \in [2]$. Thus, $\deg_G(u) = 1$, and so the vertex~$v$ is the unique neighbor of the vertex~$u$ in the graph $G$, implying that the set $S \setminus \{u\}$ is a strong dominating set of $G$, contradicting the minimality of the set $S$. Hence, $|X_1| + |X_2| = 1$. Let $X_1 \cup X_2 = \{x\}$, and so $N_G(u) = \{v,x\}$. In this case, the set $(S \setminus \{u\}) \cup \{x\}$ is a strong dominating set of $G$ noting that the vertex $v$ has degree at least~$2$ in $G$,  contradicting our earlier observation that every $\gst$-set of $G$ contains both $u$ and~$v$. Hence, $|X_1| + |X_2| \ge 2$. Analogously, $|Y_1| + |Y_2| \ge 2$.

(b) Since $\deg_G(u) = |X_1| + |X_2| + 1$ and $\deg_G(v) = |Y_1| + |Y_2| + 1$, by part~(a) we infer that $\deg_G(u) \ge 3$ and $\deg_G(v) \ge 3$.

(c) This follows from Equation~(\ref{Eq:3}) and Part~(a).~\smallqed

\begin{subclaim}
\label{c:claim1.2}
If all four vertices in $A$ are strong, then $\gst(G) \ge \Phi_{1,2}$.
\end{subclaim}
\proof Suppose that all four vertices in $A$ are strong. In this case, the set
\[
S_i = (S \cap V_i) \cup \{u_i,v_i\} \1
\]
is a strong dominating set of $G_i$ for $i \in [2]$, and so $\gst(G_i)  \le |S_i|$ for $i \in [2]$. Thus by Equation~(\ref{Eq:2}) we have $\Phi_{1,2} + \Psi_{1,2} - 5 = \gst(G_1) + \gst(G_2)  \le |S_1| + |S_2| = |S| + 4 - 2 = \gst(G) + 2$, or, equivalently, $\gst(G) \ge \Phi_{1,2} + \Psi_{1,2} - 7$. By Claim~\ref{c:claim1.1}(c), $\Psi_{1,2} \ge 8$, implying that $\gst(G) \ge \Phi_{1,2} + 1 > \Phi_{1,2}$.~\smallqed

\begin{subclaim}
\label{c:claim1.3}
If exactly three vertices in $A$ are strong, then $\gst(G) \ge \Phi_{1,2}$.
\end{subclaim}
\proof Suppose that exactly three vertices in $A$ are strong. Renaming vertices if necessary, we may assume that $u_1$ is weak, and so the vertices $v_1,u_2,v_2$ are strong. We now consider the sets
\[
S_1 = (S \cap V_1) \cup (\{v_1\} \cup X_1) \hspace*{0.5cm} \mbox{and} \hspace*{0.5cm} S_2 = (S \cap V_2) \cup \{u_2,v_2\}.
\]

The set $S_i$ is a strong dominating set of $G_i$ for $i \in [2]$. By Claim~\ref{c:claim1.1}(a), $|X_2| + |Y_2| \ge 2$. Thus since $|Y_1| \ge 0$, we infer from Equation~(\ref{Eq:3}) that
\[
\begin{array}{lcl}
\gst(G_1) + \gst(G_2) & \le & |S_1| + |S_2| \1 \\
& \le & \displaystyle{ |S| + 1 + |X_1|  } \2 \\
& = & \displaystyle{ |S| + 1 + \Psi_{1,2} - 4 - |X_2| - |Y_1| - |Y_2| } \2 \\
& \le & \displaystyle{ |S| + 1 + \Psi_{1,2} - 4 - 2 } \2 \\
& \le & \displaystyle{ |S| - 5 + \Psi_{1,2}} \2 \\
\end{array}
\]
yielding the desired result.~\smallqed

\begin{subclaim}
\label{c:claim1.4}
If exactly two vertices in $A$ are strong, then $\gst(G) \ge \Phi_{1,2}$.
\end{subclaim}
\proof Suppose that exactly two vertices in $A$ are strong. Suppose that $u_1$ and $u_2$ are strong and $v_1$ and $v_2$ are weak, or $u_1$ and $u_2$ are weak and $v_1$ and $v_2$ are strong. Renaming vertices if necessary, we may assume that the vertices $u_1$ and $u_2$ are strong and $v_1$ and $v_2$ are weak. In this case, the set
\[
S_i = ( S \cap V_i ) \cup (\{u_i\} \cup Y_i) \1
\]
is a strong dominating set of $G_i$ for $i \in [2]$, implying by Equation~(\ref{Eq:3}) and Claim~\ref{c:claim1.1}(a) that

\[
\begin{array}{lcl}
\gst(G_1) + \gst(G_2) & \le & |S_1| + |S_2| \1 \\
& \le & \displaystyle{ |S| - 1 + \sum_{i=1}^2 |Y_i|  } \2 \\
& = & \displaystyle{ |S| + \Psi_{1,2} - 5 - \sum_{i=1}^2 |X_i| } \1 \\
& \le & \displaystyle{ \gst(G) - 7 + \Psi_{1,2} }, \1 \\
\end{array}
\]
or, equivalently, $\gst(G) \ge \gst(G_1) + \gst(G_2) + 7 - \Psi_{1,2} > \Phi_{1,2}$, as desired. Hence we may assume that exactly one of $u_1$ and $u_2$ is strong and exactly one of $v_1$ and $v_2$ is strong.

Suppose that the vertices $u_1$ and $v_1$ are strong and $u_2$ and $v_2$ are weak. In this case, the set
\[
S_2 = (S \cap V_2) \cup (X_2 \cup Y_2) \1
\]
is a strong dominating set of $G_2$. By Inequality~(\ref{Eq:3b}), $|X_1| + |Y_1| \ge 1$.  If $|X_1| + |Y_1| = 1$, then renaming $u_1$ and $v_1$ if necessary, we may assume that $\deg_{G_1}(u_1) = 1$ and $\deg_{G_1}(v_1) = 2$. In this case, the set
\[
S_1 = (S \cap V_1) \cup \{v_1\} \1
\]
is a strong dominating set of $G_1$, and so by Equation~(\ref{Eq:3}), we have
\[
\begin{array}{lcl}
\gst(G_1) + \gst(G_2) & \le & |S_1| + |S_2| \1 \\
& \le & \displaystyle{ |S| - 2 + 1 + |X_2| + |Y_2|  } \1 \\
& = & \displaystyle{ |S| - 1 + \Psi_{1,2} - 4 - |X_1| - |Y_1| } \1 \\
& = & \displaystyle{ |S| + \Psi_{1,2} - 6}.
\end{array}
\]

If $|X_1| + |Y_1| \ge 2$, then we consider the set
\[
S_1 = (S \cap V_1) \cup \{u_1,v_1\} \1
\]
which is a strong dominating set of $G_1$. By Equation~(\ref{Eq:3}) and Claim~\ref{c:claim1.1}(a), we have
\[
\begin{array}{lcl}
\gst(G_1) + \gst(G_2) & \le & |S_1| + |S_2| \1 \\
& \le & \displaystyle{ |S| - 2 + 2 + |X_2| + |Y_2|  } \1 \\
& = & \displaystyle{ |S| + \Psi_{1,2} - 4 - |X_1| - |Y_1| } \1 \\
& \le & \displaystyle{ |S| + \Psi_{1,2} - 6}.
\end{array}
\]

In both cases, $\gst(G) \ge \gst(G_1) + \gst(G_2) + 6 - \Psi_{1,2} = \Phi_{1,2} + 1 > \Phi_{1,2}$, as desired. Analogous argument yield the desired bound when the vertices $u_1$ and $v_1$ are weak and $u_2$ and $v_2$ are strong.

Renaming vertices if necessary, we may therefore assume that the vertices $u_1$ and $v_2$ are strong and $u_2$ and $v_1$ are weak. We now consider the sets
\[
S_1 = (S \cap V_1) \cup (\{u_1\} \cup Y_1) \hspace*{0.5cm} \mbox{and} \hspace*{0.5cm} S_2 = (S \cap V_2)  \cup (\{v_2\} \cup X_2).
\]

The sets $S_1$ and $S_2$ are strong dominating set of $G_1$ and $G_2$, respectively. By Equation~(\ref{Eq:3}), we have
\[
\begin{array}{lcl}
\gst(G_1) + \gst(G_2) & \le & |S_1| + |S_2| \1 \\
& \le & \displaystyle{ |S| - 2 + 2 + |Y_1| + |X_2| } \1 \\
& = & \displaystyle{ |S| + \Psi_{1,2} - 4 - |X_1| - |Y_2| }.
\end{array}
\]

If $|X_1| + |Y_2| \ge 1$, then $\gst(G) = |S| \ge \gst(G_1) + \gst(G_2) + 5 - \Psi_{1,2} = \Phi_{1,2}$, as desired. Hence we may assume that $|X_1| + |Y_2| = 0$. By Claim~\ref{c:claim1.1}(a), $|X_2| \ge 2$ and $|Y_1| \ge 2$, and so $\deg_G(u) = |X_2| + 1 \ge 3$ and $\deg_G(v) = |Y_1| + 1 \ge 3$. In this case,  we consider the sets
\[
S_1 = (S \cap V_1) \cup \{v_1\} \hspace*{0.5cm} \mbox{and} \hspace*{0.5cm} S_2 = (S \cap V_2) \cup \{u_2\}.
\]

The sets $S_1$ and $S_2$ are strong dominating set of $G_1$ and $G_2$, respectively. Thus, $\gst(G_1) + \gst(G_2) \le |S_1| + |S_2| = |S| = \gst(G)$, implying by Inequality~(\ref{Eq:3c}) that $\gst(G) \ge \Phi_{1,2} + 1 > \Phi_{1,2} \le \gst(G)$, as desired.~\smallqed

\begin{subclaim}
\label{c:claim1.5}
If exactly one vertex in $A$ is strong, then $\gst(G) \ge \Phi_{1,2}$.
\end{subclaim}
\proof Suppose that exactly one vertex in $A$ is strong. Renaming vertices if necessary, we may assume that $u_1$ is strong, and so the vertices $u_2$, $v_1$ and $v_2$ are weak. In particular, this implies that $|X_2| \ge 1$ and  $|Y_i| \ge 1$ for $i \in [2]$. We now consider the sets
\[
S_1 = (S \cap V_1) \cup (\{u_1\} \cup Y_1) \hspace*{0.5cm} \mbox{and} \hspace*{0.5cm} S_2 = (S \cap V_2) \cup (X_2 \cup Y_2).
\]

The set $S_i$ is a strong dominating set of $G_i$ for $i \in [2]$, implying by Equation~(\ref{Eq:3}) that
\[
\begin{array}{lcl}
\gst(G_1) + \gst(G_2) & \le & |S_1| + |S_2| \1 \\
& \le & \displaystyle{ |S| - 1 + |X_2| + \sum_{i=1}^2 |Y_i|  } \2 \\
& = & \displaystyle{ |S| - 5 + \Psi_{1,2} - |X_1|} \2 \\
& \le & \displaystyle{ |S| - 5 + \Psi_{1,2},}
\end{array}
\]
yielding the desired result.~\smallqed

\begin{subclaim}
\label{c:claim1.6}
If no vertex in $A$ is strong, then $\gst(G) \ge \Phi_{1,2}$.
\end{subclaim}
\proof Suppose that no vertex in $A$ are strong, and so all four vertices in $A$ are weak. In particular, this implies that $|X_i| \ge 1$ and  $|Y_i| \ge 1$ for all $i \in [2]$. The set
\[
S_i = ( S \cap V_i ) \cup (X_i \cup Y_i)  \1
\]
is a strong dominating set of $G_i$ for $i \in [2]$, implying by Equation~(\ref{Eq:3}) that
\[
\begin{array}{lcl}
\gst(G_1) + \gst(G_2) & \le & |S_1| + |S_2| \1 \\
& \le & \displaystyle{ |S| - 2 + \sum_{i=1}^2 |X_i| + \sum_{i=1}^2 |Y_i|  } \2 \\
& = & \displaystyle{ |S| - 6 + \Psi_{1,2} }, \1
\end{array}
\]
or, equivalently, $\gst(G) = |S| \ge \gst(G_1) + \gst(G_2) + 6 - \Psi_{1,2} = \Phi_{1,2} + 1 > \Phi_{1,2}$.~\smallqed

\medskip
The proof of Claim~\ref{c:claim1} now follows from Claims~\ref{c:claim1.1},~\ref{c:claim1.2},~\ref{c:claim1.3},~\ref{c:claim1.4},~\ref{c:claim1.5}, and~\ref{c:claim1.6}.~\smallqed

\begin{claim}
\label{c:claim2}
If exactly one of $u$ and $v$ belong to~$S$, then $\gst(G) \ge \Phi_{1,2}$.
\end{claim}
\proof Renaming vertices if necessary, we may assume that $u \in S$ and $v \notin S$.  We proceed further with the following subclaims.

\begin{subclaim}
\label{c:claim2.1}
If both $u_1$ and $u_2$ are strong, then $\gst(G) \ge \Phi_{1,2}$.
\end{subclaim}
\proof Suppose that both $u_1$ and $u_2$ are strong. In this case, the set $S_i = (S \cap V_i) \cup \{u_i\}$ is a strong dominating set of $G_i$ for $i \in [2]$, and so $\gst(G_i)  \le |S_i|$ for $i \in [2]$. Thus, by Equation~(\ref{Eq:2}) we have $\Phi_{1,2} + \Psi_{1,2} - 5 = \gst(G_1) + \gst(G_2)  \le |S_1| + |S_2| = |S| + 2 - 1 = \gst(G) + 1$, or, equivalently, $\gst(G) \ge \Phi_{1,2} + \Psi_{1,2} - 6 \ge \Phi_{1,2}$, as desired.~\smallqed

\begin{subclaim}
\label{c:claim2.2}
If exactly one of $u_1$ and $u_2$ is strong, then $\gst(G) \ge \Phi_{1,2}$.
\end{subclaim}
\proof Renaming vertices if necessary, we may assume that $u_1$ is strong and $u_2$ is weak. In this case we let
\[
S_1 = (S \cap V_1) \cup \{u_1\} \hspace*{0.5cm} \mbox{and} \hspace*{0.5cm} S_2 = (S \cap V_2)  \cup X_2.
\]

The set $S_i$ is a strong dominating set of $G_i$ for $i \in [2]$. By Claim~\ref{c:claim1.1}(a), $|Y_1| + |Y_2| \ge 2$. Thus since $|X_1| \ge 0$, we infer from Equation~(\ref{Eq:3}) that
\[
\begin{array}{lcl}
\gst(G_1) + \gst(G_2) & \le & |S_1| + |S_2| \1 \\
& \le & \displaystyle{ |S| + |X_2|  } \2 \\
& = & \displaystyle{ |S| + \Psi_{1,2} - 4 - |X_1| - |Y_1| - |Y_2| } \1 \\
& \le & \displaystyle{ |S| + \Psi_{1,2} - 6 }, \1
\end{array}
\]
or, equivalently, $\gst(G) = |S| \ge \gst(G_1) + \gst(G_2) + 6 - \Psi_{1,2} = \Phi_{1,2} + 1 > \Phi_{1,2}$.~\smallqed

\begin{subclaim}
\label{c:claim2.3}
If neither $u_1$ nor $u_2$ is strong, then $\gst(G) \ge \Phi_{1,2}$.
\end{subclaim}
\proof In this case when neither $u_1$ nor $u_2$ is strong,  we let
\[
S_i = (S \cap V_i)  \cup X_i
\]
for $i \in [2]$. The set $S_i$ is a strong dominating set of $G_i$ for $i \in [2]$. By Equation~(\ref{Eq:3}), noting that $|Y_1| + |Y_2| \ge 2$, we have
\[
\begin{array}{lcl}
\gst(G_1) + \gst(G_2) & \le & |S_1| + |S_2| \1 \\
& \le & \displaystyle{ |S| - 1 + |X_1| + |X_2|  } \2 \\
& = & \displaystyle{ |S| - 1 + \Psi_{1,2} - 4 - |Y_1| - |Y_2| } \1 \\
& \le & \displaystyle{ |S| + \Psi_{1,2} - 7 }, \1
\end{array}
\]
or, equivalently, $\gst(G) = |S| \ge \gst(G_1) + \gst(G_2) + 7 - \Psi_{1,2} = \Phi_{1,2} + 2 > \Phi_{1,2}$.~\smallqed

\medskip
The proof of Claim~\ref{c:claim2} now follows from Claims~\ref{c:claim2.1},~\ref{c:claim2.2}, and~\ref{c:claim2.3}.~\smallqed

\begin{claim}
\label{c:claim3}
If neither $u$ nor $v$ belong to~$S$, then $\gst(G) \ge \Phi_{1,2}$.
\end{claim}
\proof Suppose that neither $u$ nor $v$ belong to~$S$.  In this case we let $S_1 = S \cap V_1$ and we let $S_2 = (S \cap V_2) \cup \{u_2\}$. The set $S_i$ is a strong dominating set of $G_i$ for $i \in [2]$. Thus, by Equation~(\ref{Eq:2}) we have $\Phi_{1,2} + \Psi_{1,2} - 5 = \gst(G_1) + \gst(G_2)  \le |S_1| + |S_2| = |S| + 1 = \gst(G) + 1$, or, equivalently, $\gst(G) \ge \Phi_{1,2} + \Psi_{1,2} - 6 \ge \Phi_{1,2}$, as desired.~\smallqed

\medskip
By Claims~\ref{c:claim1},~\ref{c:claim2}, and~\ref{c:claim3} we have $\gst(G) \ge \Phi_{1,2}$. This completes the proof of Theorem~\ref{thm:2-gluing-lower}.~\smallqed

\medskip
We remark that the lower bound in Theorem~\ref{thm:2-gluing-lower} is best possible. As a simple example, let $G_1$ and $G_2$ be the two graphs shown in Figure~\ref{f:example}(a) and~\ref{f:example}(b), respectively. Let $G = G_1 \cup_{K_2} G_2$, and let $uv$ be the edge in $G$ associated with the gluing of the edges  $u_1v_1\in E(G_1)$ and $u_2v_2\in E(G_2)$, where $u$ is the vertex resulting from gluing the vertices $u_1$ and $u_2$, and where $v$ is the vertex resulting from gluing the vertices $v_1$ and $v_2$. The graph $G$ is shown in Figure~\ref{f:example}(c). In this example, we have $\gst(G_1) = 1$,  $\gst(G_2) = 2$, and $\gst(G) = 2$, where the black vertices are $\gst$-sets in $G_1$, $G_2$ and $G$, respectively. In this example,
\[
\Psi_{1,2} = \sum_{i=1}^2 ( \deg_{G_i}(u_i) + \deg_{G_i}(v_i)) = 6.
\]
Thus, the lower bound in the statement of Theorem~\ref{thm:2-gluing-lower} is
\[
\Phi_{1,2} = \gst(G_1) + \gst(G_2) + 5 - \Psi_{1,2} = 1 + 2 + 5 - 7 = 2.
\]
As observed earlier, $\gst(G) = 2$. Therefore, $\gst(G) = \Phi_{1,2}$, that is, the graph $G$ achieves the lower bound in Theorem~\ref{thm:2-gluing-lower}.

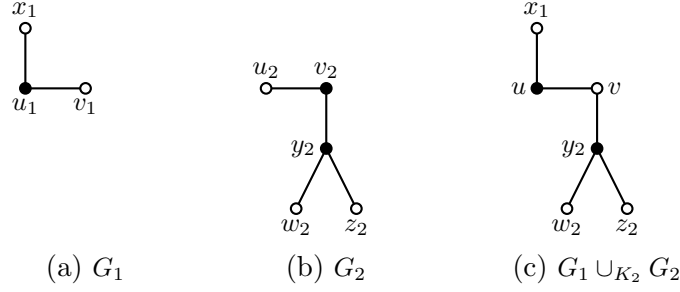
\begin{figure}[htb]
\begin{center}
\begin{tikzpicture}[scale=.8,style=thick,x=1cm,y=1cm]
\def\vr{2.5pt} % \vr = vertex radius;
% define vertices
%%%
\path (0,2) coordinate (u1);
\path (0,3) coordinate (x1);
\path (1,2) coordinate (v1);
\path (4,2) coordinate (u2);
\path (5,2) coordinate (v2);
\path (5,1) coordinate (y2);
\path (4.5,0) coordinate (w2);
\path (5.5,0) coordinate (z2);
\path (8.5,3) coordinate (x1n);
\path (8.5,2) coordinate (u);
\path (9.5,2) coordinate (v);
\path (9.5,1) coordinate (y2n);
\path (9,0) coordinate (w2n);
\path (10,0) coordinate (z2n);
%   edges
\draw (x1)--(u1)--(v1);
\draw (u2)--(v2)--(y2)--(w2);
\draw (y2)--(z2);
\draw (x1n)--(u)--(v)--(y2n)--(w2n);
\draw (y2n)--(z2n);
\draw (x1) [fill=white] circle (\vr);
\draw (u1) [fill=black] circle (\vr);
\draw (v1) [fill=white] circle (\vr);
\draw (u2) [fill=white] circle (\vr);
\draw (v2) [fill=black] circle (\vr);
\draw (y2) [fill=black] circle (\vr);
\draw (w2) [fill=white] circle (\vr);
\draw (z2) [fill=white] circle (\vr);
\draw (x1n) [fill=white] circle (\vr);
\draw (u) [fill=black] circle (\vr);
\draw (v) [fill=white] circle (\vr);
\draw (y2n) [fill=black] circle (\vr);
\draw (w2n) [fill=white] circle (\vr);
\draw (z2n) [fill=white] circle (\vr);
%%%
\draw (1,-1) node {(a) {\small $G_1$}};
\draw (5,-1) node {(b) {\small $G_2$}};
\draw (9.5,-1) node {(c) {\small $G_1 \cup_{K_2} G_2$}};
\draw[anchor = south] (x1) node {{\small $x_1$}};
\draw[anchor = north] (u1) node {{\small $u_1$}};
\draw[anchor = north] (v1) node {{\small $v_1$}};
\draw[anchor = south] (u2) node {{\small $u_2$}};
\draw[anchor = south] (v2) node {{\small $v_2$}};
\draw[anchor = east] (y2) node {{\small $y_2$}};
\draw[anchor = north] (w2) node {{\small $w_2$}};
\draw[anchor = north] (z2) node {{\small $z_2$}};
\draw[anchor = south] (x1n) node {{\small $x_1$}};
\draw[anchor = east] (u) node {{\small $u$}};
\draw[anchor = west] (v) node {{\small $v$}};
\draw[anchor = east] (y2n) node {{\small $y_2$}};
\draw[anchor = north] (w2n) node {{\small $w_2$}};
\draw[anchor = north] (z2n) node {{\small $z_2$}};
\end{tikzpicture}
\end{center}
\vskip -0.5cm
\caption{A graph achieving the lower bound in Theorem~\ref{thm:2-gluing-lower}} \label{f:example}
\end{figure}

We remark that Theorem~\ref{thm:2-gluing-lower} can be strengthened slightly if we impose the restriction that $\Psi_{1,2} \ge 7$. The proof is almost identical to that of Theorem~\ref{thm:2-gluing-lower}, and is therefore omitted.

\begin{theorem}
\label{thm:2-gluing-lower2}
For the edge gluing of connected graphs $G_1$ and $G_2$ both of order at least~$3$, with gluing edges $u_1v_1\in E(G_1)$ and $u_2v_2\in E(G_2)$ and with
\[
\Psi_{1,2} = \sum_{i=1}^2 ( \deg_{G_i}(u_i) + \deg_{G_i}(v_i)) \ge 7,
\]
we have
\[
\gst(G_1\cup_{K_2}G_2) \ge \gst(G_1)+\gst(G_2) + 6 - \Psi_{1,2}.
\]
\end{theorem}

The result in Theorem~\ref{thm:2-gluing-upper} on the $K_2$-gluing of two graphs can be extended to $K_r$-gluing of two graphs for all $r \ge 2$.

\begin{theorem}
\label{thm:2-gluing-upper-Kr}
For $r \ge 2$, let $G_1$ and $G_2$ be connected graphs of order at least~$r+1$ that both contain $r$-cliques. For the $r$-gluing of $G_1$ and $G_2$, we have
\[
\gst(G_1 \cup_{K_r} G_2) \le \gst(G_1) + \gst(G_2) + 1.
\]
\end{theorem}
\proof
Let $G = G_1 \cup_{K_r} G_2$, and let $R_1$ and $R_2$ be the two $r$-cliques chosen in $G_1$ and $G_2$, respectively, that are identified in the $r$-gluing of $G_1$ and $G_2$ to produce the $r$-clique $R$ in $G$. Let $R_i = \{v_{i,1},v_{i,2}, \ldots, v_{i,r}\}$ for $i \in [2]$, and let $R = \{v_1,v_2,\ldots,v_r\}$, where $v_j$ is the vertex resulting from the gluing of the vertices $v_{1,j}$ and $v_{2,j}$ for $j \in [r]$. By the edge gluing operation, we note that $\deg_G(v_j) \ge \max \{ \deg_{G_1}(v_{1,j}),\deg_{G_2}(v_{2,j})\}$ for $j \in [r]$. Let $D_1$ and $D_2$ be $\gst$-sets of $G_1$ and $G_2$, respectively. Let $A_i = D_i \cap R_i$ for $i \in [2]$, and so $A_i$ is the set of vertices in $D_i$ that belong to the $r$-clique $R_i$ in the graph $G_i$. Let
\[
A = \{ v_j \in Q \colon v_{1,j} \in D_1 \mbox{ or } v_{2,j} \in D_1 \mbox{ where } j \in [r] \}.
\]

We note that $|A_1| + |A_2| \ge |A|$. Suppose that $A = R$, and so $|A| = r$ and $|A_1| + |A_2| \ge r$. In this case, we let
\[
D = \Big( D_1 \cup D_2 \cup A \Big) \setminus (A_1 \cup A_2).
\]

The resulting set $D$ is a strong dominating set of $G$, and so $\gst(G) \le |D| = |D_1| + |D_2| + |A| - (|A_1| + |A_2|) \le \gst(G_1) + \gst(G_2)$. Suppose next that $A \subset R$. Let $B = R \setminus A$. Among all vertices that belong to the set $B$ we choose a vertex $v$ of maximum degree in $G$, and we let
\[
D = \Big( D_1 \cup D_2 \cup A \cup \{v\} \Big) \setminus (A_1 \cup A_2).
\]

The resulting set $D$ is a strong dominating set of $G$, noting that the vertex $v$ strongly dominates all vertices the vertices that belong to the set $B$ in the graph $G$. Therefore, $\gst(G) \le |D| = |D_1| + |D_2| + |A| + 1 - (|A_1| + |A_2|) \le \gst(G_1) + \gst(G_2) + 1$.~\QED

\medskip
We pose the following question that we have yet to settle which, if true, extends the result in Theorem~\ref{thm:2-gluing-lower} on the $K_2$-gluing of two graphs to $K_r$-gluing of two graphs for all $r \ge 2$

\begin{question}
\label{Q:2-gluing-lower-Kr}
For $r \ge 2$, let $G_1$ and $G_2$ be connected graphs of order at least~$r+1$ that both contain $r$-cliques. Is it true that for the $r$-gluing of $G_1$ and $G_2$ with $r$-gluing cliques $U_1 \subset V(G_1)$ and $U_2 \subset V(G_2)$, we have
\begin{equation}
\label{Eq:r}
\gst(G_1\cup_{K_r}G_2) \ge \gst(G_1)+\gst(G_2) + 2r^2 - 2r + 1 - \Psi_1^r - \Psi_2^r,
\end{equation}
where
\[
\Psi_{i}^r = \sum_{u \in U_i} \deg_{G_i}(u_i)
\]
for all $i \in [r]$.
\end{question}

We remark that if Question~\ref{Q:2-gluing-lower-Kr} is true, then the lower bound in Inequality~(\ref{Eq:r}) in the statement of the question is best possible. As a simple example, let $r \ge 2$ and let $G_1$ be obtained from a complete graph $K_r$ with vertex set $U_1$ by adding a new vertex $x_1$ and adding an edge joining $x_1$ to exactly one vertex $u_1$ in the complete graph. Further, let $G_2$ be obtained from a complete graph $K_r$ with vertex set $U_2$ by adding a path $w_2y_2z_2$ and an edge joining $y_2$ to exactly one vertex $v_2$ in the complete graph. In the special case when $r=3$, the graphs $G_1$ and $G_2$ are shown in Figure~\ref{f:example3}(a) and~\ref{f:example3}(b), respectively. Let $G = G_1 \cup_{K_r} G_2$, where the $r$-gluing of $G_1$ and $G_2$ is with $r$-gluing cliques $U_1$ and $U_2$. In this example, we have $\gst(G_1) = 1$,  $\gst(G_2) = 2$, $\gst(G) = 2$, and
\[
\Psi_{i}^r = \sum_{u \in U_i} \deg_{G_i}(u_i) = r^2 - r + 1
\]
for $i \in [2]$. Thus, the lower bound in the statement of Question~\ref{Q:2-gluing-lower-Kr} is
\[
\gst(G_1) + \gst(G_2) + 2r^2 - 2r + 1 - \Psi_1^r - \Psi_2^r = 1 + 2 + 2r^2 - 2r + 1 - 2(r^2 - r + 1) = 2.
\]
As observed earlier, $\gst(G) = 2$. Therefore, the graph $G$ achieves equality in Inequality~(\ref{Eq:r}).

\begin{figure}[htb]
\begin{center}
\begin{tikzpicture}[scale=.8,style=thick,x=1cm,y=1cm]
\def\vr{2.5pt} % \vr = vertex radius;
% define vertices
%%%
\path (0,2) coordinate (u1);
\path (0,3) coordinate (x1);
\path (1,2) coordinate (w1);
\path (2,2) coordinate (v1);
\path (4,2) coordinate (u2);
\path (5,2) coordinate (n2);
\path (6,2) coordinate (v2);
\path (6,1) coordinate (y2);
\path (5.5,0) coordinate (w2);
\path (6.5,0) coordinate (z2);
\path (8.5,3) coordinate (x1n);
\path (8.5,2) coordinate (u);
\path (9.5,2) coordinate (n);
\path (10.5,2) coordinate (v);
\path (10.5,1) coordinate (y2n);
\path (10,0) coordinate (w2n);
\path (11,0) coordinate (z2n);
%   edges
\draw (x1)--(u1)--(w1)--(v1);
\draw (v2)--(y2)--(w2);
\draw (u2)--(n2)--(v2);
\draw (y2)--(z2);
\draw (x1n)--(u)--(v)--(y2n)--(w2n);
\draw (y2n)--(z2n);

\draw (u1) to[out=270,in=270, distance=0.75cm] (v1);
\draw (u2) to[out=270,in=270, distance=0.75cm] (v2);
\draw (u) to[out=270,in=270, distance=0.75cm] (v);
\draw (x1) [fill=white] circle (\vr);
\draw (u1) [fill=black] circle (\vr);
\draw (v1) [fill=white] circle (\vr);
\draw (w1) [fill=white] circle (\vr);
\draw (u2) [fill=white] circle (\vr);
\draw (n2) [fill=white] circle (\vr);
\draw (v2) [fill=black] circle (\vr);
\draw (y2) [fill=black] circle (\vr);
\draw (w2) [fill=white] circle (\vr);
\draw (z2) [fill=white] circle (\vr);
\draw (x1n) [fill=white] circle (\vr);
\draw (u) [fill=black] circle (\vr);
\draw (v) [fill=white] circle (\vr);
\draw (n) [fill=white] circle (\vr);
\draw (y2n) [fill=black] circle (\vr);
\draw (w2n) [fill=white] circle (\vr);
\draw (z2n) [fill=white] circle (\vr);
%%%%
%%%
\draw (1,-1) node {(a) {\small $G_1$}};
\draw (5,-1) node {(b) {\small $G_2$}};
\draw (9.5,-1) node {(c) {\small $G_1 \cup_{K_2} G_2$}};
\draw[anchor = south] (x1) node {{\small $x_1$}};
\draw[anchor = east] (u1) node {{\small $u_1$}};
\draw[anchor = west] (v1) node {{\small $v_1$}};
\draw[anchor = south] (u2) node {{\small $u_2$}};
\draw[anchor = south] (v2) node {{\small $v_2$}};
\draw[anchor = east] (y2) node {{\small $y_2$}};
\draw[anchor = north] (w2) node {{\small $w_2$}};
\draw[anchor = north] (z2) node {{\small $z_2$}};
\draw[anchor = south] (x1n) node {{\small $x_1$}};
\draw[anchor = east] (u) node {{\small $u$}};
\draw[anchor = west] (v) node {{\small $v$}};
\draw[anchor = east] (y2n) node {{\small $y_2$}};
\draw[anchor = north] (w2n) node {{\small $w_2$}};
\draw[anchor = north] (z2n) node {{\small $z_2$}};
\end{tikzpicture}
\end{center}
\vskip -0.5cm
\caption{A graph achieving equality in Inequality~(\ref{Eq:r})} \label{f:example3}
\end{figure}
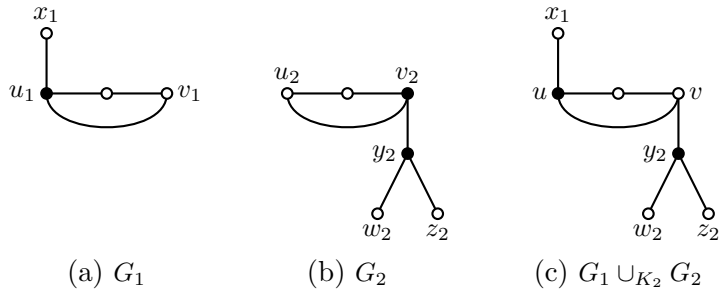

%%%%%%%%%%%%%%%%%%%%%%%%%%%%%%%%%%%%%%%%%%%%%%%%%%%%%%%%%%%%%%%%%%%%%%%%%%%

\section{Chain of Graphs}

In this section, we study the strong domination number of a chain of graphs.

	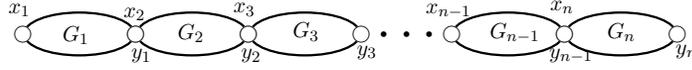
\begin{figure}
		\begin{center}
			\psscalebox{0.75 0.75}
			{
				\begin{pspicture}(0,-3.9483333)(12.236668,-2.8316667)
				\psellipse[linecolor=black, linewidth=0.04, dimen=outer](1.2533334,-3.4416668)(1.0,0.4)
				\psellipse[linecolor=black, linewidth=0.04, dimen=outer](3.2533333,-3.4416668)(1.0,0.4)
				\psellipse[linecolor=black, linewidth=0.04, dimen=outer](5.2533336,-3.4416668)(1.0,0.4)
				\psellipse[linecolor=black, linewidth=0.04, dimen=outer](8.853333,-3.4416668)(1.0,0.4)
				\psellipse[linecolor=black, linewidth=0.04, dimen=outer](10.853333,-3.4416668)(1.0,0.4)
				\psdots[linecolor=black, fillstyle=solid, dotstyle=o, dotsize=0.3, fillcolor=white](2.2533333,-3.4416666)
				\psdots[linecolor=black, fillstyle=solid, dotstyle=o, dotsize=0.3, fillcolor=white](0.25333345,-3.4416666)
				\psdots[linecolor=black, fillstyle=solid, dotstyle=o, dotsize=0.3, fillcolor=white](2.2533333,-3.4416666)
				\psdots[linecolor=black, fillstyle=solid, dotstyle=o, dotsize=0.3, fillcolor=white](4.2533336,-3.4416666)
				\psdots[linecolor=black, fillstyle=solid, dotstyle=o, dotsize=0.3, fillcolor=white](4.2533336,-3.4416666)
				\psdots[linecolor=black, fillstyle=solid, dotstyle=o, dotsize=0.3, fillcolor=white](9.853333,-3.4416666)
				\psdots[linecolor=black, fillstyle=solid, dotstyle=o, dotsize=0.3, fillcolor=white](9.853333,-3.4416666)
				\psdots[linecolor=black, fillstyle=solid, dotstyle=o, dotsize=0.3, fillcolor=white](11.853333,-3.4416666)
				\rput[bl](0.0,-3.135){$x_1$}
				\rput[bl](2.0400002,-3.2016668){$x_2$}
				\rput[bl](3.9866667,-3.1216667){$x_3$}
				\rput[bl](2.1733334,-3.9483335){$y_1$}
				\rput[bl](4.12,-3.9483335){$y_2$}
				\rput[bl](6.1733336,-3.8816667){$y_3$}
				\rput[bl](0.9600001,-3.6283333){$G_1$}
				\rput[bl](3.0,-3.5883334){$G_2$}
				\rput[bl](5.04,-3.5616667){$G_3$}
				\psdots[linecolor=black, fillstyle=solid, dotstyle=o, dotsize=0.3, fillcolor=white](6.2533336,-3.4416666)
				\psdots[linecolor=black, fillstyle=solid, dotstyle=o, dotsize=0.3, fillcolor=white](7.8533335,-3.4416666)
				\psdots[linecolor=black, dotsize=0.1](6.6533337,-3.4416666)
				\psdots[linecolor=black, dotsize=0.1](7.0533333,-3.4416666)
				\psdots[linecolor=black, dotsize=0.1](7.4533334,-3.4416666)
				\rput[bl](9.6,-3.0816667){$x_n$}
				\rput[bl](11.826667,-3.8683333){$y_n$}
				\rput[bl](9.586667,-3.9483335){$y_{n-1}$}
				\rput[bl](8.533334,-3.6016667){$G_{n-1}$}
				\rput[bl](7.4,-3.1616666){$x_{n-1}$}
				\rput[bl](10.613334,-3.575){$G_n$}
				\end{pspicture}
			}
		\end{center}
		\caption{Chain of $n$ graphs $G_1,G_2, \ldots , G_n$} \label{chain-n}
	\end{figure}

	\begin{theorem} \label{thm:chain}
		Let $G_1,G_2, \ldots , G_n$ be a finite sequence of pairwise disjoint connected graphs and let
		$x_i,y_i \in V(G_i)$. If $C(G_1,...,G_n)$ is the chain of graphs $\{G_i\}_{i=1}^n$ with respect to the vertices $\{x_i, y_i\}_{i=1}^n$ obtained by identifying the vertex $y_i$ with the vertex $x_{i+1}$ for $i=1,2,\ldots,n-1$ (Figure \ref{chain-n}), then
		\begin{align*}
		\left( \sum_{i=1}^{n}\gst(G_i)\right) -\left(\sum_{i=2}^{n}\deg(x_i)\right)-\left(\sum_{i=1}^{n-1}\deg(y_i)\right) +n-1 &\leq  \gst(C(G_1,...,G_n)) \\
&\leq   \left( \sum_{i=1}^{n}\gst(G_i)\right) +n-1.
		\end{align*}
	\end{theorem}
\begin{proof}	
First we prove the upper bound.	Suppose that $D_i$ is a $\gst$-set of $G_i$, for $i \in [n]$. Also, suppose that $z_i$ is the identifying of the vertex $y_i$ with the vertex $x_{i+1}$ for $i \in [n-1]$. The set
$$D=\left(\bigcup\limits_{i=1}^{n} D_i \cup \{z_1,z_2,\ldots,z_{n-1}\}\right)\setminus\{x_2, x_3, \ldots x_n, y_1 y_2, \ldots, y_{n-1}\},$$
is a strong dominating set of $C(G_1,\ldots,G_n)$, and we are done. The lower bound is an immediate result of Theorem \ref{thm:1-gluing}, by using induction.
\qed
	\end{proof}

	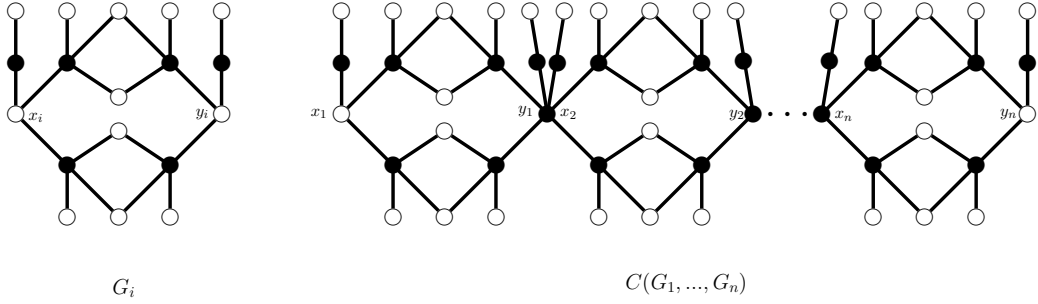
\begin{figure}
		\begin{center}
			\psscalebox{0.57 0.57}
{
\begin{pspicture}(0,-6.439306)(24.002779,0.44208345)
\psline[linecolor=black, linewidth=0.08](0.2013889,-2.1593053)(1.4013889,-0.9593054)(2.601389,0.24069458)(2.601389,0.24069458)
\psline[linecolor=black, linewidth=0.08](2.601389,0.24069458)(3.801389,-0.9593054)(5.001389,-2.1593053)(5.001389,-2.1593053)(5.001389,-2.1593053)
\psline[linecolor=black, linewidth=0.08](5.001389,-2.1593053)(2.601389,-4.559305)(2.601389,-4.559305)
\psline[linecolor=black, linewidth=0.08](2.601389,-4.559305)(0.2013889,-2.1593053)(0.2013889,-2.1593053)
\psline[linecolor=black, linewidth=0.08](1.4013889,-0.9593054)(2.601389,-1.7593055)(2.601389,-1.7593055)
\psline[linecolor=black, linewidth=0.08](2.601389,-1.7593055)(3.801389,-0.9593054)(3.801389,-0.9593054)
\psline[linecolor=black, linewidth=0.08](1.4013889,-3.3593054)(2.601389,-2.5593054)(2.601389,-2.5593054)
\psline[linecolor=black, linewidth=0.08](2.601389,-2.5593054)(3.801389,-3.3593054)(3.801389,-3.3593054)
\psline[linecolor=black, linewidth=0.08](5.001389,-2.1593053)(5.001389,-0.9593054)(5.001389,-0.9593054)
\psline[linecolor=black, linewidth=0.08](5.001389,-0.9593054)(5.001389,0.24069458)(5.001389,0.24069458)
\psline[linecolor=black, linewidth=0.08](0.2013889,-2.1593053)(0.2013889,-0.9593054)(0.2013889,-0.9593054)(0.2013889,-0.9593054)
\psline[linecolor=black, linewidth=0.08](0.2013889,-0.9593054)(0.2013889,0.24069458)(0.2013889,0.24069458)
\psline[linecolor=black, linewidth=0.08](1.4013889,-0.9593054)(1.4013889,0.24069458)(1.4013889,0.24069458)
\psline[linecolor=black, linewidth=0.08](3.801389,-0.9593054)(3.801389,0.24069458)(3.801389,0.24069458)
\psline[linecolor=black, linewidth=0.08](3.801389,-3.3593054)(3.801389,-4.559305)(3.801389,-4.559305)
\psline[linecolor=black, linewidth=0.08](1.4013889,-3.3593054)(1.4013889,-4.559305)(1.4013889,-4.559305)
\psdots[linecolor=black, dotsize=0.4](1.4013889,-0.9593054)
\psdots[linecolor=black, dotsize=0.4](3.801389,-0.9593054)
\psdots[linecolor=black, dotsize=0.4](1.4013889,-3.3593054)
\psdots[linecolor=black, dotsize=0.4](3.801389,-3.3593054)
\psdots[linecolor=black, dotsize=0.4](5.001389,-0.9593054)
\psdots[linecolor=black, dotsize=0.4](0.2013889,-0.9593054)
\psdots[linecolor=black, dotstyle=o, dotsize=0.4, fillcolor=white](0.2013889,0.24069458)
\psdots[linecolor=black, dotstyle=o, dotsize=0.4, fillcolor=white](1.4013889,0.24069458)
\psdots[linecolor=black, dotstyle=o, dotsize=0.4, fillcolor=white](3.801389,0.24069458)
\psdots[linecolor=black, dotstyle=o, dotsize=0.4, fillcolor=white](5.001389,0.24069458)
\psdots[linecolor=black, dotstyle=o, dotsize=0.4, fillcolor=white](2.601389,0.24069458)
\psdots[linecolor=black, dotstyle=o, dotsize=0.4, fillcolor=white](2.601389,-1.7593055)
\psdots[linecolor=black, dotstyle=o, dotsize=0.4, fillcolor=white](2.601389,-2.5593054)
\psdots[linecolor=black, dotstyle=o, dotsize=0.4, fillcolor=white](0.2013889,-2.1593053)
\psdots[linecolor=black, dotstyle=o, dotsize=0.4, fillcolor=white](5.001389,-2.1593053)
\psdots[linecolor=black, dotstyle=o, dotsize=0.4, fillcolor=white](2.601389,-4.559305)
\psdots[linecolor=black, dotstyle=o, dotsize=0.4, fillcolor=white](1.4013889,-4.559305)
\psdots[linecolor=black, dotstyle=o, dotsize=0.4, fillcolor=white](3.801389,-4.559305)
\rput[bl](2.4613888,-6.4393053){\Large{$G_i$}}
\rput[bl](14.421389,-6.3793054){\Large{$C(G_1,...,G_n)$}}
\rput[bl](0.5013889,-2.3593054){$x_i$}
\rput[bl](4.401389,-2.2793055){$y_i$}
\rput[bl](7.081389,-2.2593055){$x_1$}
\rput[bl](11.921389,-2.2193055){$y_1$}
\rput[bl](12.901389,-2.2793055){$x_2$}
\rput[bl](16.841389,-2.2993054){$y_2$}
\rput[bl](19.30139,-2.3193054){$x_n$}
\rput[bl](23.161388,-2.2793055){$y_n$}
\psline[linecolor=black, linewidth=0.08](7.8013887,-2.1593053)(9.001389,-0.9593054)(10.201389,0.24069458)(10.201389,0.24069458)
\psline[linecolor=black, linewidth=0.08](10.201389,0.24069458)(11.401389,-0.9593054)(12.601389,-2.1593053)(12.601389,-2.1593053)(12.601389,-2.1593053)
\psline[linecolor=black, linewidth=0.08](12.601389,-2.1593053)(10.201389,-4.559305)(10.201389,-4.559305)
\psline[linecolor=black, linewidth=0.08](10.201389,-4.559305)(7.8013887,-2.1593053)(7.8013887,-2.1593053)
\psline[linecolor=black, linewidth=0.08](9.001389,-0.9593054)(10.201389,-1.7593055)(10.201389,-1.7593055)
\psline[linecolor=black, linewidth=0.08](10.201389,-1.7593055)(11.401389,-0.9593054)(11.401389,-0.9593054)
\psline[linecolor=black, linewidth=0.08](9.001389,-3.3593054)(10.201389,-2.5593054)(10.201389,-2.5593054)
\psline[linecolor=black, linewidth=0.08](10.201389,-2.5593054)(11.401389,-3.3593054)(11.401389,-3.3593054)
\psline[linecolor=black, linewidth=0.08](9.001389,-0.9593054)(9.001389,0.24069458)(9.001389,0.24069458)
\psline[linecolor=black, linewidth=0.08](11.401389,-0.9593054)(11.401389,0.24069458)(11.401389,0.24069458)
\psline[linecolor=black, linewidth=0.08](11.401389,-3.3593054)(11.401389,-4.559305)(11.401389,-4.559305)
\psline[linecolor=black, linewidth=0.08](9.001389,-3.3593054)(9.001389,-4.559305)(9.001389,-4.559305)
\psdots[linecolor=black, dotsize=0.4](9.001389,-0.9593054)
\psdots[linecolor=black, dotsize=0.4](11.401389,-0.9593054)
\psdots[linecolor=black, dotsize=0.4](9.001389,-3.3593054)
\psdots[linecolor=black, dotsize=0.4](11.401389,-3.3593054)
\psdots[linecolor=black, dotstyle=o, dotsize=0.4, fillcolor=white](9.001389,0.24069458)
\psdots[linecolor=black, dotstyle=o, dotsize=0.4, fillcolor=white](11.401389,0.24069458)
\psdots[linecolor=black, dotstyle=o, dotsize=0.4, fillcolor=white](10.201389,0.24069458)
\psdots[linecolor=black, dotstyle=o, dotsize=0.4, fillcolor=white](10.201389,-1.7593055)
\psdots[linecolor=black, dotstyle=o, dotsize=0.4, fillcolor=white](10.201389,-2.5593054)
\psdots[linecolor=black, dotstyle=o, dotsize=0.4, fillcolor=white](10.201389,-4.559305)
\psdots[linecolor=black, dotstyle=o, dotsize=0.4, fillcolor=white](9.001389,-4.559305)
\psdots[linecolor=black, dotstyle=o, dotsize=0.4, fillcolor=white](11.401389,-4.559305)
\psline[linecolor=black, linewidth=0.08](12.601389,-2.1593053)(13.801389,-0.9593054)(15.001389,0.24069458)(15.001389,0.24069458)
\psline[linecolor=black, linewidth=0.08](15.001389,0.24069458)(16.20139,-0.9593054)(17.401388,-2.1593053)(17.401388,-2.1593053)(17.401388,-2.1593053)
\psline[linecolor=black, linewidth=0.08](17.401388,-2.1593053)(15.001389,-4.559305)(15.001389,-4.559305)
\psline[linecolor=black, linewidth=0.08](15.001389,-4.559305)(12.601389,-2.1593053)(12.601389,-2.1593053)
\psline[linecolor=black, linewidth=0.08](13.801389,-0.9593054)(15.001389,-1.7593055)(15.001389,-1.7593055)
\psline[linecolor=black, linewidth=0.08](15.001389,-1.7593055)(16.20139,-0.9593054)(16.20139,-0.9593054)
\psline[linecolor=black, linewidth=0.08](13.801389,-3.3593054)(15.001389,-2.5593054)(15.001389,-2.5593054)
\psline[linecolor=black, linewidth=0.08](15.001389,-2.5593054)(16.20139,-3.3593054)(16.20139,-3.3593054)
\psline[linecolor=black, linewidth=0.08](13.801389,-0.9593054)(13.801389,0.24069458)(13.801389,0.24069458)
\psline[linecolor=black, linewidth=0.08](16.20139,-0.9593054)(16.20139,0.24069458)(16.20139,0.24069458)
\psline[linecolor=black, linewidth=0.08](16.20139,-3.3593054)(16.20139,-4.559305)(16.20139,-4.559305)
\psline[linecolor=black, linewidth=0.08](13.801389,-3.3593054)(13.801389,-4.559305)(13.801389,-4.559305)
\psdots[linecolor=black, dotsize=0.4](13.801389,-0.9593054)
\psdots[linecolor=black, dotsize=0.4](16.20139,-0.9593054)
\psdots[linecolor=black, dotsize=0.4](13.801389,-3.3593054)
\psdots[linecolor=black, dotsize=0.4](16.20139,-3.3593054)
\psdots[linecolor=black, dotstyle=o, dotsize=0.4, fillcolor=white](13.801389,0.24069458)
\psdots[linecolor=black, dotstyle=o, dotsize=0.4, fillcolor=white](16.20139,0.24069458)
\psdots[linecolor=black, dotstyle=o, dotsize=0.4, fillcolor=white](15.001389,0.24069458)
\psdots[linecolor=black, dotstyle=o, dotsize=0.4, fillcolor=white](15.001389,-1.7593055)
\psdots[linecolor=black, dotstyle=o, dotsize=0.4, fillcolor=white](15.001389,-2.5593054)
\psdots[linecolor=black, dotstyle=o, dotsize=0.4, fillcolor=white](15.001389,-4.559305)
\psdots[linecolor=black, dotstyle=o, dotsize=0.4, fillcolor=white](13.801389,-4.559305)
\psdots[linecolor=black, dotstyle=o, dotsize=0.4, fillcolor=white](16.20139,-4.559305)
\psline[linecolor=black, linewidth=0.08](19.001389,-2.1593053)(20.20139,-0.9593054)(21.401388,0.24069458)(21.401388,0.24069458)
\psline[linecolor=black, linewidth=0.08](21.401388,0.24069458)(22.601389,-0.9593054)(23.80139,-2.1593053)(23.80139,-2.1593053)(23.80139,-2.1593053)
\psline[linecolor=black, linewidth=0.08](23.80139,-2.1593053)(21.401388,-4.559305)(21.401388,-4.559305)
\psline[linecolor=black, linewidth=0.08](21.401388,-4.559305)(19.001389,-2.1593053)(19.001389,-2.1593053)
\psline[linecolor=black, linewidth=0.08](20.20139,-0.9593054)(21.401388,-1.7593055)(21.401388,-1.7593055)
\psline[linecolor=black, linewidth=0.08](21.401388,-1.7593055)(22.601389,-0.9593054)(22.601389,-0.9593054)
\psline[linecolor=black, linewidth=0.08](20.20139,-3.3593054)(21.401388,-2.5593054)(21.401388,-2.5593054)
\psline[linecolor=black, linewidth=0.08](21.401388,-2.5593054)(22.601389,-3.3593054)(22.601389,-3.3593054)
\psline[linecolor=black, linewidth=0.08](20.20139,-0.9593054)(20.20139,0.24069458)(20.20139,0.24069458)
\psline[linecolor=black, linewidth=0.08](22.601389,-0.9593054)(22.601389,0.24069458)(22.601389,0.24069458)
\psline[linecolor=black, linewidth=0.08](22.601389,-3.3593054)(22.601389,-4.559305)(22.601389,-4.559305)
\psline[linecolor=black, linewidth=0.08](20.20139,-3.3593054)(20.20139,-4.559305)(20.20139,-4.559305)
\psdots[linecolor=black, dotsize=0.4](20.20139,-0.9593054)
\psdots[linecolor=black, dotsize=0.4](22.601389,-0.9593054)
\psdots[linecolor=black, dotsize=0.4](20.20139,-3.3593054)
\psdots[linecolor=black, dotsize=0.4](22.601389,-3.3593054)
\psdots[linecolor=black, dotstyle=o, dotsize=0.4, fillcolor=white](20.20139,0.24069458)
\psdots[linecolor=black, dotstyle=o, dotsize=0.4, fillcolor=white](22.601389,0.24069458)
\psdots[linecolor=black, dotstyle=o, dotsize=0.4, fillcolor=white](21.401388,0.24069458)
\psdots[linecolor=black, dotstyle=o, dotsize=0.4, fillcolor=white](21.401388,-1.7593055)
\psdots[linecolor=black, dotstyle=o, dotsize=0.4, fillcolor=white](21.401388,-2.5593054)
\psdots[linecolor=black, dotstyle=o, dotsize=0.4, fillcolor=white](21.401388,-4.559305)
\psdots[linecolor=black, dotstyle=o, dotsize=0.4, fillcolor=white](20.20139,-4.559305)
\psdots[linecolor=black, dotstyle=o, dotsize=0.4, fillcolor=white](22.601389,-4.559305)
\psdots[linecolor=black, dotsize=0.1](17.80139,-2.1593053)
\psdots[linecolor=black, dotsize=0.1](18.20139,-2.1593053)
\psdots[linecolor=black, dotsize=0.1](18.601389,-2.1593053)
\psline[linecolor=black, linewidth=0.08](7.8013887,-2.1593053)(7.8013887,0.24069458)(7.8013887,0.24069458)
\psline[linecolor=black, linewidth=0.08](12.601389,-2.1593053)(12.201389,0.24069458)(12.201389,0.24069458)
\psline[linecolor=black, linewidth=0.08](12.601389,-2.1593053)(13.001389,0.24069458)(13.001389,0.24069458)
\psline[linecolor=black, linewidth=0.08](17.401388,-2.1593053)(17.001389,0.24069458)(17.001389,0.24069458)
\psline[linecolor=black, linewidth=0.08](19.001389,-2.1593053)(19.401388,0.24069458)(19.401388,0.24069458)
\psline[linecolor=black, linewidth=0.08](23.80139,-2.1593053)(23.80139,0.24069458)(23.80139,0.24069458)
\psdots[linecolor=black, dotsize=0.4](7.8013887,-0.9593054)
\psdots[linecolor=black, dotsize=0.4](23.80139,-0.9593054)
\psdots[linecolor=black, dotstyle=o, dotsize=0.4, fillcolor=white](23.80139,0.24069458)
\psdots[linecolor=black, dotstyle=o, dotsize=0.4, fillcolor=white](19.401388,0.24069458)
\psdots[linecolor=black, dotstyle=o, dotsize=0.4, fillcolor=white](17.001389,0.24069458)
\psdots[linecolor=black, dotstyle=o, dotsize=0.4, fillcolor=white](13.001389,0.24069458)
\psdots[linecolor=black, dotstyle=o, dotsize=0.4, fillcolor=white](12.201389,0.24069458)
\psdots[linecolor=black, dotstyle=o, dotsize=0.4, fillcolor=white](7.8013887,0.24069458)
\psdots[linecolor=black, dotsize=0.4](12.601389,-2.1593053)
\psdots[linecolor=black, dotsize=0.4](17.401388,-2.1593053)
\psdots[linecolor=black, dotsize=0.4](19.001389,-2.1593053)
\psdots[linecolor=black, dotstyle=o, dotsize=0.4, fillcolor=white](23.80139,-2.1593053)
\psdots[linecolor=black, dotstyle=o, dotsize=0.4, fillcolor=white](7.8013887,-2.1593053)
\psdots[linecolor=black, dotsize=0.4](12.361389,-0.9393054)
\psdots[linecolor=black, dotsize=0.4](12.841389,-0.9593054)
\psdots[linecolor=black, dotsize=0.4](17.181389,-0.91930544)
\psdots[linecolor=black, dotsize=0.4](19.181389,-0.91930544)
\end{pspicture}
}
		\end{center}
		\caption{Graphs $G_i$, for $i=1,2,\ldots,n$, and $C(G_1,...,G_n)$, respectively.} \label{fig:chain-upper}
	\end{figure}

\begin{remark}
{\rm We remark that the bounds in Theorem \ref{thm:chain} are tight. For the upper bound, it suffices to consider graphs $G_i$, for $i \in [n]$, and their chain with respect to the vertices $\{x_i, y_i\}_{i=1}^n$, as we see in Figure~\ref{fig:chain-upper}. The set of black vertices in each graph is a $\gst$-set. We can generalize this idea by adding vertices which are adjacent to $x_i$ and $y_i$,  and their neighbours, yielding an infinite family of graphs such that equality holds in this upper bounds.  For the lower bound, consider Figure~\ref{fig:chain-lower}. By the same argument as the upper bound, the set of black vertices in each graph is a $\gst$-set. This idea can be generalized, and therefore we have an infinite family of graphs such that the equality of the lower bound holds.}
	\end{remark}

	\begin{figure}
		\begin{center}
			\psscalebox{0.52 0.52}
{
\begin{pspicture}(0,-7.8293056)(26.402779,-2.1679165)
\psline[linecolor=black, linewidth=0.08](1.4013889,-3.1693053)(2.2013888,-2.3693054)(3.0013888,-2.3693054)(3.8013887,-2.3693054)(4.601389,-3.1693053)(3.8013887,-3.9693055)(2.2013888,-3.9693055)(1.4013889,-3.1693053)(1.4013889,-3.1693053)
\psline[linecolor=black, linewidth=0.08](1.4013889,-5.5693054)(2.2013888,-4.769305)(3.0013888,-4.769305)(3.8013887,-4.769305)(4.601389,-5.5693054)(3.8013887,-6.3693056)(2.2013888,-6.3693056)(1.4013889,-5.5693054)(1.4013889,-5.5693054)
\psline[linecolor=black, linewidth=0.08](4.601389,-3.1693053)(5.8013887,-4.3693056)(4.601389,-5.5693054)(4.601389,-5.5693054)
\psline[linecolor=black, linewidth=0.08](1.4013889,-3.1693053)(0.20138885,-4.3693056)(1.4013889,-5.5693054)
\psdots[linecolor=black, dotsize=0.4](3.0013888,-2.3693054)
\psdots[linecolor=black, dotsize=0.4](3.0013888,-3.9693055)
\psdots[linecolor=black, dotsize=0.4](3.0013888,-4.769305)
\psdots[linecolor=black, dotsize=0.4](3.0013888,-6.3693056)
\psdots[linecolor=black, dotstyle=o, dotsize=0.4, fillcolor=white](2.2013888,-2.3693054)
\psdots[linecolor=black, dotstyle=o, dotsize=0.4, fillcolor=white](3.8013887,-2.3693054)
\psdots[linecolor=black, dotstyle=o, dotsize=0.4, fillcolor=white](3.8013887,-3.9693055)
\psdots[linecolor=black, dotstyle=o, dotsize=0.4, fillcolor=white](2.2013888,-3.9693055)
\psdots[linecolor=black, dotstyle=o, dotsize=0.4, fillcolor=white](2.2013888,-4.769305)
\psdots[linecolor=black, dotstyle=o, dotsize=0.4, fillcolor=white](3.8013887,-4.769305)
\psdots[linecolor=black, dotstyle=o, dotsize=0.4, fillcolor=white](3.8013887,-6.3693056)
\psdots[linecolor=black, dotstyle=o, dotsize=0.4, fillcolor=white](2.2013888,-6.3693056)
\psline[linecolor=black, linewidth=0.08](9.001389,-3.1693053)(9.801389,-2.3693054)(10.601389,-2.3693054)(11.401389,-2.3693054)(12.201389,-3.1693053)(11.401389,-3.9693055)(9.801389,-3.9693055)(9.001389,-3.1693053)(9.001389,-3.1693053)
\psline[linecolor=black, linewidth=0.08](9.001389,-5.5693054)(9.801389,-4.769305)(10.601389,-4.769305)(11.401389,-4.769305)(12.201389,-5.5693054)(11.401389,-6.3693056)(9.801389,-6.3693056)(9.001389,-5.5693054)(9.001389,-5.5693054)
\psline[linecolor=black, linewidth=0.08](12.201389,-3.1693053)(13.401389,-4.3693056)(12.201389,-5.5693054)(12.201389,-5.5693054)
\psline[linecolor=black, linewidth=0.08](9.001389,-3.1693053)(7.8013887,-4.3693056)(9.001389,-5.5693054)
\psdots[linecolor=black, dotsize=0.4](10.601389,-2.3693054)
\psdots[linecolor=black, dotsize=0.4](10.601389,-3.9693055)
\psdots[linecolor=black, dotsize=0.4](10.601389,-4.769305)
\psdots[linecolor=black, dotsize=0.4](10.601389,-6.3693056)
\psdots[linecolor=black, dotstyle=o, dotsize=0.4, fillcolor=white](9.801389,-2.3693054)
\psdots[linecolor=black, dotstyle=o, dotsize=0.4, fillcolor=white](11.401389,-2.3693054)
\psdots[linecolor=black, dotstyle=o, dotsize=0.4, fillcolor=white](11.401389,-3.9693055)
\psdots[linecolor=black, dotstyle=o, dotsize=0.4, fillcolor=white](9.801389,-3.9693055)
\psdots[linecolor=black, dotstyle=o, dotsize=0.4, fillcolor=white](9.801389,-4.769305)
\psdots[linecolor=black, dotstyle=o, dotsize=0.4, fillcolor=white](11.401389,-4.769305)
\psdots[linecolor=black, dotstyle=o, dotsize=0.4, fillcolor=white](11.401389,-6.3693056)
\psdots[linecolor=black, dotstyle=o, dotsize=0.4, fillcolor=white](9.801389,-6.3693056)
\psline[linecolor=black, linewidth=0.08](14.601389,-3.1693053)(15.401389,-2.3693054)(16.20139,-2.3693054)(17.001389,-2.3693054)(17.80139,-3.1693053)(17.001389,-3.9693055)(15.401389,-3.9693055)(14.601389,-3.1693053)(14.601389,-3.1693053)
\psline[linecolor=black, linewidth=0.08](14.601389,-5.5693054)(15.401389,-4.769305)(16.20139,-4.769305)(17.001389,-4.769305)(17.80139,-5.5693054)(17.001389,-6.3693056)(15.401389,-6.3693056)(14.601389,-5.5693054)(14.601389,-5.5693054)
\psline[linecolor=black, linewidth=0.08](17.80139,-3.1693053)(19.001389,-4.3693056)(17.80139,-5.5693054)(17.80139,-5.5693054)
\psline[linecolor=black, linewidth=0.08](14.601389,-3.1693053)(13.401389,-4.3693056)(14.601389,-5.5693054)
\psdots[linecolor=black, dotsize=0.4](16.20139,-2.3693054)
\psdots[linecolor=black, dotsize=0.4](16.20139,-3.9693055)
\psdots[linecolor=black, dotsize=0.4](16.20139,-4.769305)
\psdots[linecolor=black, dotsize=0.4](16.20139,-6.3693056)
\psdots[linecolor=black, dotstyle=o, dotsize=0.4, fillcolor=white](15.401389,-2.3693054)
\psdots[linecolor=black, dotstyle=o, dotsize=0.4, fillcolor=white](17.001389,-2.3693054)
\psdots[linecolor=black, dotstyle=o, dotsize=0.4, fillcolor=white](17.001389,-3.9693055)
\psdots[linecolor=black, dotstyle=o, dotsize=0.4, fillcolor=white](15.401389,-3.9693055)
\psdots[linecolor=black, dotstyle=o, dotsize=0.4, fillcolor=white](15.401389,-4.769305)
\psdots[linecolor=black, dotstyle=o, dotsize=0.4, fillcolor=white](17.001389,-4.769305)
\psdots[linecolor=black, dotstyle=o, dotsize=0.4, fillcolor=white](17.001389,-6.3693056)
\psdots[linecolor=black, dotstyle=o, dotsize=0.4, fillcolor=white](15.401389,-6.3693056)
\psline[linecolor=black, linewidth=0.08](21.80139,-3.1693053)(22.601389,-2.3693054)(23.401388,-2.3693054)(24.20139,-2.3693054)(25.001389,-3.1693053)(24.20139,-3.9693055)(22.601389,-3.9693055)(21.80139,-3.1693053)(21.80139,-3.1693053)
\psline[linecolor=black, linewidth=0.08](21.80139,-5.5693054)(22.601389,-4.769305)(23.401388,-4.769305)(24.20139,-4.769305)(25.001389,-5.5693054)(24.20139,-6.3693056)(22.601389,-6.3693056)(21.80139,-5.5693054)(21.80139,-5.5693054)
\psline[linecolor=black, linewidth=0.08](25.001389,-3.1693053)(26.20139,-4.3693056)(25.001389,-5.5693054)(25.001389,-5.5693054)
\psline[linecolor=black, linewidth=0.08](21.80139,-3.1693053)(20.601389,-4.3693056)(21.80139,-5.5693054)
\psdots[linecolor=black, dotsize=0.4](23.401388,-2.3693054)
\psdots[linecolor=black, dotsize=0.4](23.401388,-3.9693055)
\psdots[linecolor=black, dotsize=0.4](23.401388,-4.769305)
\psdots[linecolor=black, dotsize=0.4](23.401388,-6.3693056)
\psdots[linecolor=black, dotstyle=o, dotsize=0.4, fillcolor=white](22.601389,-2.3693054)
\psdots[linecolor=black, dotstyle=o, dotsize=0.4, fillcolor=white](24.20139,-2.3693054)
\psdots[linecolor=black, dotstyle=o, dotsize=0.4, fillcolor=white](24.20139,-3.9693055)
\psdots[linecolor=black, dotstyle=o, dotsize=0.4, fillcolor=white](22.601389,-3.9693055)
\psdots[linecolor=black, dotstyle=o, dotsize=0.4, fillcolor=white](22.601389,-4.769305)
\psdots[linecolor=black, dotstyle=o, dotsize=0.4, fillcolor=white](24.20139,-4.769305)
\psdots[linecolor=black, dotstyle=o, dotsize=0.4, fillcolor=white](24.20139,-6.3693056)
\psdots[linecolor=black, dotstyle=o, dotsize=0.4, fillcolor=white](22.601389,-6.3693056)
\psdots[linecolor=black, dotsize=0.1](19.401388,-4.3693056)
\psdots[linecolor=black, dotsize=0.1](19.80139,-4.3693056)
\psdots[linecolor=black, dotsize=0.1](20.20139,-4.3693056)
\psdots[linecolor=black, dotsize=0.4](20.601389,-4.3693056)
\psdots[linecolor=black, dotsize=0.4](19.001389,-4.3693056)
\psdots[linecolor=black, dotsize=0.4](13.401389,-4.3693056)
\psdots[linecolor=black, dotstyle=o, dotsize=0.4, fillcolor=white](26.20139,-4.3693056)
\psdots[linecolor=black, dotstyle=o, dotsize=0.4, fillcolor=white](7.8013887,-4.3693056)
\psdots[linecolor=black, dotstyle=o, dotsize=0.4, fillcolor=white](12.201389,-3.1693053)
\psdots[linecolor=black, dotstyle=o, dotsize=0.4, fillcolor=white](12.201389,-5.5693054)
\psdots[linecolor=black, dotstyle=o, dotsize=0.4, fillcolor=white](14.601389,-3.1693053)
\psdots[linecolor=black, dotstyle=o, dotsize=0.4, fillcolor=white](14.601389,-5.5693054)
\psdots[linecolor=black, dotstyle=o, dotsize=0.4, fillcolor=white](21.80139,-3.1693053)
\psdots[linecolor=black, dotstyle=o, dotsize=0.4, fillcolor=white](21.80139,-5.5693054)
\psdots[linecolor=black, dotsize=0.4](25.001389,-3.1693053)
\psdots[linecolor=black, dotsize=0.4](25.001389,-5.5693054)
\psdots[linecolor=black, dotsize=0.4](9.001389,-3.1693053)
\psdots[linecolor=black, dotsize=0.4](9.001389,-5.5693054)
\psdots[linecolor=black, dotstyle=o, dotsize=0.4, fillcolor=white](17.80139,-3.1693053)
\psdots[linecolor=black, dotstyle=o, dotsize=0.4, fillcolor=white](17.80139,-5.5693054)
\psdots[linecolor=black, dotstyle=o, dotsize=0.4, fillcolor=white](5.8013887,-4.3693056)
\psdots[linecolor=black, dotstyle=o, dotsize=0.4, fillcolor=white](0.20138885,-4.3693056)
\psdots[linecolor=black, dotsize=0.4](4.601389,-3.1693053)
\psdots[linecolor=black, dotsize=0.4](1.4013889,-3.1693053)
\psdots[linecolor=black, dotsize=0.4](1.4013889,-5.5693054)
\psdots[linecolor=black, dotsize=0.4](4.601389,-5.5693054)
\rput[bl](2.7613888,-7.8293056){\Large{$G_i$}}
\rput[bl](16.001389,-7.7093053){\Large{$C(G_1,...,G_n)$}}
\rput[bl](0.5213888,-4.4893055){$x_i$}
\rput[bl](5.161389,-4.5093055){$y_i$}
\rput[bl](8.081388,-4.5293055){$x_1$}
\rput[bl](12.821389,-4.5293055){$y_1$}
\rput[bl](13.661388,-4.4893055){$x_2$}
\rput[bl](18.38139,-4.5093055){$y_2$}
\rput[bl](20.86139,-4.5493054){$x_n$}
\rput[bl](25.62139,-4.4693055){$y_n$}
\end{pspicture}
}
		\end{center}
		\caption{Graphs $G_i$, for $i=1,2,\ldots,n$, and $C(G_1,...,G_n)$, respectively.} \label{fig:chain-lower}
	\end{figure}
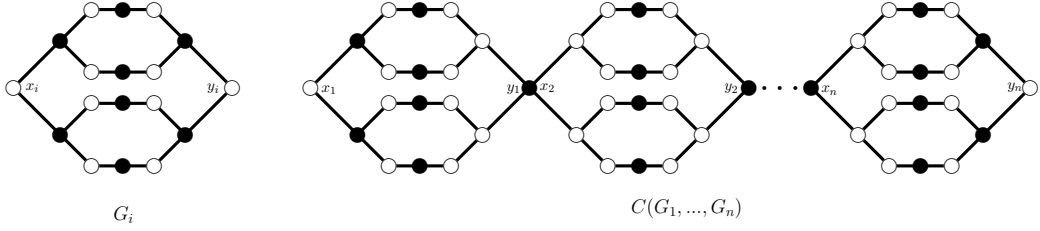

%%%%%%%%%%%%%%%%%%%%%%%%%%%%%%%%%%%%%%%%%%%%%%%%%%%%%%%%%%%%%%%%%%%%%%%%%	

\section{Link of Graphs}

In this section, we study the strong domination number of a link of graphs. The following result is proved in~\cite{AlGhZa-22+}.

\begin{theorem}{\rm (\cite{AlGhZa-22+})} 
 \label{edge-deletion}\cite{AlGhZa-22+}
If $G=(V,E)$ is a graph which is not $K_2$ and $e=uv\in E$, then
 	$$\gamma_{st} (G)-1\leq \gamma_{st}(G-e)\leq \gamma_{st} (G)+\deg (u)+\deg (v)-2.$$	
 \end{theorem}

	\begin{figure}
		\begin{center}
			\psscalebox{0.6 0.6}
{
\begin{pspicture}(0,-5.2)(9.2,-2.8)
\psellipse[linecolor=black, linewidth=0.08, dimen=outer](2.0,-4.0)(2.0,1.2)
\psellipse[linecolor=black, linewidth=0.08, dimen=outer](7.2,-4.0)(2.0,1.2)
\psline[linecolor=black, linewidth=0.08](4.0,-4.0)(5.2,-4.0)(5.2,-4.0)
\psdots[linecolor=black, dotstyle=o, dotsize=0.4, fillcolor=white](4.0,-4.0)
\psdots[linecolor=black, dotstyle=o, dotsize=0.4, fillcolor=white](5.2,-4.0)
\rput[bl](3.94,-4.54){$u$}
\rput[bl](4.92,-4.58){$v$}
\rput[bl](1.8,-4.2){\Large{$G_1$}}
\rput[bl](6.72,-4.2){\Large{$G_2$}}
\end{pspicture}
}
		\end{center}
	\caption{Graph $G$.} \label{fig:bridge}
	\end{figure}

As an immediate result of Theorems \ref{thm:1-gluing} and \ref{edge-deletion}, we have:

\begin{theorem} \label{thm:bridge}
If $G$ is  a graph and $uv\in E(G)$ is a bridge, as shown in Figure \ref{fig:bridge}, then
	$$\gst(G)\geq \gst(G_1)+\gst(G_2)-\deg(u)-\deg(v)+2.$$
	\end{theorem}

	\begin{figure}
		\begin{center}
			\psscalebox{0.8 0.8}
			{
				\begin{pspicture}(0,-4.08)(15.436667,-3.12)
				\psellipse[linecolor=black, linewidth=0.04, dimen=outer](1.2533334,-3.52)(1.0,0.4)
				\psellipse[linecolor=black, linewidth=0.04, dimen=outer](4.0533333,-3.52)(1.0,0.4)
				\psellipse[linecolor=black, linewidth=0.04, dimen=outer](6.853334,-3.52)(1.0,0.4)
				\psellipse[linecolor=black, linewidth=0.04, dimen=outer](11.253334,-3.52)(1.0,0.4)
				\psellipse[linecolor=black, linewidth=0.04, dimen=outer](14.053333,-3.52)(1.0,0.4)
				\psline[linecolor=black, linewidth=0.04](2.2533333,-3.52)(3.0533335,-3.52)(3.0533335,-3.52)
				\psline[linecolor=black, linewidth=0.04](5.0533333,-3.52)(5.8533335,-3.52)(5.8533335,-3.52)
				\psline[linecolor=black, linewidth=0.04](12.253333,-3.52)(13.053333,-3.52)(13.053333,-3.52)
				\psdots[linecolor=black, dotstyle=o, dotsize=0.3, fillcolor=white](2.2533333,-3.52)
				\psdots[linecolor=black, dotstyle=o, dotsize=0.3, fillcolor=white](0.25333345,-3.52)
				\psdots[linecolor=black, dotstyle=o, dotsize=0.3, fillcolor=white](3.0533335,-3.52)
				\psdots[linecolor=black, dotstyle=o, dotsize=0.3, fillcolor=white](5.0533333,-3.52)
				\psdots[linecolor=black, dotstyle=o, dotsize=0.3, fillcolor=white](5.8533335,-3.52)
				\psdots[linecolor=black, dotstyle=o, dotsize=0.3, fillcolor=white](12.253333,-3.52)
				\psdots[linecolor=black, dotstyle=o, dotsize=0.3, fillcolor=white](13.053333,-3.52)
				\psdots[linecolor=black, dotstyle=o, dotsize=0.3, fillcolor=white](15.053333,-3.52)
				\rput[bl](0.0,-4.0133333){$x_1$}
				\rput[bl](2.8400002,-4.08){$x_2$}
				\rput[bl](5.5866666,-4.0){$x_3$}
				\rput[bl](2.1733334,-4.0266666){$y_1$}
				\rput[bl](4.92,-4.0266666){$y_2$}
				\rput[bl](7.7733335,-3.96){$y_3$}
				\rput[bl](0.9600001,-3.7066667){$G_1$}
				\rput[bl](3.8000002,-3.6666667){$G_2$}
				\rput[bl](6.64,-3.64){$G_3$}
				\psline[linecolor=black, linewidth=0.04](8.253333,-3.52)(7.8533335,-3.52)(7.8533335,-3.52)
				\psline[linecolor=black, linewidth=0.04](9.853333,-3.52)(10.253333,-3.52)(10.253333,-3.52)
				\psdots[linecolor=black, dotstyle=o, dotsize=0.3, fillcolor=white](7.8533335,-3.52)
				\psdots[linecolor=black, dotstyle=o, dotsize=0.3, fillcolor=white](10.253333,-3.52)
				\psdots[linecolor=black, dotsize=0.1](8.653334,-3.52)
				\psdots[linecolor=black, dotsize=0.1](9.053333,-3.52)
				\psdots[linecolor=black, dotsize=0.1](9.453334,-3.52)
				\rput[bl](12.8,-3.96){$x_n$}
				\rput[bl](15.026667,-3.9466667){$y_n$}
				\rput[bl](11.986667,-4.0266666){$y_{n-1}$}
				\rput[bl](10.933333,-3.68){$G_{n-1}$}
				\rput[bl](9.8,-4.04){$x_{n-1}$}
				\rput[bl](13.8133335,-3.6533334){$G_n$}
				\end{pspicture}
			}
		\end{center}
	\caption{Link of $n$ graphs $G_1,G_2, \ldots , G_n$.} \label{link-n}
	\end{figure}
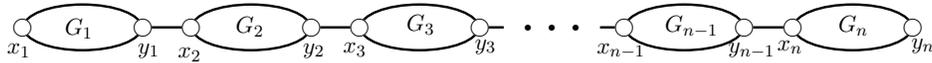

\begin{theorem} \label{thm:link}
	Let $G_1,G_2, \ldots , G_n$ be a finite sequence of pairwise disjoint connected graphs and let
	$x_i,y_i \in V(G_i)$. If $L(G_1,...,G_n)$ is the link of graphs $\{G_i\}_{i=1}^n$ with respect to the vertices $\{x_i, y_i\}_{i=1}^n$ (see Figure \ref{link-n}), then 
		\begin{align*}
		\left( \sum_{i=1}^{n}\gst(G_i)\right) -\left(\sum_{i=2}^{n}\deg(x_i)\right)-\left(\sum_{i=1}^{n-1}\deg(y_i)\right) +2n-2 &\leq  \gst(L(G_1,...,G_n)) \\
&\leq   \left( \sum_{i=1}^{n}\gst(G_i)\right) +n-1.
		\end{align*}
	\end{theorem}
\begin{proof}	
First we prove the upper bound.	Suppose that $D_i$ is a $\gst$-sets of $G_i$, for $i \in [n]$. Also, for $j\in [n-1]$, let
	\[
	z_j=\left\{
	\begin{array}{ll}
	{\displaystyle
		x_{j+1}}&
	\quad\mbox{if $\deg(x_{j+1})\geq \deg(y_{j})  $,}\\[15pt]
	{\displaystyle
		y_{j}}&
	\quad\mbox{otherwise.}
	\end{array}
	\right.
	\]
In this case, the set 
$$D=\bigcup\limits_{i=1}^{n} D_i \cup \{z_1,z_2,\ldots,z_{n-1}\},$$
is a strong dominating set of $C(G_1,...,G_n)$, because each vertex in $\overline{D}$ is strongly dominated by the same vertex as before or possibly with $z_j$, for $j \in [n-1]$, and we are done. The lower bound is an immediate result of Theorem \ref{thm:bridge}, and by using induction.
\qed
	\end{proof}

	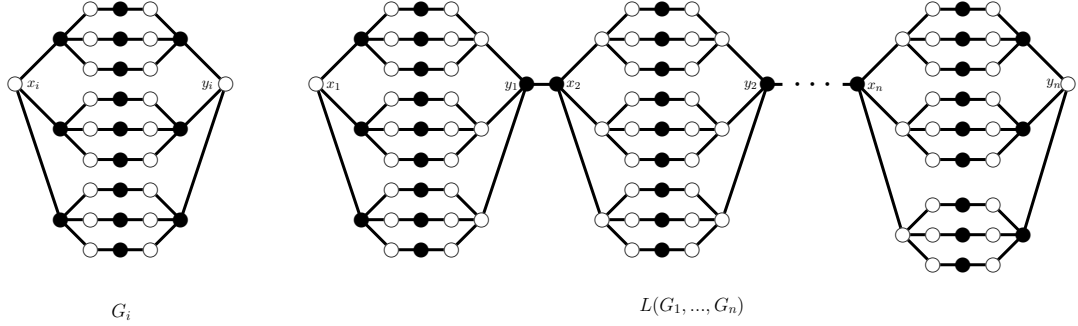
\begin{figure}
		\begin{center}
			\psscalebox{0.50 0.50}
{
\begin{pspicture}(0,-9.229305)(28.402779,-0.76791656)
\psline[linecolor=black, linewidth=0.08](1.4013889,-1.7693055)(2.2013888,-0.9693054)(3.0013888,-0.9693054)(3.8013887,-0.9693054)(4.601389,-1.7693055)(3.8013887,-2.5693054)(2.2013888,-2.5693054)(1.4013889,-1.7693055)(1.4013889,-1.7693055)
\psline[linecolor=black, linewidth=0.08](1.4013889,-4.1693053)(2.2013888,-3.3693054)(3.0013888,-3.3693054)(3.8013887,-3.3693054)(4.601389,-4.1693053)(3.8013887,-4.9693055)(2.2013888,-4.9693055)(1.4013889,-4.1693053)(1.4013889,-4.1693053)
\psline[linecolor=black, linewidth=0.08](4.601389,-1.7693055)(5.8013887,-2.9693055)(4.601389,-4.1693053)(4.601389,-4.1693053)
\psline[linecolor=black, linewidth=0.08](1.4013889,-1.7693055)(0.20138885,-2.9693055)(1.4013889,-4.1693053)
\psdots[linecolor=black, dotsize=0.4](3.0013888,-0.9693054)
\psdots[linecolor=black, dotsize=0.4](3.0013888,-2.5693054)
\psdots[linecolor=black, dotsize=0.4](3.0013888,-3.3693054)
\psdots[linecolor=black, dotsize=0.4](3.0013888,-4.9693055)
\psdots[linecolor=black, fillstyle=solid, dotstyle=o, dotsize=0.4, fillcolor=white](2.2013888,-0.9693054)
\psdots[linecolor=black, fillstyle=solid, dotstyle=o, dotsize=0.4, fillcolor=white](3.8013887,-0.9693054)
\psdots[linecolor=black, fillstyle=solid, dotstyle=o, dotsize=0.4, fillcolor=white](3.8013887,-2.5693054)
\psdots[linecolor=black, fillstyle=solid, dotstyle=o, dotsize=0.4, fillcolor=white](2.2013888,-2.5693054)
\psdots[linecolor=black, fillstyle=solid, dotstyle=o, dotsize=0.4, fillcolor=white](2.2013888,-3.3693054)
\psdots[linecolor=black, fillstyle=solid, dotstyle=o, dotsize=0.4, fillcolor=white](3.8013887,-3.3693054)
\psdots[linecolor=black, fillstyle=solid, dotstyle=o, dotsize=0.4, fillcolor=white](3.8013887,-4.9693055)
\psdots[linecolor=black, fillstyle=solid, dotstyle=o, dotsize=0.4, fillcolor=white](2.2013888,-4.9693055)
\psline[linecolor=black, linewidth=0.08](9.401389,-1.7693055)(10.201389,-0.9693054)(11.001389,-0.9693054)(11.801389,-0.9693054)(12.601389,-1.7693055)(11.801389,-2.5693054)(10.201389,-2.5693054)(9.401389,-1.7693055)(9.401389,-1.7693055)
\psline[linecolor=black, linewidth=0.08](9.401389,-4.1693053)(10.201389,-3.3693054)(11.001389,-3.3693054)(11.801389,-3.3693054)(12.601389,-4.1693053)(11.801389,-4.9693055)(10.201389,-4.9693055)(9.401389,-4.1693053)(9.401389,-4.1693053)
\psline[linecolor=black, linewidth=0.08](12.601389,-1.7693055)(13.801389,-2.9693055)(12.601389,-4.1693053)(12.601389,-4.1693053)
\psline[linecolor=black, linewidth=0.08](9.401389,-1.7693055)(8.201389,-2.9693055)(9.401389,-4.1693053)
\psdots[linecolor=black, dotsize=0.4](11.001389,-0.9693054)
\psdots[linecolor=black, dotsize=0.4](11.001389,-2.5693054)
\psdots[linecolor=black, dotsize=0.4](11.001389,-3.3693054)
\psdots[linecolor=black, dotsize=0.4](11.001389,-4.9693055)
\psdots[linecolor=black, fillstyle=solid, dotstyle=o, dotsize=0.4, fillcolor=white](10.201389,-0.9693054)
\psdots[linecolor=black, fillstyle=solid, dotstyle=o, dotsize=0.4, fillcolor=white](11.801389,-0.9693054)
\psdots[linecolor=black, fillstyle=solid, dotstyle=o, dotsize=0.4, fillcolor=white](11.801389,-2.5693054)
\psdots[linecolor=black, fillstyle=solid, dotstyle=o, dotsize=0.4, fillcolor=white](10.201389,-2.5693054)
\psdots[linecolor=black, fillstyle=solid, dotstyle=o, dotsize=0.4, fillcolor=white](10.201389,-3.3693054)
\psdots[linecolor=black, fillstyle=solid, dotstyle=o, dotsize=0.4, fillcolor=white](11.801389,-3.3693054)
\psdots[linecolor=black, fillstyle=solid, dotstyle=o, dotsize=0.4, fillcolor=white](11.801389,-4.9693055)
\psdots[linecolor=black, fillstyle=solid, dotstyle=o, dotsize=0.4, fillcolor=white](10.201389,-4.9693055)
\psline[linecolor=black, linewidth=0.08](15.801389,-1.7693055)(16.601389,-0.9693054)(17.401388,-0.9693054)(18.20139,-0.9693054)(19.001389,-1.7693055)(18.20139,-2.5693054)(16.601389,-2.5693054)(15.801389,-1.7693055)(15.801389,-1.7693055)
\psline[linecolor=black, linewidth=0.08](15.801389,-4.1693053)(16.601389,-3.3693054)(17.401388,-3.3693054)(18.20139,-3.3693054)(19.001389,-4.1693053)(18.20139,-4.9693055)(16.601389,-4.9693055)(15.801389,-4.1693053)(15.801389,-4.1693053)
\psline[linecolor=black, linewidth=0.08](19.001389,-1.7693055)(20.20139,-2.9693055)(19.001389,-4.1693053)(19.001389,-4.1693053)
\psline[linecolor=black, linewidth=0.08](15.801389,-1.7693055)(14.601389,-2.9693055)(15.801389,-4.1693053)
\psdots[linecolor=black, dotsize=0.4](17.401388,-0.9693054)
\psdots[linecolor=black, dotsize=0.4](17.401388,-2.5693054)
\psdots[linecolor=black, dotsize=0.4](17.401388,-3.3693054)
\psdots[linecolor=black, dotsize=0.4](17.401388,-4.9693055)
\psdots[linecolor=black, fillstyle=solid, dotstyle=o, dotsize=0.4, fillcolor=white](16.601389,-0.9693054)
\psdots[linecolor=black, fillstyle=solid, dotstyle=o, dotsize=0.4, fillcolor=white](18.20139,-0.9693054)
\psdots[linecolor=black, fillstyle=solid, dotstyle=o, dotsize=0.4, fillcolor=white](18.20139,-2.5693054)
\psdots[linecolor=black, fillstyle=solid, dotstyle=o, dotsize=0.4, fillcolor=white](16.601389,-2.5693054)
\psdots[linecolor=black, fillstyle=solid, dotstyle=o, dotsize=0.4, fillcolor=white](16.601389,-3.3693054)
\psdots[linecolor=black, fillstyle=solid, dotstyle=o, dotsize=0.4, fillcolor=white](18.20139,-3.3693054)
\psdots[linecolor=black, fillstyle=solid, dotstyle=o, dotsize=0.4, fillcolor=white](18.20139,-4.9693055)
\psdots[linecolor=black, fillstyle=solid, dotstyle=o, dotsize=0.4, fillcolor=white](16.601389,-4.9693055)
\psline[linecolor=black, linewidth=0.08](23.80139,-1.7693055)(24.601389,-0.9693054)(25.401388,-0.9693054)(26.20139,-0.9693054)(27.001389,-1.7693055)(26.20139,-2.5693054)(24.601389,-2.5693054)(23.80139,-1.7693055)(23.80139,-1.7693055)
\psline[linecolor=black, linewidth=0.08](23.80139,-4.1693053)(24.601389,-3.3693054)(25.401388,-3.3693054)(26.20139,-3.3693054)(27.001389,-4.1693053)(26.20139,-4.9693055)(24.601389,-4.9693055)(23.80139,-4.1693053)(23.80139,-4.1693053)
\psline[linecolor=black, linewidth=0.08](27.001389,-1.7693055)(28.20139,-2.9693055)(27.001389,-4.1693053)(27.001389,-4.1693053)
\psline[linecolor=black, linewidth=0.08](23.80139,-1.7693055)(22.601389,-2.9693055)(23.80139,-4.1693053)
\psdots[linecolor=black, dotsize=0.4](25.401388,-0.9693054)
\psdots[linecolor=black, dotsize=0.4](25.401388,-2.5693054)
\psdots[linecolor=black, dotsize=0.4](25.401388,-3.3693054)
\psdots[linecolor=black, dotsize=0.4](25.401388,-4.9693055)
\psdots[linecolor=black, fillstyle=solid, dotstyle=o, dotsize=0.4, fillcolor=white](24.601389,-0.9693054)
\psdots[linecolor=black, fillstyle=solid, dotstyle=o, dotsize=0.4, fillcolor=white](26.20139,-0.9693054)
\psdots[linecolor=black, fillstyle=solid, dotstyle=o, dotsize=0.4, fillcolor=white](26.20139,-2.5693054)
\psdots[linecolor=black, fillstyle=solid, dotstyle=o, dotsize=0.4, fillcolor=white](24.601389,-2.5693054)
\psdots[linecolor=black, fillstyle=solid, dotstyle=o, dotsize=0.4, fillcolor=white](24.601389,-3.3693054)
\psdots[linecolor=black, fillstyle=solid, dotstyle=o, dotsize=0.4, fillcolor=white](26.20139,-3.3693054)
\psdots[linecolor=black, fillstyle=solid, dotstyle=o, dotsize=0.4, fillcolor=white](26.20139,-4.9693055)
\psdots[linecolor=black, fillstyle=solid, dotstyle=o, dotsize=0.4, fillcolor=white](24.601389,-4.9693055)
\psdots[linecolor=black, dotsize=0.1](21.001389,-2.9693055)
\psdots[linecolor=black, dotsize=0.1](21.401388,-2.9693055)
\psdots[linecolor=black, dotsize=0.1](21.80139,-2.9693055)
\rput[bl](2.7613888,-9.229305){\Large{$G_i$}}
\rput[bl](0.5213888,-3.0893054){$x_i$}
\rput[bl](5.161389,-3.1093054){$y_i$}
\rput[bl](8.481389,-3.1293054){$x_1$}
\rput[bl](13.221389,-3.1293054){$y_1$}
\rput[bl](14.861389,-3.0893054){$x_2$}
\rput[bl](19.581388,-3.1093054){$y_2$}
\rput[bl](22.86139,-3.1493053){$x_n$}
\rput[bl](27.62139,-3.0693054){$y_n$}
\psline[linecolor=black, linewidth=0.08](1.4013889,-1.7693055)(4.601389,-1.7693055)(4.601389,-1.7693055)
\psline[linecolor=black, linewidth=0.08](1.4013889,-4.1693053)(4.601389,-4.1693053)
\psline[linecolor=black, linewidth=0.08](2.2013888,-5.769305)(3.8013887,-5.769305)(4.601389,-6.5693054)(3.8013887,-7.3693056)(2.2013888,-7.3693056)(1.4013889,-6.5693054)(2.2013888,-5.769305)(2.2013888,-5.769305)
\psline[linecolor=black, linewidth=0.08](1.4013889,-6.5693054)(4.601389,-6.5693054)(4.601389,-6.5693054)
\psline[linecolor=black, linewidth=0.08](10.201389,-5.769305)(11.801389,-5.769305)(12.601389,-6.5693054)(11.801389,-7.3693056)(10.201389,-7.3693056)(9.401389,-6.5693054)(10.201389,-5.769305)(10.201389,-5.769305)
\psline[linecolor=black, linewidth=0.08](9.401389,-6.5693054)(12.601389,-6.5693054)(12.601389,-6.5693054)
\psline[linecolor=black, linewidth=0.08](16.601389,-5.769305)(18.20139,-5.769305)(19.001389,-6.5693054)(18.20139,-7.3693056)(16.601389,-7.3693056)(15.801389,-6.5693054)(16.601389,-5.769305)(16.601389,-5.769305)
\psline[linecolor=black, linewidth=0.08](15.801389,-6.5693054)(19.001389,-6.5693054)(19.001389,-6.5693054)
\psline[linecolor=black, linewidth=0.08](24.601389,-6.1693053)(26.20139,-6.1693053)(27.001389,-6.9693055)(26.20139,-7.769305)(24.601389,-7.769305)(23.80139,-6.9693055)(24.601389,-6.1693053)(24.601389,-6.1693053)
\psline[linecolor=black, linewidth=0.08](23.80139,-6.9693055)(27.001389,-6.9693055)(27.001389,-6.9693055)
\psline[linecolor=black, linewidth=0.08](0.20138885,-2.9693055)(1.4013889,-6.5693054)(1.4013889,-6.5693054)
\psline[linecolor=black, linewidth=0.08](5.8013887,-2.9693055)(4.601389,-6.5693054)(4.601389,-6.5693054)
\psline[linecolor=black, linewidth=0.08](9.401389,-1.7693055)(12.601389,-1.7693055)
\psline[linecolor=black, linewidth=0.08](9.401389,-4.1693053)(12.601389,-4.1693053)
\psline[linecolor=black, linewidth=0.08](15.801389,-1.7693055)(19.001389,-1.7693055)
\psline[linecolor=black, linewidth=0.08](15.801389,-4.1693053)(19.001389,-4.1693053)
\psline[linecolor=black, linewidth=0.08](23.80139,-1.7693055)(27.001389,-1.7693055)
\psline[linecolor=black, linewidth=0.08](23.80139,-4.1693053)(27.001389,-4.1693053)
\psline[linecolor=black, linewidth=0.08](8.201389,-2.9693055)(9.401389,-6.5693054)(9.401389,-6.5693054)
\psline[linecolor=black, linewidth=0.08](13.801389,-2.9693055)(12.601389,-6.5693054)(12.601389,-6.5693054)
\psline[linecolor=black, linewidth=0.08](13.801389,-2.9693055)(14.601389,-2.9693055)(14.601389,-2.9693055)
\psline[linecolor=black, linewidth=0.08](14.601389,-2.9693055)(15.801389,-6.5693054)(15.801389,-6.5693054)
\psline[linecolor=black, linewidth=0.08](20.20139,-2.9693055)(19.001389,-6.5693054)(19.001389,-6.5693054)
\psline[linecolor=black, linewidth=0.08](20.601389,-2.9693055)(20.20139,-2.9693055)
\psline[linecolor=black, linewidth=0.08](22.20139,-2.9693055)(22.601389,-2.9693055)
\psline[linecolor=black, linewidth=0.08](22.601389,-2.9693055)(23.80139,-6.9693055)(23.80139,-6.9693055)
\psline[linecolor=black, linewidth=0.08](28.20139,-2.9693055)(27.001389,-6.9693055)(27.001389,-6.9693055)
\rput[bl](16.80139,-9.109305){\Large{$L(G_1,...,G_n)$}}
\psdots[linecolor=black, dotsize=0.4](3.0013888,-1.7693055)
\psdots[linecolor=black, dotsize=0.4](3.0013888,-4.1693053)
\psdots[linecolor=black, dotsize=0.4](3.0013888,-5.769305)
\psdots[linecolor=black, dotsize=0.4](3.0013888,-6.5693054)
\psdots[linecolor=black, dotsize=0.4](3.0013888,-7.3693056)
\psdots[linecolor=black, dotsize=0.4](11.001389,-1.7693055)
\psdots[linecolor=black, dotsize=0.4](11.001389,-4.1693053)
\psdots[linecolor=black, dotsize=0.4](11.001389,-5.769305)
\psdots[linecolor=black, dotsize=0.4](11.001389,-6.5693054)
\psdots[linecolor=black, dotsize=0.4](11.001389,-7.3693056)
\psdots[linecolor=black, dotsize=0.4](17.401388,-1.7693055)
\psdots[linecolor=black, dotsize=0.4](17.401388,-4.1693053)
\psdots[linecolor=black, dotsize=0.4](17.401388,-5.769305)
\psdots[linecolor=black, dotsize=0.4](17.401388,-6.5693054)
\psdots[linecolor=black, dotsize=0.4](17.401388,-7.3693056)
\psdots[linecolor=black, dotsize=0.4](25.401388,-1.7693055)
\psdots[linecolor=black, dotsize=0.4](25.401388,-4.1693053)
\psdots[linecolor=black, dotsize=0.4](25.401388,-6.1693053)
\psdots[linecolor=black, dotsize=0.4](25.401388,-6.9693055)
\psdots[linecolor=black, dotsize=0.4](25.401388,-7.769305)
\psdots[linecolor=black, dotsize=0.4](1.4013889,-1.7693055)
\psdots[linecolor=black, dotsize=0.4](1.4013889,-4.1693053)
\psdots[linecolor=black, dotsize=0.4](1.4013889,-6.5693054)
\psdots[linecolor=black, dotsize=0.4](4.601389,-1.7693055)
\psdots[linecolor=black, dotsize=0.4](4.601389,-4.1693053)
\psdots[linecolor=black, dotsize=0.4](4.601389,-6.5693054)
\psdots[linecolor=black, dotsize=0.4](9.401389,-1.7693055)
\psdots[linecolor=black, dotsize=0.4](9.401389,-4.1693053)
\psdots[linecolor=black, dotsize=0.4](9.401389,-6.5693054)
\psdots[linecolor=black, dotsize=0.4](27.001389,-1.7693055)
\psdots[linecolor=black, dotsize=0.4](27.001389,-4.1693053)
\psdots[linecolor=black, dotsize=0.4](27.001389,-6.9693055)
\psdots[linecolor=black, dotsize=0.4](13.801389,-2.9693055)
\psdots[linecolor=black, dotsize=0.4](14.601389,-2.9693055)
\psdots[linecolor=black, dotsize=0.4](20.20139,-2.9693055)
\psdots[linecolor=black, dotsize=0.4](22.601389,-2.9693055)
\psdots[linecolor=black, dotstyle=o, dotsize=0.4, fillcolor=white](2.2013888,-4.1693053)
\psdots[linecolor=black, dotstyle=o, dotsize=0.4, fillcolor=white](3.8013887,-4.1693053)
\psdots[linecolor=black, dotstyle=o, dotsize=0.4, fillcolor=white](2.2013888,-6.5693054)
\psdots[linecolor=black, dotstyle=o, dotsize=0.4, fillcolor=white](3.8013887,-6.5693054)
\psdots[linecolor=black, dotstyle=o, dotsize=0.4, fillcolor=white](3.8013887,-5.769305)
\psdots[linecolor=black, dotstyle=o, dotsize=0.4, fillcolor=white](2.2013888,-5.769305)
\psdots[linecolor=black, dotstyle=o, dotsize=0.4, fillcolor=white](2.2013888,-7.3693056)
\psdots[linecolor=black, dotstyle=o, dotsize=0.4, fillcolor=white](3.8013887,-7.3693056)
\psdots[linecolor=black, dotstyle=o, dotsize=0.4, fillcolor=white](5.8013887,-2.9693055)
\psdots[linecolor=black, dotstyle=o, dotsize=0.4, fillcolor=white](0.20138885,-2.9693055)
\psdots[linecolor=black, dotstyle=o, dotsize=0.4, fillcolor=white](8.201389,-2.9693055)
\psdots[linecolor=black, dotstyle=o, dotsize=0.4, fillcolor=white](10.201389,-1.7693055)
\psdots[linecolor=black, dotstyle=o, dotsize=0.4, fillcolor=white](10.201389,-4.1693053)
\psdots[linecolor=black, dotstyle=o, dotsize=0.4, fillcolor=white](10.201389,-5.769305)
\psdots[linecolor=black, dotstyle=o, dotsize=0.4, fillcolor=white](10.201389,-6.5693054)
\psdots[linecolor=black, dotstyle=o, dotsize=0.4, fillcolor=white](10.201389,-7.3693056)
\psdots[linecolor=black, dotstyle=o, dotsize=0.4, fillcolor=white](11.801389,-1.7693055)
\psdots[linecolor=black, dotstyle=o, dotsize=0.4, fillcolor=white](11.801389,-4.1693053)
\psdots[linecolor=black, dotstyle=o, dotsize=0.4, fillcolor=white](11.801389,-5.769305)
\psdots[linecolor=black, dotstyle=o, dotsize=0.4, fillcolor=white](11.801389,-6.5693054)
\psdots[linecolor=black, dotstyle=o, dotsize=0.4, fillcolor=white](11.801389,-7.3693056)
\psdots[linecolor=black, dotstyle=o, dotsize=0.4, fillcolor=white](12.601389,-1.7693055)
\psdots[linecolor=black, dotstyle=o, dotsize=0.4, fillcolor=white](12.601389,-4.1693053)
\psdots[linecolor=black, dotstyle=o, dotsize=0.4, fillcolor=white](12.601389,-6.5693054)
\psdots[linecolor=black, dotstyle=o, dotsize=0.4, fillcolor=white](15.801389,-1.7693055)
\psdots[linecolor=black, dotstyle=o, dotsize=0.4, fillcolor=white](15.801389,-4.1693053)
\psdots[linecolor=black, dotstyle=o, dotsize=0.4, fillcolor=white](15.801389,-6.5693054)
\psdots[linecolor=black, dotstyle=o, dotsize=0.4, fillcolor=white](16.601389,-1.7693055)
\psdots[linecolor=black, dotstyle=o, dotsize=0.4, fillcolor=white](16.601389,-4.1693053)
\psdots[linecolor=black, dotstyle=o, dotsize=0.4, fillcolor=white](16.601389,-6.5693054)
\psdots[linecolor=black, dotstyle=o, dotsize=0.4, fillcolor=white](16.601389,-5.769305)
\psdots[linecolor=black, dotstyle=o, dotsize=0.4, fillcolor=white](16.601389,-7.3693056)
\psdots[linecolor=black, dotstyle=o, dotsize=0.4, fillcolor=white](18.20139,-1.7693055)
\psdots[linecolor=black, dotstyle=o, dotsize=0.4, fillcolor=white](18.20139,-4.1693053)
\psdots[linecolor=black, dotstyle=o, dotsize=0.4, fillcolor=white](18.20139,-5.769305)
\psdots[linecolor=black, dotstyle=o, dotsize=0.4, fillcolor=white](18.20139,-6.5693054)
\psdots[linecolor=black, dotstyle=o, dotsize=0.4, fillcolor=white](18.20139,-7.3693056)
\psdots[linecolor=black, dotstyle=o, dotsize=0.4, fillcolor=white](19.001389,-1.7693055)
\psdots[linecolor=black, dotstyle=o, dotsize=0.4, fillcolor=white](19.001389,-4.1693053)
\psdots[linecolor=black, dotstyle=o, dotsize=0.4, fillcolor=white](19.001389,-6.5693054)
\psdots[linecolor=black, dotstyle=o, dotsize=0.4, fillcolor=white](23.80139,-1.7693055)
\psdots[linecolor=black, dotstyle=o, dotsize=0.4, fillcolor=white](23.80139,-4.1693053)
\psdots[linecolor=black, dotstyle=o, dotsize=0.4, fillcolor=white](23.80139,-6.9693055)
\psdots[linecolor=black, dotstyle=o, dotsize=0.4, fillcolor=white](24.601389,-1.7693055)
\psdots[linecolor=black, dotstyle=o, dotsize=0.4, fillcolor=white](24.601389,-4.1693053)
\psdots[linecolor=black, dotstyle=o, dotsize=0.4, fillcolor=white](24.601389,-6.1693053)
\psdots[linecolor=black, dotstyle=o, dotsize=0.4, fillcolor=white](24.601389,-6.9693055)
\psdots[linecolor=black, dotstyle=o, dotsize=0.4, fillcolor=white](24.601389,-7.769305)
\psdots[linecolor=black, dotstyle=o, dotsize=0.4, fillcolor=white](26.20139,-1.7693055)
\psdots[linecolor=black, dotstyle=o, dotsize=0.4, fillcolor=white](26.20139,-4.1693053)
\psdots[linecolor=black, dotstyle=o, dotsize=0.4, fillcolor=white](26.20139,-6.1693053)
\psdots[linecolor=black, dotstyle=o, dotsize=0.4, fillcolor=white](26.20139,-6.9693055)
\psdots[linecolor=black, dotstyle=o, dotsize=0.4, fillcolor=white](26.20139,-7.769305)
\psdots[linecolor=black, dotstyle=o, dotsize=0.4, fillcolor=white](28.20139,-2.9693055)
\psdots[linecolor=black, dotstyle=o, dotsize=0.4, fillcolor=white](2.2013888,-1.7693055)
\psdots[linecolor=black, dotstyle=o, dotsize=0.4, fillcolor=white](3.8013887,-1.7693055)
\end{pspicture}
}
		\end{center}
		\caption{Graphs $G_i$, for $i=1,2,\ldots,n$, and $L(G_1,...,G_n)$, respectively.} \label{fig:link-lower}
	\end{figure}

\begin{remark}
{\rm We remark that the bounds in Theorem \ref{thm:link} are tight. For the lower bound, it suffices to consider graphs $G_i$, for $i \in [n]$, and their link with respect to the vertices $\{x_i, y_i\}_{i=1}^n$ as illustrated in Figure \ref{fig:link-lower}. The set of black vertices in each graph is a $\gst$-set, and we are done. This idea can be generalized, and therefore we have an infinite family of graphs such that the equality of the lower bound holds.
 For the upper bound, consider Figure~\ref{fig:link-upper}. By the same argument as the lower bound, the set of black vertices in each graph is a $\gst$-set, and we are done. This idea can be generalized, so we have an infinite family of graphs such that the equality of the upper bound holds.}
	\end{remark}

	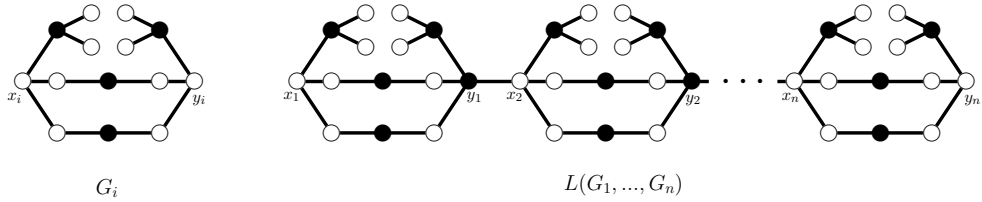
\begin{figure}
		\begin{center}
			\psscalebox{0.57 0.57}
{
\begin{pspicture}(0,-8.049305)(22.73,-3.5479162)
\rput[bl](2.06,-8.049305){\Large{$G_i$}}
\rput[bl](0.0,-5.8493056){$x_i$}
\rput[bl](4.32,-5.8893056){$y_i$}
\rput[bl](6.44,-5.809305){$x_1$}
\rput[bl](10.72,-5.8293056){$y_1$}
\rput[bl](11.64,-5.809305){$x_2$}
\rput[bl](15.82,-5.8893056){$y_2$}
\rput[bl](18.06,-5.8293056){$x_n$}
\rput[bl](22.32,-5.8893056){$y_n$}
\rput[bl](12.98,-8.009305){\Large{$L(G_1,...,G_n)$}}
\psline[linecolor=black, linewidth=0.08](0.36,-5.3493056)(1.16,-4.1493053)(1.16,-4.1493053)
\psline[linecolor=black, linewidth=0.08](0.36,-5.3493056)(1.16,-6.5493054)(1.16,-6.5493054)
\psline[linecolor=black, linewidth=0.08](1.16,-6.5493054)(1.96,-6.5493054)(2.76,-6.5493054)(3.56,-6.5493054)(3.56,-6.5493054)
\psline[linecolor=black, linewidth=0.08](0.36,-5.3493056)(4.36,-5.3493056)(4.36,-5.3493056)
\psline[linecolor=black, linewidth=0.08](4.36,-5.3493056)(3.56,-4.1493053)(3.56,-4.1493053)
\psline[linecolor=black, linewidth=0.08](1.16,-4.1493053)(1.96,-4.5493054)(1.96,-4.5493054)
\psline[linecolor=black, linewidth=0.08](1.16,-4.1493053)(1.96,-3.7493055)(1.96,-3.7493055)
\psline[linecolor=black, linewidth=0.08](2.76,-3.7493055)(3.56,-4.1493053)(3.56,-4.1493053)
\psline[linecolor=black, linewidth=0.08](2.76,-4.5493054)(3.56,-4.1493053)(3.56,-4.1493053)
\psline[linecolor=black, linewidth=0.08](4.36,-5.3493056)(3.56,-6.5493054)(3.56,-6.5493054)
\psdots[linecolor=black, dotstyle=o, dotsize=0.4, fillcolor=white](2.76,-3.7493055)
\psdots[linecolor=black, dotstyle=o, dotsize=0.4, fillcolor=white](1.96,-3.7493055)
\psdots[linecolor=black, dotstyle=o, dotsize=0.4, fillcolor=white](1.96,-4.5493054)
\psdots[linecolor=black, dotstyle=o, dotsize=0.4, fillcolor=white](2.76,-4.5493054)
\psdots[linecolor=black, dotstyle=o, dotsize=0.4, fillcolor=white](1.16,-5.3493056)
\psdots[linecolor=black, dotstyle=o, dotsize=0.4, fillcolor=white](3.56,-5.3493056)
\psdots[linecolor=black, dotstyle=o, dotsize=0.4, fillcolor=white](1.16,-6.5493054)
\psdots[linecolor=black, dotstyle=o, dotsize=0.4, fillcolor=white](3.56,-6.5493054)
\psdots[linecolor=black, dotstyle=o, dotsize=0.4, fillcolor=white](0.36,-5.3493056)
\psdots[linecolor=black, dotsize=0.4](1.16,-4.1493053)
\psdots[linecolor=black, dotsize=0.4](3.56,-4.1493053)
\psdots[linecolor=black, dotsize=0.4](2.36,-5.3493056)
\psdots[linecolor=black, dotsize=0.4](2.36,-6.5493054)
\psline[linecolor=black, linewidth=0.08](6.76,-5.3493056)(7.56,-4.1493053)(7.56,-4.1493053)
\psline[linecolor=black, linewidth=0.08](6.76,-5.3493056)(7.56,-6.5493054)(7.56,-6.5493054)
\psline[linecolor=black, linewidth=0.08](7.56,-6.5493054)(8.36,-6.5493054)(9.16,-6.5493054)(9.96,-6.5493054)(9.96,-6.5493054)
\psline[linecolor=black, linewidth=0.08](6.76,-5.3493056)(10.76,-5.3493056)(10.76,-5.3493056)
\psline[linecolor=black, linewidth=0.08](10.76,-5.3493056)(9.96,-4.1493053)(9.96,-4.1493053)
\psline[linecolor=black, linewidth=0.08](7.56,-4.1493053)(8.36,-4.5493054)(8.36,-4.5493054)
\psline[linecolor=black, linewidth=0.08](7.56,-4.1493053)(8.36,-3.7493055)(8.36,-3.7493055)
\psline[linecolor=black, linewidth=0.08](9.16,-3.7493055)(9.96,-4.1493053)(9.96,-4.1493053)
\psline[linecolor=black, linewidth=0.08](9.16,-4.5493054)(9.96,-4.1493053)(9.96,-4.1493053)
\psline[linecolor=black, linewidth=0.08](10.76,-5.3493056)(9.96,-6.5493054)(9.96,-6.5493054)
\psdots[linecolor=black, dotstyle=o, dotsize=0.4, fillcolor=white](9.16,-3.7493055)
\psdots[linecolor=black, dotstyle=o, dotsize=0.4, fillcolor=white](8.36,-3.7493055)
\psdots[linecolor=black, dotstyle=o, dotsize=0.4, fillcolor=white](8.36,-4.5493054)
\psdots[linecolor=black, dotstyle=o, dotsize=0.4, fillcolor=white](9.16,-4.5493054)
\psdots[linecolor=black, dotstyle=o, dotsize=0.4, fillcolor=white](7.56,-5.3493056)
\psdots[linecolor=black, dotstyle=o, dotsize=0.4, fillcolor=white](9.96,-5.3493056)
\psdots[linecolor=black, dotstyle=o, dotsize=0.4, fillcolor=white](7.56,-6.5493054)
\psdots[linecolor=black, dotstyle=o, dotsize=0.4, fillcolor=white](9.96,-6.5493054)
\psdots[linecolor=black, dotsize=0.4](7.56,-4.1493053)
\psdots[linecolor=black, dotsize=0.4](9.96,-4.1493053)
\psdots[linecolor=black, dotsize=0.4](8.76,-5.3493056)
\psdots[linecolor=black, dotsize=0.4](8.76,-6.5493054)
\psline[linecolor=black, linewidth=0.08](11.96,-5.3493056)(12.76,-4.1493053)(12.76,-4.1493053)
\psline[linecolor=black, linewidth=0.08](11.96,-5.3493056)(12.76,-6.5493054)(12.76,-6.5493054)
\psline[linecolor=black, linewidth=0.08](12.76,-6.5493054)(13.56,-6.5493054)(14.36,-6.5493054)(15.16,-6.5493054)(15.16,-6.5493054)
\psline[linecolor=black, linewidth=0.08](11.96,-5.3493056)(15.96,-5.3493056)(15.96,-5.3493056)
\psline[linecolor=black, linewidth=0.08](15.96,-5.3493056)(15.16,-4.1493053)(15.16,-4.1493053)
\psline[linecolor=black, linewidth=0.08](12.76,-4.1493053)(13.56,-4.5493054)(13.56,-4.5493054)
\psline[linecolor=black, linewidth=0.08](12.76,-4.1493053)(13.56,-3.7493055)(13.56,-3.7493055)
\psline[linecolor=black, linewidth=0.08](14.36,-3.7493055)(15.16,-4.1493053)(15.16,-4.1493053)
\psline[linecolor=black, linewidth=0.08](14.36,-4.5493054)(15.16,-4.1493053)(15.16,-4.1493053)
\psline[linecolor=black, linewidth=0.08](15.96,-5.3493056)(15.16,-6.5493054)(15.16,-6.5493054)
\psdots[linecolor=black, dotstyle=o, dotsize=0.4, fillcolor=white](14.36,-3.7493055)
\psdots[linecolor=black, dotstyle=o, dotsize=0.4, fillcolor=white](13.56,-3.7493055)
\psdots[linecolor=black, dotstyle=o, dotsize=0.4, fillcolor=white](13.56,-4.5493054)
\psdots[linecolor=black, dotstyle=o, dotsize=0.4, fillcolor=white](14.36,-4.5493054)
\psdots[linecolor=black, dotstyle=o, dotsize=0.4, fillcolor=white](12.76,-5.3493056)
\psdots[linecolor=black, dotstyle=o, dotsize=0.4, fillcolor=white](15.16,-5.3493056)
\psdots[linecolor=black, dotstyle=o, dotsize=0.4, fillcolor=white](12.76,-6.5493054)
\psdots[linecolor=black, dotstyle=o, dotsize=0.4, fillcolor=white](15.16,-6.5493054)
\psdots[linecolor=black, dotsize=0.4](12.76,-4.1493053)
\psdots[linecolor=black, dotsize=0.4](15.16,-4.1493053)
\psdots[linecolor=black, dotsize=0.4](13.96,-5.3493056)
\psdots[linecolor=black, dotsize=0.4](13.96,-6.5493054)
\psline[linecolor=black, linewidth=0.08](18.36,-5.3493056)(19.16,-4.1493053)(19.16,-4.1493053)
\psline[linecolor=black, linewidth=0.08](18.36,-5.3493056)(19.16,-6.5493054)(19.16,-6.5493054)
\psline[linecolor=black, linewidth=0.08](19.16,-6.5493054)(19.96,-6.5493054)(20.76,-6.5493054)(21.56,-6.5493054)(21.56,-6.5493054)
\psline[linecolor=black, linewidth=0.08](18.36,-5.3493056)(22.36,-5.3493056)(22.36,-5.3493056)
\psline[linecolor=black, linewidth=0.08](22.36,-5.3493056)(21.56,-4.1493053)(21.56,-4.1493053)
\psline[linecolor=black, linewidth=0.08](19.16,-4.1493053)(19.96,-4.5493054)(19.96,-4.5493054)
\psline[linecolor=black, linewidth=0.08](19.16,-4.1493053)(19.96,-3.7493055)(19.96,-3.7493055)
\psline[linecolor=black, linewidth=0.08](20.76,-3.7493055)(21.56,-4.1493053)(21.56,-4.1493053)
\psline[linecolor=black, linewidth=0.08](20.76,-4.5493054)(21.56,-4.1493053)(21.56,-4.1493053)
\psline[linecolor=black, linewidth=0.08](22.36,-5.3493056)(21.56,-6.5493054)(21.56,-6.5493054)
\psdots[linecolor=black, dotstyle=o, dotsize=0.4, fillcolor=white](20.76,-3.7493055)
\psdots[linecolor=black, dotstyle=o, dotsize=0.4, fillcolor=white](19.96,-3.7493055)
\psdots[linecolor=black, dotstyle=o, dotsize=0.4, fillcolor=white](19.96,-4.5493054)
\psdots[linecolor=black, dotstyle=o, dotsize=0.4, fillcolor=white](20.76,-4.5493054)
\psdots[linecolor=black, dotstyle=o, dotsize=0.4, fillcolor=white](19.16,-5.3493056)
\psdots[linecolor=black, dotstyle=o, dotsize=0.4, fillcolor=white](21.56,-5.3493056)
\psdots[linecolor=black, dotstyle=o, dotsize=0.4, fillcolor=white](19.16,-6.5493054)
\psdots[linecolor=black, dotstyle=o, dotsize=0.4, fillcolor=white](21.56,-6.5493054)
\psdots[linecolor=black, dotsize=0.4](19.16,-4.1493053)
\psdots[linecolor=black, dotsize=0.4](21.56,-4.1493053)
\psdots[linecolor=black, dotsize=0.4](20.36,-5.3493056)
\psdots[linecolor=black, dotsize=0.4](20.36,-6.5493054)
\psline[linecolor=black, linewidth=0.08](16.36,-5.3493056)(15.96,-5.3493056)(15.96,-5.3493056)
\psline[linecolor=black, linewidth=0.08](17.96,-5.3493056)(18.36,-5.3493056)(18.36,-5.3493056)
\psline[linecolor=black, linewidth=0.08](10.76,-5.3493056)(11.96,-5.3493056)(11.96,-5.3493056)
\psdots[linecolor=black, dotsize=0.4](10.76,-5.3493056)
\psdots[linecolor=black, dotsize=0.4](15.96,-5.3493056)
\psdots[linecolor=black, dotstyle=o, dotsize=0.4, fillcolor=white](22.36,-5.3493056)
\psdots[linecolor=black, dotstyle=o, dotsize=0.4, fillcolor=white](18.36,-5.3493056)
\psdots[linecolor=black, dotstyle=o, dotsize=0.4, fillcolor=white](11.96,-5.3493056)
\psdots[linecolor=black, dotstyle=o, dotsize=0.4, fillcolor=white](6.76,-5.3493056)
\psdots[linecolor=black, dotstyle=o, dotsize=0.4, fillcolor=white](4.36,-5.3493056)
\psdots[linecolor=black, dotsize=0.1](16.76,-5.3493056)
\psdots[linecolor=black, dotsize=0.1](17.16,-5.3493056)
\psdots[linecolor=black, dotsize=0.1](17.56,-5.3493056)
\end{pspicture}
}
		\end{center}
		\caption{Graphs $G_i$, for $i=1,2,\ldots,n$, and $C(G_1,...,G_n)$, respectively.} \label{fig:link-upper}
	\end{figure}

	%%%%%%%%%%%%%%%%%%%%%%%%%%%%%%%%%%%%%%%%%%%%%%%%%%%%%%%%%%%%%%%%%%%%%%%%%	

\section{Circuit of Graphs}

In this section, we study the strong domination number of circuit of graphs.

	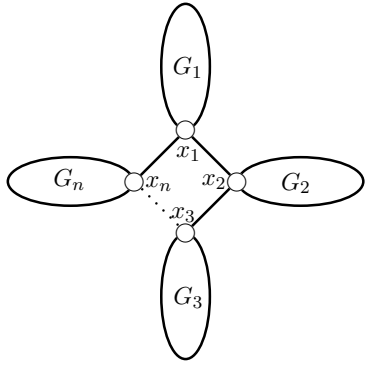
\begin{figure}
		\begin{center}
			\psscalebox{0.85 0.85}
			{
				\begin{pspicture}(0,-7.6)(5.6,-2.0)
				\rput[bl](2.6533334,-4.48){$x_1$}
				\rput[bl](3.0533333,-4.92){$x_2$}
				\rput[bl](2.5733333,-5.4266667){$x_3$}
				\rput[bl](2.6,-3.1733334){$G_1$}
				\rput[bl](4.2933335,-4.9866667){$G_2$}
				\rput[bl](2.6133332,-6.7733335){$G_3$}
				\rput[bl](2.1733334,-4.9466667){$x_n$}
				\rput[bl](0.73333335,-4.9333334){$G_n$}
				\psellipse[linecolor=black, linewidth=0.04, dimen=outer](1.0,-4.8)(1.0,0.4)
				\psellipse[linecolor=black, linewidth=0.04, dimen=outer](4.6,-4.8)(1.0,0.4)
				\psellipse[linecolor=black, linewidth=0.04, dimen=outer](2.8,-3.0)(0.4,1.0)
				\psellipse[linecolor=black, linewidth=0.04, dimen=outer](2.8,-6.6)(0.4,1.0)
				\psline[linecolor=black, linewidth=0.04](2.0,-4.8)(2.8,-4.0)(3.6,-4.8)(2.8,-5.6)(2.8,-5.6)
				\psdots[linecolor=black, fillstyle=solid, dotstyle=o, dotsize=0.3, fillcolor=white](2.8,-4.0)
				\psdots[linecolor=black, fillstyle=solid, dotstyle=o, dotsize=0.3, fillcolor=white](3.6,-4.8)
				\psline[linecolor=black, linewidth=0.04, linestyle=dotted, dotsep=0.10583334cm](2.8,-5.6)(2.0,-4.8)(2.0,-4.8)
				\psdots[linecolor=black, fillstyle=solid, dotstyle=o, dotsize=0.3, fillcolor=white](2.0,-4.8)
				\psdots[linecolor=black, fillstyle=solid, dotstyle=o, dotsize=0.3, fillcolor=white](2.8,-5.6)
				\end{pspicture}
			}
		\end{center}
	\caption{Circuit of $n$ graphs $G_1,G_2, \ldots , G_n$.} \label{circuit-n}
	\end{figure}

	\begin{theorem} \label{thm:circuit}
			Let $G_1,G_2, \ldots , G_n$ be a finite sequence of pairwise disjoint connected graphs and let
	$x_i \in V(G_i)$. If $G$ is the circuit of graphs $\{G_i\}_{i=1}^n$ with respect to the vertices $\{x_i\}_{i=1}^n$ and obtained by identifying the vertex $x_i$ of the graph $G_i$ with the $i$-th vertex of the
	cycle graph $C_n$ (Figure \ref{circuit-n}), then
		\begin{align*}
		\left( \sum_{i=1}^{n}\gst(G_i)\right) -\left(\sum_{i=1}^{n}\deg(x_i)\right)+n \leq  \gst(G)
\leq   \left( \sum_{i=1}^{n}\gst(G_i)\right) +\lfloor\frac{n}{2}\rfloor.
		\end{align*}
	\end{theorem}

	\begin{proof}	
First we prove the lower bound.	Suppose that $D$ is a $\gst$-sets of $G$. We find strong dominating set of $G_i$, for $i=1,2,\ldots,n$, based on $D$. For $1\leq i \leq n$, we consider the following cases for $G_i$:
\begin{itemize}
\item[(i)]
$x_i\notin D$. In this case, there exists $x_i'\in D$ that strongly dominates  $x_i$. If $x_i'\notin \{x_{i-1},x_{i+1}\}$, then let
$$D_i=D\setminus \left(\bigcup\limits_{j=1}^{i-1} V(G_j)\cup \bigcup\limits_{j=i+1}^{n} V(G_j) \right).$$
If $x_i'\in \{x_{i-1},x_{i+1}\}$, then let
$$D_i=\left(D\cup\{x_i\}\right)\setminus \left(\bigcup\limits_{j=1}^{i-1} V(G_j)\cup \bigcup\limits_{j=i+1}^{n} V(G_j) \right).$$
The resulting set $D_i$ is a strong dominating set of $G_i$.
\item[(ii)]
$x_i\in D$. If $x_i$ does not strongly dominate other vertices, then let
$$D_i=D\setminus \left(\bigcup\limits_{j=1}^{i-1} V(G_j)\cup \bigcup\limits_{j=i+1}^{n} V(G_j) \right).$$
If $x_i$ strongly dominates other vertices in its neighbour but not all of them, say $A$, let
$$D_i=\Big(D\cup (N(x_i)\setminus A)\Big)\setminus \left(\bigcup\limits_{j=1}^{i-1} V(G_j)\cup \bigcup\limits_{j=i+1}^{n} V(G_j) \right).$$
The worst case happens when $x_i$ strongly dominate all of its neighbours in $V(G_i)$, but after forming $G_i$ from $G$, we have $\deg(x_i)<\deg(u)$, for all $u\in N(x_i)$. Then let
$$D_i=\left(D\cup N(x_i)\right)\setminus \left(\bigcup\limits_{j=1}^{i-1} V(G_j)\cup \bigcup\limits_{j=i+1}^{n} V(G_j) \cup\{x_i\}\right).$$
The resulting set $D_i$ is a strong dominating set of $G_i$, in each case.
\end{itemize}
So, generally, by considering $\bigcup\limits_{i=1}^{n} D_i$, we have
$$\sum_{i=1}^{n}\gst(G_i)  \leq  \gst(G) +\left(\sum_{i=1}^{n}\deg(x_i)\right)-n,$$
and we are done with the lower bound. Now, we find the upper bound. Suppose that $S_i$ is a $\gst$-sets of $G_i$, for $i=1,2,\ldots,n$. If $x_i\in S_i$, then after forming $G$, the condition of $x_i$ won't change, since its degree increases by two. But if $x_i\notin S_i$, then there exist a vertex $x_i'\in S_i$ that strongly dominates  $x_i$. If
$\deg_G(x_i)\leq \deg(x_i')$, then it can be strongly dominated by $x_i$. But if we have $\deg_G(x_i)> \deg(x_i')$, for $i=1,2,\ldots,n$, we have a different situation. The worst case happens when for all $1\leq i \leq n$, $x_i\notin S_i$.  In this case, we do as follow; First, we find the vertex with maximum degree among $x_1,x_2,\ldots,x_n$. Without loss of generality, suppose that this vertex is $x_{n-1}$. Then if $n$ is odd, for $j=1,2,\ldots,\lfloor\frac{n-2}{2}\rfloor$, let
	\[
	z_j=\left\{
	\begin{array}{ll}
	{\displaystyle
		x_{2j}}&
	\quad\mbox{if $\deg(x_{2j})\geq \deg(x_{2j-1})  $,}\\[15pt]
	{\displaystyle
		x_{2j-1}}&
	\quad\mbox{otherwise,}
	\end{array}
	\right.
	\]
and
if $n$ is even, for $j=1,2,\ldots,\frac{n-4}{2}$, let
	\[
	z_j=\left\{
	\begin{array}{ll}
	{\displaystyle
		x_{2j}}&
	\quad\mbox{if $\deg(x_{2j})\geq \deg(x_{2j-1})  $,}\\[15pt]
	{\displaystyle
		x_{2j-1}}&
	\quad\mbox{otherwise,}
	\end{array}
	\right.
	\]
and $z_{\frac{n-2}{2}}=x_{n-3}$. Then let
$$S=\bigcup\limits_{i=1}^{n} S_i\cup\{z_i\}_{i=1}^{\lfloor\frac{n-2}{2}\rfloor}\cup\{x_{n-1}\}.$$
The resulting set $S$ is a strong dominating set of  $G$.
Note that the way we constructed the strong dominating set in the worst case can be considered in general, and we infer that
$$\gst(G) \leq   \left( \sum_{i=1}^{n}\gst(G_i)\right) +\lfloor\frac{n}{2}\rfloor,$$
and therefore we are done.
\qed
	\end{proof}

	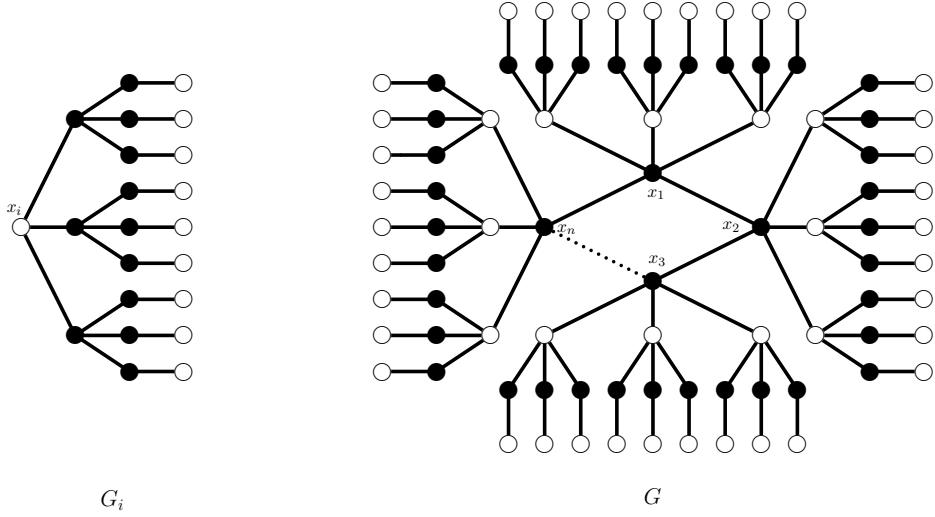
\begin{figure}
		\begin{center}
			\psscalebox{0.60 0.60}
{
\begin{pspicture}(0,-6.2293057)(20.521389,5.0320835)
\rput[bl](2.08,-6.2293053){\Large{$G_i$}}
\rput[bl](0.0,0.31069458){$x_i$}
\rput[bl](14.2,0.6506946){$x_1$}
\rput[bl](15.86,-0.06930542){$x_2$}
\rput[bl](12.2,-0.12930542){$x_n$}
\rput[bl](14.12,-6.1093054){\Large{$G$}}
\psline[linecolor=black, linewidth=0.08](1.52,2.4306946)(2.72,3.2306945)(2.72,3.2306945)
\psline[linecolor=black, linewidth=0.08](1.52,2.4306946)(2.72,1.6306946)(2.72,1.6306946)
\psline[linecolor=black, linewidth=0.08](2.72,3.2306945)(3.92,3.2306945)(3.92,3.2306945)
\psline[linecolor=black, linewidth=0.08](2.72,1.6306946)(3.92,1.6306946)(3.92,1.6306946)
\psline[linecolor=black, linewidth=0.08](1.52,2.4306946)(3.92,2.4306946)(3.92,2.4306946)
\psline[linecolor=black, linewidth=0.08](1.52,0.03069458)(2.72,0.83069456)(2.72,0.83069456)
\psline[linecolor=black, linewidth=0.08](1.52,0.03069458)(2.72,-0.7693054)(2.72,-0.7693054)
\psline[linecolor=black, linewidth=0.08](2.72,0.83069456)(3.92,0.83069456)(3.92,0.83069456)
\psline[linecolor=black, linewidth=0.08](2.72,-0.7693054)(3.92,-0.7693054)(3.92,-0.7693054)
\psline[linecolor=black, linewidth=0.08](1.52,0.03069458)(3.92,0.03069458)(3.92,0.03069458)
\psline[linecolor=black, linewidth=0.08](1.52,-2.3693054)(2.72,-1.5693054)(2.72,-1.5693054)
\psline[linecolor=black, linewidth=0.08](1.52,-2.3693054)(2.72,-3.1693053)(2.72,-3.1693053)
\psline[linecolor=black, linewidth=0.08](2.72,-1.5693054)(3.92,-1.5693054)(3.92,-1.5693054)
\psline[linecolor=black, linewidth=0.08](2.72,-3.1693053)(3.92,-3.1693053)(3.92,-3.1693053)
\psline[linecolor=black, linewidth=0.08](1.52,-2.3693054)(3.92,-2.3693054)(3.92,-2.3693054)
\psline[linecolor=black, linewidth=0.08](1.52,0.03069458)(0.32,0.03069458)(0.32,0.03069458)
\psline[linecolor=black, linewidth=0.08](1.52,2.4306946)(0.32,0.03069458)(0.32,0.03069458)
\psline[linecolor=black, linewidth=0.08](0.32,0.03069458)(1.52,-2.3693054)(1.52,-2.3693054)
\psdots[linecolor=black, dotsize=0.4](2.72,3.2306945)
\psdots[linecolor=black, dotsize=0.4](2.72,2.4306946)
\psdots[linecolor=black, dotsize=0.4](2.72,1.6306946)
\psdots[linecolor=black, dotsize=0.4](2.72,0.83069456)
\psdots[linecolor=black, dotsize=0.4](2.72,0.03069458)
\psdots[linecolor=black, dotsize=0.4](2.72,-0.7693054)
\psdots[linecolor=black, dotsize=0.4](2.72,-1.5693054)
\psdots[linecolor=black, dotsize=0.4](2.72,-2.3693054)
\psdots[linecolor=black, dotsize=0.4](2.72,-3.1693053)
\psdots[linecolor=black, dotstyle=o, dotsize=0.4, fillcolor=white](3.92,3.2306945)
\psdots[linecolor=black, dotstyle=o, dotsize=0.4, fillcolor=white](3.92,2.4306946)
\psdots[linecolor=black, dotstyle=o, dotsize=0.4, fillcolor=white](3.92,1.6306946)
\psdots[linecolor=black, dotstyle=o, dotsize=0.4, fillcolor=white](3.92,0.83069456)
\psdots[linecolor=black, dotstyle=o, dotsize=0.4, fillcolor=white](3.92,0.03069458)
\psdots[linecolor=black, dotstyle=o, dotsize=0.4, fillcolor=white](3.92,-0.7693054)
\psdots[linecolor=black, dotstyle=o, dotsize=0.4, fillcolor=white](3.92,-1.5693054)
\psdots[linecolor=black, dotstyle=o, dotsize=0.4, fillcolor=white](3.92,-2.3693054)
\psdots[linecolor=black, dotstyle=o, dotsize=0.4, fillcolor=white](3.92,-3.1693053)
\psline[linecolor=black, linewidth=0.08](17.92,2.4306946)(19.12,3.2306945)(19.12,3.2306945)
\psline[linecolor=black, linewidth=0.08](17.92,2.4306946)(19.12,1.6306946)(19.12,1.6306946)
\psline[linecolor=black, linewidth=0.08](19.12,3.2306945)(20.32,3.2306945)(20.32,3.2306945)
\psline[linecolor=black, linewidth=0.08](19.12,1.6306946)(20.32,1.6306946)(20.32,1.6306946)
\psline[linecolor=black, linewidth=0.08](17.92,2.4306946)(20.32,2.4306946)(20.32,2.4306946)
\psline[linecolor=black, linewidth=0.08](17.92,0.03069458)(19.12,0.83069456)(19.12,0.83069456)
\psline[linecolor=black, linewidth=0.08](17.92,0.03069458)(19.12,-0.7693054)(19.12,-0.7693054)
\psline[linecolor=black, linewidth=0.08](19.12,0.83069456)(20.32,0.83069456)(20.32,0.83069456)
\psline[linecolor=black, linewidth=0.08](19.12,-0.7693054)(20.32,-0.7693054)(20.32,-0.7693054)
\psline[linecolor=black, linewidth=0.08](17.92,0.03069458)(20.32,0.03069458)(20.32,0.03069458)
\psline[linecolor=black, linewidth=0.08](17.92,-2.3693054)(19.12,-1.5693054)(19.12,-1.5693054)
\psline[linecolor=black, linewidth=0.08](17.92,-2.3693054)(19.12,-3.1693053)(19.12,-3.1693053)
\psline[linecolor=black, linewidth=0.08](19.12,-1.5693054)(20.32,-1.5693054)(20.32,-1.5693054)
\psline[linecolor=black, linewidth=0.08](19.12,-3.1693053)(20.32,-3.1693053)(20.32,-3.1693053)
\psline[linecolor=black, linewidth=0.08](17.92,-2.3693054)(20.32,-2.3693054)(20.32,-2.3693054)
\psline[linecolor=black, linewidth=0.08](17.92,0.03069458)(16.72,0.03069458)(16.72,0.03069458)
\psline[linecolor=black, linewidth=0.08](17.92,2.4306946)(16.72,0.03069458)(16.72,0.03069458)
\psline[linecolor=black, linewidth=0.08](16.72,0.03069458)(17.92,-2.3693054)(17.92,-2.3693054)
\psdots[linecolor=black, dotsize=0.4](19.12,3.2306945)
\psdots[linecolor=black, dotsize=0.4](19.12,2.4306946)
\psdots[linecolor=black, dotsize=0.4](19.12,1.6306946)
\psdots[linecolor=black, dotsize=0.4](19.12,0.83069456)
\psdots[linecolor=black, dotsize=0.4](19.12,0.03069458)
\psdots[linecolor=black, dotsize=0.4](19.12,-0.7693054)
\psdots[linecolor=black, dotsize=0.4](19.12,-1.5693054)
\psdots[linecolor=black, dotsize=0.4](19.12,-2.3693054)
\psdots[linecolor=black, dotsize=0.4](19.12,-3.1693053)
\psdots[linecolor=black, dotstyle=o, dotsize=0.4, fillcolor=white](20.32,3.2306945)
\psdots[linecolor=black, dotstyle=o, dotsize=0.4, fillcolor=white](20.32,2.4306946)
\psdots[linecolor=black, dotstyle=o, dotsize=0.4, fillcolor=white](20.32,1.6306946)
\psdots[linecolor=black, dotstyle=o, dotsize=0.4, fillcolor=white](20.32,0.83069456)
\psdots[linecolor=black, dotstyle=o, dotsize=0.4, fillcolor=white](20.32,0.03069458)
\psdots[linecolor=black, dotstyle=o, dotsize=0.4, fillcolor=white](20.32,-0.7693054)
\psdots[linecolor=black, dotstyle=o, dotsize=0.4, fillcolor=white](20.32,-1.5693054)
\psdots[linecolor=black, dotstyle=o, dotsize=0.4, fillcolor=white](20.32,-2.3693054)
\psdots[linecolor=black, dotstyle=o, dotsize=0.4, fillcolor=white](20.32,-3.1693053)
\psline[linecolor=black, linewidth=0.08](16.72,0.03069458)(14.32,1.2306945)(14.32,1.2306945)
\psline[linecolor=black, linewidth=0.08](16.72,0.03069458)(14.32,-1.1693054)(14.32,-1.1693054)
\psline[linecolor=black, linewidth=0.08](14.32,1.2306945)(11.92,0.03069458)(11.92,0.03069458)
\psline[linecolor=black, linewidth=0.08, linestyle=dotted, dotsep=0.10583334cm](14.32,-1.1693054)(11.92,0.03069458)(11.92,0.03069458)
\psline[linecolor=black, linewidth=0.08](11.92,0.03069458)(10.72,0.03069458)(10.72,0.03069458)
\psline[linecolor=black, linewidth=0.08](10.72,2.4306946)(11.92,0.03069458)(11.92,0.03069458)
\psline[linecolor=black, linewidth=0.08](11.92,0.03069458)(10.72,-2.3693054)
\psline[linecolor=black, linewidth=0.08](10.72,2.4306946)(9.52,3.2306945)(9.52,3.2306945)
\psline[linecolor=black, linewidth=0.08](10.72,2.4306946)(8.32,2.4306946)(8.32,2.4306946)
\psline[linecolor=black, linewidth=0.08](10.72,2.4306946)(9.52,1.6306946)(9.52,1.6306946)
\psline[linecolor=black, linewidth=0.08](8.72,1.6306946)(8.72,1.6306946)
\psline[linecolor=black, linewidth=0.08](8.32,1.6306946)(9.52,1.6306946)
\psline[linecolor=black, linewidth=0.08](8.32,3.2306945)(9.52,3.2306945)
\psline[linecolor=black, linewidth=0.08](10.72,0.03069458)(9.52,0.83069456)(9.52,0.83069456)
\psline[linecolor=black, linewidth=0.08](10.72,0.03069458)(8.32,0.03069458)(8.32,0.03069458)
\psline[linecolor=black, linewidth=0.08](10.72,0.03069458)(9.52,-0.7693054)(9.52,-0.7693054)
\psline[linecolor=black, linewidth=0.08](8.32,-0.7693054)(9.52,-0.7693054)
\psline[linecolor=black, linewidth=0.08](8.32,0.83069456)(9.52,0.83069456)
\psline[linecolor=black, linewidth=0.08](10.72,-2.3693054)(9.52,-1.5693054)(9.52,-1.5693054)
\psline[linecolor=black, linewidth=0.08](10.72,-2.3693054)(8.32,-2.3693054)(8.32,-2.3693054)
\psline[linecolor=black, linewidth=0.08](10.72,-2.3693054)(9.52,-3.1693053)(9.52,-3.1693053)
\psline[linecolor=black, linewidth=0.08](8.32,-3.1693053)(9.52,-3.1693053)
\psline[linecolor=black, linewidth=0.08](8.32,-1.5693054)(9.52,-1.5693054)
\psdots[linecolor=black, dotstyle=o, dotsize=0.4, fillcolor=white](0.32,0.03069458)
\psdots[linecolor=black, dotsize=0.4](1.52,2.4306946)
\psdots[linecolor=black, dotsize=0.4](1.52,0.03069458)
\psdots[linecolor=black, dotsize=0.4](1.52,-2.3693054)
\psdots[linecolor=black, dotsize=0.4](16.72,0.03069458)
\psdots[linecolor=black, dotsize=0.4](11.92,0.03069458)
\psdots[linecolor=black, dotsize=0.4](9.52,3.2306945)
\psdots[linecolor=black, dotsize=0.4](9.52,2.4306946)
\psdots[linecolor=black, dotsize=0.4](9.52,1.6306946)
\psdots[linecolor=black, dotsize=0.4](9.52,0.83069456)
\psdots[linecolor=black, dotsize=0.4](9.52,0.03069458)
\psdots[linecolor=black, dotsize=0.4](9.52,-0.7693054)
\psdots[linecolor=black, dotsize=0.4](9.52,-1.5693054)
\psdots[linecolor=black, dotsize=0.4](9.52,-2.3693054)
\psdots[linecolor=black, dotsize=0.4](9.52,-3.1693053)
\psdots[linecolor=black, dotstyle=o, dotsize=0.4, fillcolor=white](8.32,3.2306945)
\psdots[linecolor=black, dotstyle=o, dotsize=0.4, fillcolor=white](8.32,2.4306946)
\psdots[linecolor=black, dotstyle=o, dotsize=0.4, fillcolor=white](8.32,1.6306946)
\psdots[linecolor=black, dotstyle=o, dotsize=0.4, fillcolor=white](8.32,0.83069456)
\psdots[linecolor=black, dotstyle=o, dotsize=0.4, fillcolor=white](8.32,0.03069458)
\psdots[linecolor=black, dotstyle=o, dotsize=0.4, fillcolor=white](8.32,-0.7693054)
\psdots[linecolor=black, dotstyle=o, dotsize=0.4, fillcolor=white](8.32,-1.5693054)
\psdots[linecolor=black, dotstyle=o, dotsize=0.4, fillcolor=white](8.32,-2.3693054)
\psdots[linecolor=black, dotstyle=o, dotsize=0.4, fillcolor=white](8.32,-3.1693053)
\psline[linecolor=black, linewidth=0.08](14.32,1.2306945)(14.32,2.4306946)(14.32,2.4306946)
\psline[linecolor=black, linewidth=0.08](16.72,2.4306946)(14.32,1.2306945)(14.32,1.2306945)
\psline[linecolor=black, linewidth=0.08](11.92,2.4306946)(14.32,1.2306945)(14.32,1.2306945)
\psline[linecolor=black, linewidth=0.08](16.72,2.4306946)(17.52,3.6306946)(17.52,4.8306947)
\psline[linecolor=black, linewidth=0.08](16.72,4.8306947)(16.72,2.4306946)(16.72,2.4306946)
\psline[linecolor=black, linewidth=0.08](15.92,4.8306947)(15.92,3.6306946)(16.72,2.4306946)
\psline[linecolor=black, linewidth=0.08](14.32,2.4306946)(15.12,3.6306946)(15.12,4.8306947)
\psline[linecolor=black, linewidth=0.08](14.32,4.8306947)(14.32,2.4306946)(14.32,2.4306946)
\psline[linecolor=black, linewidth=0.08](13.52,4.8306947)(13.52,3.6306946)(14.32,2.4306946)
\psline[linecolor=black, linewidth=0.08](11.92,2.4306946)(12.72,3.6306946)(12.72,4.8306947)
\psline[linecolor=black, linewidth=0.08](11.92,4.8306947)(11.92,2.4306946)(11.92,2.4306946)
\psline[linecolor=black, linewidth=0.08](11.12,4.8306947)(11.12,3.6306946)(11.92,2.4306946)
\psline[linecolor=black, linewidth=0.08](14.32,-1.1693054)(14.32,-2.3693054)(14.32,-2.3693054)
\psline[linecolor=black, linewidth=0.08](14.32,-1.1693054)(16.72,-2.3693054)
\psline[linecolor=black, linewidth=0.08](14.32,-1.1693054)(11.92,-2.3693054)
\psline[linecolor=black, linewidth=0.08](11.92,-2.3693054)(11.12,-3.5693054)(11.12,-4.769305)
\psline[linecolor=black, linewidth=0.08](11.92,-2.3693054)(11.92,-3.5693054)(11.92,-4.769305)(11.92,-4.769305)
\psline[linecolor=black, linewidth=0.08](11.92,-2.3693054)(12.72,-3.5693054)(12.72,-4.769305)
\psline[linecolor=black, linewidth=0.08](14.32,-2.3693054)(13.52,-3.5693054)(13.52,-4.769305)
\psline[linecolor=black, linewidth=0.08](14.32,-2.3693054)(14.32,-3.5693054)(14.32,-4.769305)(14.32,-4.769305)
\psline[linecolor=black, linewidth=0.08](14.32,-2.3693054)(15.12,-3.5693054)(15.12,-4.769305)
\psline[linecolor=black, linewidth=0.08](16.72,-2.3693054)(15.92,-3.5693054)(15.92,-4.769305)
\psline[linecolor=black, linewidth=0.08](16.72,-2.3693054)(16.72,-3.5693054)(16.72,-4.769305)(16.72,-4.769305)
\psline[linecolor=black, linewidth=0.08](16.72,-2.3693054)(17.52,-3.5693054)(17.52,-4.769305)
\psdots[linecolor=black, dotstyle=o, dotsize=0.4, fillcolor=white](11.12,4.8306947)
\psdots[linecolor=black, dotstyle=o, dotsize=0.4, fillcolor=white](11.92,4.8306947)
\psdots[linecolor=black, dotstyle=o, dotsize=0.4, fillcolor=white](12.72,4.8306947)
\psdots[linecolor=black, dotstyle=o, dotsize=0.4, fillcolor=white](13.52,4.8306947)
\psdots[linecolor=black, dotstyle=o, dotsize=0.4, fillcolor=white](14.32,4.8306947)
\psdots[linecolor=black, dotstyle=o, dotsize=0.4, fillcolor=white](15.12,4.8306947)
\psdots[linecolor=black, dotstyle=o, dotsize=0.4, fillcolor=white](15.92,4.8306947)
\psdots[linecolor=black, dotstyle=o, dotsize=0.4, fillcolor=white](16.72,4.8306947)
\psdots[linecolor=black, dotstyle=o, dotsize=0.4, fillcolor=white](17.52,4.8306947)
\psdots[linecolor=black, dotstyle=o, dotsize=0.4, fillcolor=white](10.72,2.4306946)
\psdots[linecolor=black, dotstyle=o, dotsize=0.4, fillcolor=white](11.92,2.4306946)
\psdots[linecolor=black, dotstyle=o, dotsize=0.4, fillcolor=white](14.32,2.4306946)
\psdots[linecolor=black, dotstyle=o, dotsize=0.4, fillcolor=white](16.72,2.4306946)
\psdots[linecolor=black, dotstyle=o, dotsize=0.4, fillcolor=white](17.92,2.4306946)
\psdots[linecolor=black, dotstyle=o, dotsize=0.4, fillcolor=white](17.92,0.03069458)
\psdots[linecolor=black, dotstyle=o, dotsize=0.4, fillcolor=white](17.92,-2.3693054)
\psdots[linecolor=black, dotstyle=o, dotsize=0.4, fillcolor=white](16.72,-2.3693054)
\psdots[linecolor=black, dotstyle=o, dotsize=0.4, fillcolor=white](14.32,-2.3693054)
\psdots[linecolor=black, dotstyle=o, dotsize=0.4, fillcolor=white](11.92,-2.3693054)
\psdots[linecolor=black, dotstyle=o, dotsize=0.4, fillcolor=white](10.72,-2.3693054)
\psdots[linecolor=black, dotstyle=o, dotsize=0.4, fillcolor=white](10.72,0.03069458)
\psdots[linecolor=black, dotstyle=o, dotsize=0.4, fillcolor=white](11.12,-4.769305)
\psdots[linecolor=black, dotstyle=o, dotsize=0.4, fillcolor=white](11.92,-4.769305)
\psdots[linecolor=black, dotstyle=o, dotsize=0.4, fillcolor=white](12.72,-4.769305)
\psdots[linecolor=black, dotstyle=o, dotsize=0.4, fillcolor=white](13.52,-4.769305)
\psdots[linecolor=black, dotstyle=o, dotsize=0.4, fillcolor=white](14.32,-4.769305)
\psdots[linecolor=black, dotstyle=o, dotsize=0.4, fillcolor=white](15.12,-4.769305)
\psdots[linecolor=black, dotstyle=o, dotsize=0.4, fillcolor=white](15.92,-4.769305)
\psdots[linecolor=black, dotstyle=o, dotsize=0.4, fillcolor=white](16.72,-4.769305)
\psdots[linecolor=black, dotstyle=o, dotsize=0.4, fillcolor=white](17.52,-4.769305)
\psdots[linecolor=black, dotsize=0.4](11.12,3.6306946)
\psdots[linecolor=black, dotsize=0.4](11.92,3.6306946)
\psdots[linecolor=black, dotsize=0.4](12.72,3.6306946)
\psdots[linecolor=black, dotsize=0.4](13.52,3.6306946)
\psdots[linecolor=black, dotsize=0.4](14.32,3.6306946)
\psdots[linecolor=black, dotsize=0.4](15.12,3.6306946)
\psdots[linecolor=black, dotsize=0.4](15.92,3.6306946)
\psdots[linecolor=black, dotsize=0.4](16.72,3.6306946)
\psdots[linecolor=black, dotsize=0.4](17.52,3.6306946)
\psdots[linecolor=black, dotsize=0.4](17.52,-3.5693054)
\psdots[linecolor=black, dotsize=0.4](16.72,-3.5693054)
\psdots[linecolor=black, dotsize=0.4](15.92,-3.5693054)
\psdots[linecolor=black, dotsize=0.4](15.12,-3.5693054)
\psdots[linecolor=black, dotsize=0.4](14.32,-3.5693054)
\psdots[linecolor=black, dotsize=0.4](13.52,-3.5693054)
\psdots[linecolor=black, dotsize=0.4](12.72,-3.5693054)
\psdots[linecolor=black, dotsize=0.4](11.92,-3.5693054)
\psdots[linecolor=black, dotsize=0.4](11.12,-3.5693054)
\psdots[linecolor=black, dotsize=0.4](14.32,-1.1693054)
\psdots[linecolor=black, dotsize=0.4](14.32,1.2306945)
\rput[bl](14.22,-0.8293054){$x_3$}
\end{pspicture}
}
		\end{center}
		\caption{Graphs $G_i$, for $i=1,2,\ldots,n$, and their circuit $G$, respectively.} \label{fig:circuit-lower}
	\end{figure}

\begin{remark}
{\rm We remark that the bounds in Theorem \ref{thm:circuit} are tight. For the lower bound, it suffices to consider graphs $G_i$, for $i \in [n]$, and their circuit with respect to the vertices $\{x_i\}_{i=1}^n$ as illustrated in Figure \ref{fig:circuit-lower}. The set of black vertices in each graph is a $\gst$-set, and we are done. This idea can be generalized, and therefore we have an infinite family of graphs such that the equality of the lower bound holds.
 For the upper bound, consider Figure \ref{fig:circuit-upper}.
For $i \in [n]$, the set $S_i=\{y_{i_1},y_{i_2},\ldots,y_{i_i}\}$ is a $\gst$-set of $G_i$, and for odd $n$, 
$$S_o=\bigcup\limits_{i=1}^{n} S_i\cup\{x_3,x_5,\ldots,x_{n-2},x_n\},$$
and for even $n$,
$$S_e=\bigcup\limits_{i=1}^{n} S_i\cup\{x_3,x_5,\ldots,x_{n-1},x_n\},$$
are $\gst$-sets of $G$.
  This idea can be generalized, so we have an infinite family of graphs such that the equality of the upper bound holds.} 
	\end{remark}

	\begin{figure}[!h]
		\begin{center}
			\psscalebox{0.60 0.60}
{
\begin{pspicture}(0,-6.629306)(19.21,4.6320834)
\rput[bl](2.88,-6.6293054){\Large{$G_i$}}
\rput[bl](0.0,0.31069458){$x_i$}
\rput[bl](12.92,-6.5093055){\Large{$G$}}
\psline[linecolor=black, linewidth=0.08](2.32,3.6306946)(3.52,4.4306946)(3.52,4.4306946)
\psline[linecolor=black, linewidth=0.08](2.32,0.43069458)(3.52,1.2306945)(3.52,1.2306945)
\psline[linecolor=black, linewidth=0.08](2.32,-3.9693055)(3.52,-3.1693053)(3.52,-3.1693053)
\psdots[linecolor=black, dotstyle=o, dotsize=0.4, fillcolor=white](3.52,4.4306946)
\psdots[linecolor=black, dotstyle=o, dotsize=0.4, fillcolor=white](3.52,1.2306945)
\psdots[linecolor=black, dotstyle=o, dotsize=0.4, fillcolor=white](3.52,-3.1693053)
\psline[linecolor=black, linewidth=0.08](2.32,3.6306946)(3.52,3.6306946)(3.52,3.6306946)
\psline[linecolor=black, linewidth=0.08](2.32,3.6306946)(3.52,2.0306945)(3.52,2.0306945)
\psdots[linecolor=black, dotstyle=o, dotsize=0.4, fillcolor=white](3.52,3.6306946)
\psdots[linecolor=black, dotstyle=o, dotsize=0.4, fillcolor=white](3.52,2.0306945)
\psline[linecolor=black, linewidth=0.08](2.32,0.43069458)(3.52,0.43069458)(3.52,0.43069458)
\psline[linecolor=black, linewidth=0.08](2.32,0.43069458)(3.52,-1.1693054)(3.52,-1.1693054)
\psline[linecolor=black, linewidth=0.08](2.32,-3.9693055)(3.52,-3.9693055)(3.52,-3.9693055)
\psline[linecolor=black, linewidth=0.08](2.32,-3.9693055)(3.52,-5.5693054)(3.52,-5.5693054)
\psdots[linecolor=black, dotstyle=o, dotsize=0.4, fillcolor=white](3.52,0.43069458)
\psdots[linecolor=black, dotstyle=o, dotsize=0.4, fillcolor=white](3.52,-1.1693054)
\psdots[linecolor=black, dotstyle=o, dotsize=0.4, fillcolor=white](3.52,-3.9693055)
\psdots[linecolor=black, dotstyle=o, dotsize=0.4, fillcolor=white](3.52,-5.5693054)
\psdots[linecolor=black, dotsize=0.1](3.52,3.2306945)
\psdots[linecolor=black, dotsize=0.1](3.52,2.8306947)
\psdots[linecolor=black, dotsize=0.1](3.52,2.4306946)
\psdots[linecolor=black, dotsize=0.1](3.52,0.03069458)
\psdots[linecolor=black, dotsize=0.1](3.52,-0.36930543)
\psdots[linecolor=black, dotsize=0.1](3.52,-0.7693054)
\psdots[linecolor=black, dotsize=0.1](3.52,-4.3693056)
\psdots[linecolor=black, dotsize=0.1](3.52,-4.769305)
\psdots[linecolor=black, dotsize=0.1](3.52,-5.1693053)
\psdots[linecolor=black, dotsize=0.1](2.32,-1.5693054)
\psdots[linecolor=black, dotsize=0.1](2.32,-1.9693054)
\psdots[linecolor=black, dotsize=0.1](2.32,-2.3693054)
\psline[linecolor=black, linewidth=0.08](2.32,-3.9693055)(0.72,0.43069458)(0.72,0.43069458)
\psline[linecolor=black, linewidth=0.08](2.32,3.6306946)(0.72,0.43069458)(0.72,0.43069458)
\psline[linecolor=black, linewidth=0.08](0.72,0.43069458)(2.32,0.43069458)(2.32,0.43069458)
\psdots[linecolor=black, dotstyle=o, dotsize=0.4, fillcolor=white](0.72,0.43069458)
\rput[bl](2.14,2.9706945){$y_{i_1}$}
\rput[bl](2.12,0.7506946){$y_{i_2}$}
\rput[bl](2.12,-4.6693053){$y_{i_i}$}
\rput[bl](3.9,4.2706947){$z_{1_1}$}
\rput[bl](3.86,3.4306946){$z_{1_2}$}
\rput[bl](3.84,1.7906946){$z_{1_i}$}
\rput[bl](3.84,1.1106945){$z_{2_1}$}
\rput[bl](3.84,0.23069458){$z_{2_2}$}
\rput[bl](3.82,-1.3893055){$z_{2_i}$}
\rput[bl](3.88,-3.3493054){$z_{i_1}$}
\rput[bl](3.84,-4.2293053){$z_{i_2}$}
\rput[bl](3.8,-5.809305){$z_{i_i}$}
\psline[linecolor=black, linewidth=0.08](17.12,3.6306946)(18.32,4.4306946)(18.32,4.4306946)
\psline[linecolor=black, linewidth=0.08](17.12,0.43069458)(18.32,1.2306945)(18.32,1.2306945)
\psline[linecolor=black, linewidth=0.08](17.12,-3.9693055)(18.32,-3.1693053)(18.32,-3.1693053)
\psdots[linecolor=black, dotstyle=o, dotsize=0.4, fillcolor=white](18.32,4.4306946)
\psdots[linecolor=black, dotstyle=o, dotsize=0.4, fillcolor=white](18.32,1.2306945)
\psdots[linecolor=black, dotstyle=o, dotsize=0.4, fillcolor=white](18.32,-3.1693053)
\psline[linecolor=black, linewidth=0.08](17.12,3.6306946)(18.32,3.6306946)(18.32,3.6306946)
\psline[linecolor=black, linewidth=0.08](17.12,3.6306946)(18.32,2.0306945)(18.32,2.0306945)
\psdots[linecolor=black, dotstyle=o, dotsize=0.4, fillcolor=white](18.32,3.6306946)
\psdots[linecolor=black, dotstyle=o, dotsize=0.4, fillcolor=white](18.32,2.0306945)
\psline[linecolor=black, linewidth=0.08](17.12,0.43069458)(18.32,0.43069458)(18.32,0.43069458)
\psline[linecolor=black, linewidth=0.08](17.12,0.43069458)(18.32,-1.1693054)(18.32,-1.1693054)
\psline[linecolor=black, linewidth=0.08](17.12,-3.9693055)(18.32,-3.9693055)(18.32,-3.9693055)
\psline[linecolor=black, linewidth=0.08](17.12,-3.9693055)(18.32,-5.5693054)(18.32,-5.5693054)
\psdots[linecolor=black, dotstyle=o, dotsize=0.4, fillcolor=white](18.32,0.43069458)
\psdots[linecolor=black, dotstyle=o, dotsize=0.4, fillcolor=white](18.32,-1.1693054)
\psdots[linecolor=black, dotstyle=o, dotsize=0.4, fillcolor=white](18.32,-3.9693055)
\psdots[linecolor=black, dotstyle=o, dotsize=0.4, fillcolor=white](18.32,-5.5693054)
\psdots[linecolor=black, dotsize=0.1](18.32,3.2306945)
\psdots[linecolor=black, dotsize=0.1](18.32,2.8306947)
\psdots[linecolor=black, dotsize=0.1](18.32,2.4306946)
\psdots[linecolor=black, dotsize=0.1](18.32,0.03069458)
\psdots[linecolor=black, dotsize=0.1](18.32,-0.36930543)
\psdots[linecolor=black, dotsize=0.1](18.32,-0.7693054)
\psdots[linecolor=black, dotsize=0.1](18.32,-4.3693056)
\psdots[linecolor=black, dotsize=0.1](18.32,-4.769305)
\psdots[linecolor=black, dotsize=0.1](18.32,-5.1693053)
\psdots[linecolor=black, dotsize=0.1](17.12,-1.5693054)
\psdots[linecolor=black, dotsize=0.1](17.12,-1.9693054)
\psdots[linecolor=black, dotsize=0.1](17.12,-2.3693054)
\psline[linecolor=black, linewidth=0.08](17.12,-3.9693055)(15.52,0.43069458)(15.52,0.43069458)
\psline[linecolor=black, linewidth=0.08](17.12,3.6306946)(15.52,0.43069458)(15.52,0.43069458)
\psline[linecolor=black, linewidth=0.08](15.52,0.43069458)(17.12,0.43069458)(17.12,0.43069458)
\rput[bl](18.7,4.2706947){$z_{1_1}$}
\rput[bl](18.66,3.4306946){$z_{1_2}$}
\rput[bl](18.64,1.1106945){$z_{2_1}$}
\rput[bl](18.64,0.23069458){$z_{2_2}$}
\rput[bl](15.74,0.03069458){$x_n$}
\rput[bl](16.92,-4.6693053){$y_{n_n}$}
\rput[bl](16.92,0.7506946){$y_{n_2}$}
\rput[bl](16.94,2.9706945){$y_{n_1}$}
\rput[bl](18.64,1.7906946){$z_{1_n}$}
\rput[bl](18.62,-1.3893055){$z_{2_n}$}
\rput[bl](18.68,-3.3493054){$z_{n_1}$}
\rput[bl](18.64,-4.2293053){$z_{n_2}$}
\rput[bl](18.6,-5.809305){$z_{n_n}$}
\psline[linecolor=black, linewidth=0.08](15.52,0.43069458)(12.72,2.8306947)(12.72,2.8306947)
\psline[linecolor=black, linewidth=0.08](9.92,0.43069458)(12.72,2.8306947)
\psline[linecolor=black, linewidth=0.08](9.92,0.43069458)(12.72,-1.9693054)(12.72,-1.9693054)
\psdots[linecolor=black, dotstyle=o, dotsize=0.4, fillcolor=white](2.32,3.6306946)
\psdots[linecolor=black, dotstyle=o, dotsize=0.4, fillcolor=white](2.32,0.43069458)
\psdots[linecolor=black, dotstyle=o, dotsize=0.4, fillcolor=white](2.32,-3.9693055)
\psdots[linecolor=black, dotstyle=o, dotsize=0.4, fillcolor=white](17.12,3.6306946)
\psdots[linecolor=black, dotstyle=o, dotsize=0.4, fillcolor=white](17.12,0.43069458)
\psdots[linecolor=black, dotstyle=o, dotsize=0.4, fillcolor=white](17.12,-3.9693055)
\psline[linecolor=black, linewidth=0.08, linestyle=dotted, dotsep=0.10583334cm](12.72,-1.9693054)(15.52,0.43069458)
\psline[linecolor=black, linewidth=0.08](12.72,2.8306947)(12.72,3.6306946)(12.72,4.4306946)
\psline[linecolor=black, linewidth=0.08](9.92,0.43069458)(8.32,2.0306945)
\psline[linecolor=black, linewidth=0.08](9.92,0.43069458)(8.32,-1.1693054)(8.32,-1.1693054)
\psline[linecolor=black, linewidth=0.08](12.72,-1.9693054)(12.72,-3.9693055)(12.72,-3.9693055)
\psline[linecolor=black, linewidth=0.08](12.72,-3.9693055)(11.92,-5.1693053)(11.92,-5.1693053)
\psline[linecolor=black, linewidth=0.08](12.72,-3.9693055)(12.72,-5.1693053)(12.72,-5.1693053)
\psline[linecolor=black, linewidth=0.08](12.72,-3.9693055)(13.52,-5.1693053)(13.52,-5.1693053)
\psline[linecolor=black, linewidth=0.08](14.32,-5.1693053)(15.12,-3.9693055)(15.12,-3.9693055)
\psline[linecolor=black, linewidth=0.08](15.12,-3.9693055)(15.12,-5.1693053)(15.12,-5.1693053)
\psline[linecolor=black, linewidth=0.08](15.92,-5.1693053)(15.12,-3.9693055)
\psline[linecolor=black, linewidth=0.08](11.12,-5.1693053)(10.32,-3.9693055)
\psline[linecolor=black, linewidth=0.08](9.52,-5.1693053)(10.32,-3.9693055)
\psline[linecolor=black, linewidth=0.08](10.32,-5.1693053)(10.32,-3.9693055)
\psline[linecolor=black, linewidth=0.08](10.32,-3.9693055)(12.72,-1.9693054)(12.72,-1.9693054)
\psline[linecolor=black, linewidth=0.08](12.72,-1.9693054)(15.12,-3.9693055)(15.12,-3.9693055)
\psline[linecolor=black, linewidth=0.08](8.32,2.0306945)(7.12,0.83069456)(7.12,0.83069456)
\psline[linecolor=black, linewidth=0.08](7.12,3.2306945)(8.32,2.0306945)
\psline[linecolor=black, linewidth=0.08](7.12,0.03069458)(8.32,-1.1693054)(8.32,-1.1693054)
\psline[linecolor=black, linewidth=0.08](8.32,-1.1693054)(7.12,-2.3693054)
\psdots[linecolor=black, dotstyle=o, dotsize=0.4, fillcolor=white](12.72,4.4306946)
\psdots[linecolor=black, dotstyle=o, dotsize=0.4, fillcolor=white](7.12,3.2306945)
\psdots[linecolor=black, dotstyle=o, dotsize=0.4, fillcolor=white](7.12,0.83069456)
\psdots[linecolor=black, dotstyle=o, dotsize=0.4, fillcolor=white](7.12,0.03069458)
\psdots[linecolor=black, dotstyle=o, dotsize=0.4, fillcolor=white](7.12,-2.3693054)
\psdots[linecolor=black, dotstyle=o, dotsize=0.4, fillcolor=white](9.52,-5.1693053)
\psdots[linecolor=black, dotstyle=o, dotsize=0.4, fillcolor=white](10.32,-5.1693053)
\psdots[linecolor=black, dotstyle=o, dotsize=0.4, fillcolor=white](11.12,-5.1693053)
\psdots[linecolor=black, dotstyle=o, dotsize=0.4, fillcolor=white](11.92,-5.1693053)
\psdots[linecolor=black, dotstyle=o, dotsize=0.4, fillcolor=white](12.72,-5.1693053)
\psdots[linecolor=black, dotstyle=o, dotsize=0.4, fillcolor=white](13.52,-5.1693053)
\psdots[linecolor=black, dotstyle=o, dotsize=0.4, fillcolor=white](14.32,-5.1693053)
\psdots[linecolor=black, dotstyle=o, dotsize=0.4, fillcolor=white](15.12,-5.1693053)
\psdots[linecolor=black, dotstyle=o, dotsize=0.4, fillcolor=white](15.92,-5.1693053)
\psdots[linecolor=black, dotstyle=o, dotsize=0.4, fillcolor=white](8.32,2.0306945)
\psdots[linecolor=black, dotstyle=o, dotsize=0.4, fillcolor=white](8.32,-1.1693054)
\psdots[linecolor=black, dotstyle=o, dotsize=0.4, fillcolor=white](9.92,0.43069458)
\psdots[linecolor=black, dotstyle=o, dotsize=0.4, fillcolor=white](12.72,2.8306947)
\psdots[linecolor=black, dotstyle=o, dotsize=0.4, fillcolor=white](15.52,0.43069458)
\psdots[linecolor=black, dotstyle=o, dotsize=0.4, fillcolor=white](12.72,-1.9693054)
\psdots[linecolor=black, dotstyle=o, dotsize=0.4, fillcolor=white](10.32,-3.9693055)
\psdots[linecolor=black, dotstyle=o, dotsize=0.4, fillcolor=white](12.72,-3.9693055)
\psdots[linecolor=black, dotstyle=o, dotsize=0.4, fillcolor=white](15.12,-3.9693055)
\rput[bl](12.56,2.2106946){$x_1$}
\rput[bl](10.32,0.31069458){$x_2$}
\rput[bl](12.58,-1.6693054){$x_3$}
\psdots[linecolor=black, dotstyle=o, dotsize=0.4, fillcolor=white](12.72,3.6306946)
\end{pspicture}
}
		\end{center}
		\caption{Graphs $G_i$, for $i=1,2,\ldots,n$, and their circuit $G$, respectively.} \label{fig:circuit-upper}
	\end{figure}
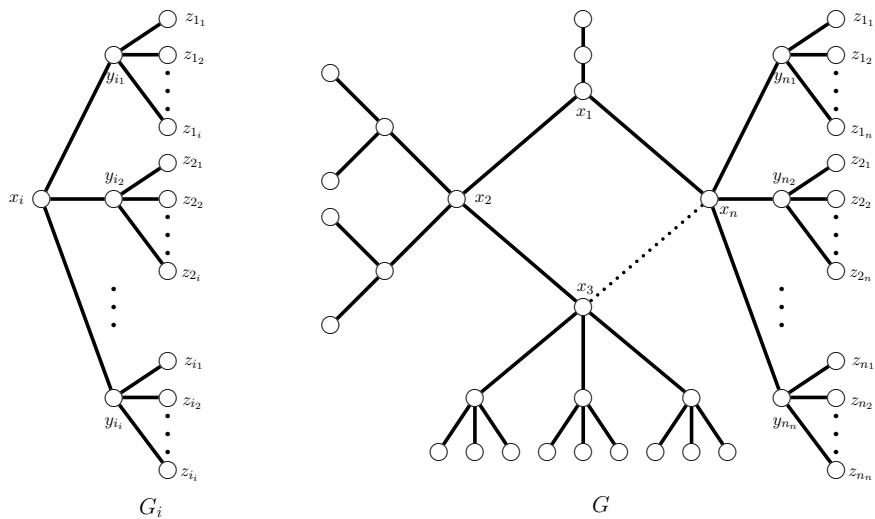

%%%%%%%%%%%%%%%%%%%%%%%%%%%%%%%%%%%%%%%%%%%%%%%%%%%%%%%%%%%%%%%%%%%%%%%%%%%%%%%%%	

\end{document}